\documentclass[10pt,reqno]{amsart}
\usepackage{latexsym,amssymb,amsmath}
\usepackage{amsmath,amsfonts}
\usepackage[hidelinks]{hyperref}
\usepackage{epsfig}
\usepackage{graphicx}
\usepackage{enumerate}

\usepackage{xcolor}
\usepackage{cancel}
\usepackage{scalerel,stackengine}
\stackMath
\newcommand\reallywidehat[1]{%
\savestack{\tmpbox}{\stretchto{%
  \scaleto{%
    \scalerel*[\widthof{\ensuremath{#1}}]{\kern-.6pt\bigwedge\kern-.6pt}%
    {\rule[-\textheight/2]{1ex}{\textheight}}
  }{\textheight}%
}{0.5ex}}%
\stackon[1pt]{#1}{\tmpbox}%
}

\theoremstyle{plain}
\newtheorem{theorem}{Theorem}[section]

\newtheorem{lemma}[theorem]{Lemma}

\newtheorem{proposition}[theorem]{Proposition}

\theoremstyle{remark}
\newtheorem{Remark}[theorem]{Remark}

\allowdisplaybreaks

\newcommand{\Hmm}[1]{\leavevmode{\marginpar{\tiny%
$\hbox to 0mm{\hspace*{-0.5mm}$\leftarrow$\hss}%
\vcenter{\vrule depth 0.1mm height 0.1mm width \the\marginparwidth}%
\hbox to 0mm{\hss$\rightarrow$\hspace*{-0.5mm}}$\\\relax\raggedright #1}}}

\begin{document}
 \title[Long-time behavior of SNLS with additive noise]{Scattering for Stochastic Nonlinear Schrödinger Equations with additive noise}

\author{Engin Ba\c{s}ako\u{g}lu}
\author{Faruk Temur}
\author{Bar{\i}\c{s} Ye\c{s}ilo\u{g}lu}
\author{O\u{g}uz Y{\i}lmaz}

\address{Institute of Mathematical Sciences, ShanghaiTech University, Shanghai, 201210, China}
\email{ebasakoglu@shanghaitech.edu.cn}
\address{Department of Mathematics,
Izmir Institute of Technology, 
Urla 35430, Izmir, Turkey}
\email{faruktemur@iyte.edu.tr}
 \address{Department of Mathematics,
Bo\u gazi\c ci University, 
Bebek 34342, Istanbul, Turkey}
\email{baris.yesiloglu@bogazici.edu.tr}
\address{Department of Mathematics,
Bo\u gazi\c ci University, 
Bebek 34342, Istanbul, Turkey}
\email{oguz.yilmaz@bogazici.edu.tr}


\begin{abstract}
    We study the scattering for the energy-subcritical stochastic nonlinear Schrödinger equation (SNLS) with additive noise. In particular, we examine the long-time behavior of solutions associated with the noise $\phi(x)g(t,\omega)dB(t,\omega)$ formed by a Schwartz function $\phi$, and an adapted process $g(t,\omega)$ satisfying certain decay. Essentially, the aim of the current paper is to prove almost sure scattering in the spaces $L^2$ and the pseudo-conformal space $\Sigma$ for an initial data in $\Sigma$; also in $H^1$ for an initial data in $H^1$.
\end{abstract}
\maketitle

\section{Introduction}

In this paper, we consider the stochastic nonlinear Schrödinger equation (SNLS) with additive noise 
    \begin{equation}\label{eq:additive_SNLS}
    \begin{cases}
    &idu - \Delta u dt + |u|^{2\sigma}udt = dW,\\
    &u|_{t=0}=u_{0},
    \end{cases}
    \quad (t,x)\in[0,\infty)\times\mathbb{R}^{n},
\end{equation}
where $n\geq 1$, $\sigma\in\mathbb{R}_+$, and $u=u(t,x,\omega)$ is a complex function-valued stochastic process. The noise term in \eqref{eq:additive_SNLS} is defined by $dW=\phi(x) g(t,\omega)dB(t,\omega)$, where $\phi$ is a Schwartz-in-space function, $g$ is a real-valued predictable process with $g\in L^{2}_{t}([0,\infty))$ $\mathbb{P}$-a.s., and $B(t,\omega)$ is a real-valued Brownian motion on a probability space $(\Omega,\{\mathcal{F}_{t}:t\geq 0\},\mathbb{P})$ with right-continuous filtration. Our essential focus will be the long-time behavior of the global solutions of \eqref{eq:additive_SNLS} with initial data coming from pseudo-conformal space $\Sigma$ or $H^{1}$. Before referring to some related literature, we shall state our main results next. The first of these corresponds to the global well-posedness of the IVP \eqref{eq:additive_SNLS} in $H^{1}$ and $\Sigma$:
\begin{theorem}\label{thm:result_well-posedness}
    Let $0<2\sigma<\frac{4}{n-2}$ if $n\geq 3$, $0<2\sigma<\infty$ if $n=1,2$. Assume that $g\in L^{\infty}_{\omega,t}(\Omega\times [0,T])$ for any $0<T<\infty$. Then, for $u_{0}\in H^{1}_{x}(\mathbb{R}^{n})$, there exists a unique solution $u$ with $u|_{t=0}=u_{0}$ to \eqref{eq:additive_SNLS} such that for any Strichartz admissible pair $(p,q)$ we have
    \begin{equation*}
        u\in C_{t}H^{1}_{x}([0,\infty)\times\mathbb{R}^{n})\cap L^{q}_{t}W^{1,p}_{x}([0,\infty)\times\mathbb{R}^{n}),\quad a.s.
    \end{equation*}
    In addition, when $u_{0}\in\Sigma$, for any $0<T<\infty$ and any Strichartz admissible pair $(p,q)$, we have
    \begin{equation*}
        xu\in C_{t}L^2_x([0,T]\times\mathbb{R}^{n})\cap L^{q}_{t}L^{p}_{x}([0,T]\times\mathbb{R}^{n}),\quad a.s.
    \end{equation*}
\end{theorem}
 The next theorem concerns the almost sure scattering in the short-range case.
\begin{theorem}\label{thm:result_L^2_scattering}
    Let $u_{0}\in\Sigma$, $\frac{2}{n}<2\sigma<\frac{4}{n}$, and $\vert g(t)\vert=o(t^{-5/2})$, a.s., as $t\to\infty$. Then the global $L^{2}_{x}$ solution $u$ of \eqref{eq:additive_SNLS} scatters forward in time, i.e., almost surely there exists $u_{+}\in L^{2}_{x}(\mathbb{R}^{n})$ such that
    \begin{equation*}
        \lim_{t\to\infty}\Vert S(-t)u(t)-u_{+}\Vert_{L^{2}_{x}}=0.
    \end{equation*}
\end{theorem}
\begin{theorem}\label{thm:result_Sigma_scattering}
    Let $u_{0}\in\Sigma$, $\sigma(n)<2\sigma<\frac{4}{n-2}$ if $n\geq 3$, $\sigma(n)<2\sigma<\infty$ with $2\sigma\neq\frac{4}{n}$ if $n=1,2$, in which $\sigma(n)$ denotes the Strauss exponent  \begin{equation}\label{eq:Strauss_exponent}
        \sigma(n)=\frac{2-n+\sqrt{n^2+12n+4}}{2n}.
    \end{equation}
    Assume that $\vert g(t)\vert=o(t^{-5/2})$, a.s., as $t\to\infty$. Then the global $\Sigma$ solution $u$ of \eqref{eq:additive_SNLS} scatters forward in time, i.e., almost surely there exists $u_{+}\in \Sigma$ such that
    \begin{equation*}
        \lim_{t\to\infty}\Vert S(-t)u(t)-u_{+}\Vert_{\Sigma}=0.
    \end{equation*}
    Moreover, with an additional small data assumption \eqref{eq:bound_of_data_for_induction}, the same result holds for the $2\sigma = \frac{4}{n}$, $n=1,2$ case.
\end{theorem}
\begin{Remark}
    It is important to note that Theorem \ref{thm:result_Sigma_scattering} implies the almost sure scattering in $L^{2}_{x}$ with $\Sigma$ initial data up to energy-critical nonlinearities in all dimensions. Therefore, our result for $L^{2}_{x}$ scattering covers the entire energy sub-critical range if $n\geq 1$, while a small initial data assumption is required in the case where $2\sigma=\frac{4}{n}$ for $n=1,2$ (for the deterministic analogue with $n\geq 2$, see \cite{Tsutsumi_Yajima}). 
    
\end{Remark}
Lastly, excluding the mass sub-critical range as in the deterministic setting, we state the almost sure energy scattering of global $H^{1}_{x}$ solutions to \eqref{eq:additive_SNLS}.
\begin{theorem}\label{thm:result_H^1_scattering}
    Suppose that $u_{0}\in H^{1}_{x}(\mathbb{R}^{n})$, $\frac{4}{n} \leq 2\sigma<\frac{4}{n-2}$ if $n\geq 3$, $\frac{4}{n}\leq 2\sigma<\infty$ if $n=1,2$. Assume further that $\vert g(t)\vert=o(t^{-1})$, a.s., as $t\to\infty$. Then the global $H^{1}_{x}$ solution $u$ of \eqref{eq:additive_SNLS} scatters forward in time, i.e., almost surely there exists $u_{+}\in H^{1}_{x}(\mathbb{R}^{n})$ such that
    \begin{equation*}
        \lim_{t\to\infty}\Vert S(-t)u(t)-u_{+}\Vert_{H^{1}_{x}}=0.
    \end{equation*}
\end{theorem}

\begin{Remark}
    In our study of scattering of \eqref{eq:additive_SNLS}, we only consider the question of asymptotic completeness. The standard arguments for establishing the existence of the wave operators for deterministic PDEs require solving the equation backward in time, which could prove to be problematic when we consider the $\{\mathcal{F}_{t}\}$-adaptedness of the solutions.
\end{Remark}

\begin{Remark}
    For simplicity, we assume that $\phi(x)$ is a Schwartz function; nonetheless, it might be possible to obtain the aforementioned results under a weaker assumption on $\phi(x)$. 
\end{Remark}
In the rest of this section, we intend to mention the physical interpretations of the additive and multiplicative noises in a more general sense, and subsequently overview the related literature for NLS/SNLS. For this purpose, we shall record the concerning equations here: 
\begin{equation}\label{eq:general_form_additive}
    idu-\Delta udt+N(u)dt=dW,
\end{equation}
and
\begin{equation}\label{eq:general_form_multiplicative}
    idu-\Delta udt+N(u)dt=u \circ dW,
\end{equation}
where $\circ$ denotes the Stratonovich product.

The effect of additive noise on a system is independent of its current state, in sharp contrast to the multiplicative case, where the impact of the noise depends on the system's configuration. Therefore, stochastic equations with additive noise are typically used to model phenomena such as thermal fluctuations and random external forcing, whereas equations with multiplicative noise are employed to model state-dependent random perturbations and feedback mechanisms. The primary reason for considering additive and multiplicative noises separately is the assumption that either internal or external random effects are the primary contributor to the random behavior of the system.

A stochastic differential equation with decaying noise can be viewed as a perturbation of the deterministic equation. In the case of additive noise with no decay over time, energy is persistently injected into the system. Therefore, without a time decay assumption, we do not expect to observe dispersion, and this absence of dispersive behavior would consequently prevent the solutions from scattering. For additional observations on the effect of additive and multiplicative noises on the system, as well as their physical significance, we refer to \cite{moss1989noise1, moss1989noise2} and references therein. 

In the multiplicative noise case, depending on the decay behavior of solutions to the deterministic equation, it is possible that under weaker assumptions on the noise term, the decay of the solution could cause the stochastic convolution term in the Duhamel formulation to decay as well. However, in the additive noise case, since the solution does not appear in the stochastic convolution term, such behavior should not be expected. 

In the context of local and global well-posedness, \eqref{eq:general_form_additive} and \eqref{eq:general_form_multiplicative} were first considered by de Bouard and Debussche in \cite{de_Bouard_1999,deBouard_H1}. For the additive case \eqref{eq:general_form_additive}, the noise $W$ is given via a possibly infinite sum of $H^{1}_{x}$ image of orthonormal basis functions of $L^{2}_{x}$ under a class of Hilbert-Schmidt operators. In both focusing and defocusing cases, they established the global well-posedness in $H^{1}_{x}(\mathbb{R}^{n})$ (with the sub-criticality assumption) by applying a method known as Da Prato-Debussche trick \cite{Da_Prato_trick}. As for \eqref{eq:general_form_multiplicative}, by choosing a mass conservative setting, the authors demonstrated the global well-posedness in $H^{1}_{x}(\mathbb{R}^{n})$ under slightly more restrictive assumptions on the nonlinear term, see Theorem 4.1, 4.6 in \cite{deBouard_H1}. Also, for earlier works concerning blow-up behavior and numerical simulations, see \cite{de_Bouard_blow-up_1,de_Bouard_blow-up_3,Debussche_Numerical}.

In \cite{Barbu_2014_rescaling}, regarding the mass conservative and nonconservative cases, Barbu-Röckner-Zhang established the global well-posedness of SNLS with multiplicative noise \eqref{eq:general_form_multiplicative} in $L^{2}_{x}(\mathbb{R}^{n})$ for all spatial dimensions in the mass-subcritical case. As a continuation of their previous work, in \cite{Barbu_2016}, they also showed $H^{1}_{x}(\mathbb{R}^{n})$ global well-posedness in the defocusing case assuming energy sub-criticality. The same result was obtained in the focusing case only for the mass sub-critical range. 

Mass-critical and energy-critical SNLS with additive noise \eqref{eq:general_form_additive}, defined through a class of Hilbert-Schmidt operators, were investigated by Oh-Okamoto in \cite{Oh_Critical}. It was stated that the defocusing mass-critical SNLS is globally well-posed in $L^{2}_{x}(\mathbb{R}^{n})$ for $n\geq 1$, and the defocusing energy-critical SNLS is globally well-posed in $H^{1}_{x}(\mathbb{R}^{n})$ for the spatial dimensions $3\leq n\leq 6$. Later, by adapting the $I$-method (which is introduced by Colliander-Keel-Staffilani-Takaoka-Tao, \cite{Colliander2002}), the defocusing cubic SNLS with additive noise as in \eqref{eq:general_form_additive} (considered as the image of the space-time white noise under a class of Hilbert-Schmidt operators) was shown to be globally well-posed in $H^{s}(\mathbb{R}^{3})$ for $s>\frac{5}{6}$ by Cheung-Li-Oh, \cite{cheung2020conservation}. Note that this result coincides with the global well-posedness range for the deterministic equation, see \cite{Colliander2002}.

 In connection with the SNLS equations in different settings, we refer to \cite{Brzezniak_compact_1,Brzezniak_compact_2} for results on compact manifolds in the multiplicative noise case, also see \cite{chouk_modulated_dispersion,de_Bouard_modulated_dispersion} for Schrödinger equations with modulated dispersion and white noise dispersion.

Next, we first review the scattering of the deterministic nonlinear Schrödinger equations with power-type nonlinearities
\begin{equation}\label{eq:general_form_deterministic_NLS}
    i\partial_{t}u+\Delta u=\lambda \vert u\vert^{p-1}u
\end{equation}
in the spaces $L^{2}_{x}$, $\Sigma$ and $H^{1}_{x}$. One of the earliest results concerning the scattering of NLS \eqref{eq:general_form_deterministic_NLS} was obtained in \cite{Tsutsumi_Yajima}, which states that for $n\geq 2$, $1+\frac{2}{n}<p<1+\frac{4}{n-2}$, and $\Sigma$ initial data, the global solution $u(t)$ has a unique strong limit in $L^{2}_{x}(\mathbb{R}^{n})$ as $t\to\pm\infty$. In \cite{Ginibre_Sigma_Scattering,Tsutsumi_Sigma_Scattering}, the result of scattering-in-time in $\Sigma$ was established when $n\geq 1$ and $1+\sigma(n)<p<1+\frac{4}{n-2}$ ($1+\sigma(n)<p<\infty$ if $n=1,2$), where $\sigma(n)$ is given as in \eqref{eq:Strauss_exponent}, see Chapter 7.4 of \cite{Cazenave_book} for details. In \cite{Cazenave_1992_Scattering}, this result was partially improved in the defocusing case: for $n\geq 3$, $1+\frac{4}{n+2}<p<1+\frac{4}{n-2}$, and initial in $\Sigma$, the global solution $u(t)$ to \eqref{eq:general_form_deterministic_NLS} has a unique strong limit in $\Sigma$ as $t\to\pm\infty$, which eventually implies that the scattering holds in $\Sigma$. Moreover, $H^{1}_{x}$ scattering in the inter-critical case was demonstrated in \cite{Ginibre_Scattering_H1} for $n\geq 3$ and \cite{Nakanishi_scattering_H1} for $n=1,2$. The mass-critical case, $p=1+\frac{4}{n}$, was completely solved for all dimensions through the series of papers \cite{Dodson_dimension_3,Dodson_dimension_1,Dodson_dimension_2}. In relation to the scattering below the energy space, and scattering in $H^{1}_{x}$ for the energy critical equation, we refer to \cite{Tzirakis_Tensor,Colliander_Scattering_rough_dim_3,Colliander_energy_crit_scattering,Visan_energy_crit_higher_dim,Visan_I_Method}. For the long-range case (which corresponds to $1<p\leq 1+\frac{2}{n}$), it was shown in \cite{Barab_Nonexistence, Strauss_Nonexistence} that scattering cannot hold even in $L^{2}_{x}$. Moreover, for $1+\frac{2}{n}<p<1+\frac{4}{n}$, both $L^{2}_{x}$ and energy scattering problems are open, and only some partial results were obtained, see \cite{Holmer_subcrit,Gyu_Eun_Lee}.

As regards to the scattering in the probabilistic setting, it is remarkable to note that the scattering problem for SNLS especially with the  multiplicative noise \eqref{eq:general_form_multiplicative} was initiated quite recently, only few related results are available in the literature, we shall mention these works in what follows. In \cite{Herr_2019}, the scattering problem of \eqref{eq:general_form_multiplicative} for $n\geq 3$ was studied comprehensively by Herr-Röckner-Zhang. 
Under certain restrictions on the noise term, the authors established global well-posedness in $H^{1}_{x}$ and $\Sigma$. 
As a consequence of it, they showed almost sure scattering in $\Sigma$ for the defocusing problem when the exponent of the nonlinearity is between the Strauss exponent \eqref{eq:Strauss_exponent} and the energy-criticality exponent, including the energy-critical case as well. They also established the $H^{1}_{x}$ scattering, for $3\leq n\leq 6$, in the mass-critical, inter-critical, and energy-critical cases for the defocusing problem, in addition to this, by imposing further restrictions on the noise term, the authors showed with high probability that scattering holds for both focusing and defocusing problems in pseudo-conformal space and $H^{1}_{x}$ for all energy sub-critical nonlinearities. 

More recently, in \cite{Fan_Zhao}, the mass-critical problem in the case of $n=3$ has been investigated, demonstrating that time decay in the noise term is essential for proving $L^{2}_{x}$ scattering. Furthermore, in the absence of time decay, there are recent partial results under small noise assumption for the linear stochastic Schrödinger equation, see \cite{Fan_Linear_Schrodinger}. 

As mentioned above, the only results regarding the scattering problem of SNLS available in the literature address the multiplicative noise case. Therefore, to the best of our knowledge, the study of the long-time behavior of SNLS with additive noise \eqref{eq:general_form_additive} is initiated in this paper. Our approach relies on the Da Prato-Debussche trick, which we utilize to estimate the tail of the stochastic convolution in the Duhamel formulation. It is well-known that the study of stochastic convolution 
	\begin{equation*}
		z(t):=i\int_{0}^{t}S(t-s)dW(s)
	\end{equation*}
	is crucial for establishing the well-posedness of SPDE with additive noise, see \cite{dapratozabczyk}, \cite{de_Bouard_1999}.   Investigations in  this article demonstrate that the tail of the stochastic convolution
	\begin{equation*}
		z_*(t):=-i\int_{t}^{\infty}S(t-s)dW(s)
	\end{equation*}
	has a similar importance for the scattering problem for nonlinear SPDE with additive noise. While an integral like
	\begin{equation*}
		-i\int_{0}^{t}S(-s)dW(s)
	\end{equation*}
	would be a martingale with respect to $t$, and would be amenable to applying the Burkholder-Davis-Gundy inequality, the stochastic convolution is not a martingale and thus is harder to study. With the tail of the stochastic convolution, we have an even greater difficulty: it is not even an adapted process. So tools of stochastic calculus cannot be directly applied. Also, the results on it, as well as their applications to the scattering problem must be pathwise. 
	In this article, as well as in  \cite{Fan_Zhao}, \cite{Herr_2019}, the study of scattering is pathwise, that is, for the nonlinear SPDE in question, it is shown that almost every path of a solution behaves in limit like a solution of the linear deterministic counterpart. 

In view of all this, a thorough study of the tail of the stochastic convolution on Hilbert and Banach spaces is warranted.   As opposed to the wide literature on stochastic convolution, e.g. \cite{Hausenblas_Stochastic_convolution_1,Hausenblas_Stochastic_convolution_2,Neerven_stochastic_convolution_3,Neerven_stochastic_convolution_2,Neerven_stochastic_convolution_1,Veraar_stochastic_convolution,Brzezniak_stochastic_convolution}, we are aware of no such study. Therefore we intend to initiate an exploration of this subject in our forthcoming work.

It is worth noting that our argument in this paper is inspired by \cite{Herr_2019}, we redesign it in order for it to be compatible with the additive noise case by overcoming such difficulties mentioned above. The rescaling transformation in \cite{Herr_2019} employed to establish global-in-time Strichartz estimates introduces additional lower-order terms into the evolution operator. Estimating these additional terms requires the local smoothing estimates developed in \cite{MARZUOLA}, and this causes the restriction $n\geq 3$, while in our case, thanks to the linearity of our transformations, the evolution operator remains unchanged, which subsequently allows us to utilize the standard Strichartz estimates associated to the Schrödinger evolution operator.

The rest of the paper is organized as follows. In Section \ref{section:notation}, we introduce some notations and tools we use throughout the paper, as well as some preliminary observations, which we make use of later. In Section \ref{section:well-posedness}, we prove Theorem \ref{thm:result_well-posedness}, which provides local and global well-posedness. Section \ref{section:pseudo-conformal_energy} contains our observations on the growth of the pseudo-conformal energy, and we utilize these in sections \ref{section:scattering_L^2}, \ref{section:scattering_Sigma}, and \ref{section:scattering_H^1}, which will be useful in the proofs of Theorem \ref{thm:result_L^2_scattering}, Theorem \ref{thm:result_Sigma_scattering}, and Theorem \ref{thm:result_H^1_scattering}, respectively. Finally, we give the proof of Proposition \ref{lemma_deterministic_scattering} in the Appendix.

\section{Preliminaries}\label{section:notation}
\subsection{Function spaces}
The mixed Lebesgue space $L^q_t L^p_x (I \times \mathbb{R}^n)$ is defined by
\begin{equation*}
    L^{q}_{t}L^{p}_{x}(I\times\mathbb{R}^{n})=\{f:I\times\mathbb{R}^{n}\to\mathbb{C}\mid \Vert f\Vert_{L^{q}_{t}L^{p}_{x}(I\times\mathbb{R}^{n})}<\infty\}
\end{equation*}
where $p,q\geq 1$, $I$ is a measurable subset of $\mathbb{R}$, and
\begin{equation*}
    \Vert f\Vert_{L^{q}_{t}L^{p}_{x}(I\times\mathbb{R}^{n})}=\left(\int_{I}\left(\int_{\mathbb{R}^{n}}\vert f(t,x)\vert^{p}dx\right)^{\frac{q}{p}}dt\right)^{\frac{1}{q}}.
\end{equation*}
The Sobolev spaces $W^{s,p}_{x}(\mathbb{R}^{n})$, for $s\in\mathbb{R}$ and $1<p<\infty$, are defined as the closure of the Schwartz class functions under the norm
\begin{equation*}
    \Vert f\Vert_{W^{s,p}_{x}}=\Vert\langle\nabla\rangle^{s}f\Vert_{L^{p}_{x}}.
\end{equation*}
If $p=2$, thanks to the Plancherel identity, the space $H^{s}_{x}=W^{s,2}_{x}$ is defined by 
\begin{equation*}
    \Vert f\Vert_{H^{s}_{x}}=\Vert\langle\xi\rangle^{s}\widehat{f}\Vert_{L^{2}_{\xi}}
\end{equation*}
where the Fourier transform $\widehat{f}$ of $f$ is given by 
\begin{equation*}
    \widehat{f}(\xi)=\int_{\mathbb{R}^{n}}e^{-ix\cdot\xi}f(x)dx.
\end{equation*}
The pseudo-conformal space $\Sigma(\mathbb{R}^{n})=\Sigma$ is defined as
\begin{equation*}
    \Sigma=\{f:\mathbb{R}^{n}\to\mathbb{C} \mid \Vert f\Vert_{\Sigma}<\infty\}
\end{equation*}
where
\begin{equation*}
    \Vert f\Vert_{\Sigma}=\Vert f\Vert_{H^{1}_{x}(\mathbb{R}^{n})}+\Vert \vert \cdot \vert f\Vert_{L^{2}_{x}(\mathbb{R}^{n})}.
\end{equation*}
The Sobolev embedding $W^{s,p}_{x}(\mathbb{R}^{n})\hookrightarrow L^{q}_{x}(\mathbb{R}^{n})$ holds for $1<p<q<\infty$, $s>0$, and $\frac{1}{q}\geq\frac{1}{p}-\frac{s}{n}$. We call a pair $(p,q)$ Strichartz admissible, if $p,q\geq 2$, and the following relation is satisfied:
\begin{equation}\label{eq:Strichartz_admissible_pair}
    \frac{2}{q}=n\left(\frac{1}{2}-\frac{1}{p}\right)\quad\text{and}\quad (p,q,n)\neq(\infty,2,2). 
\end{equation}
Next, we give the well-known linear Strichartz estimate for the Schrödinger operator $S(t)=e^{-it\Delta}$ associated to \eqref{eq:additive_SNLS}:
\begin{lemma}[Linear Strichartz estimates \cite{taononlinear}]\label{stichartzlemma}
    Let $(p,q)$, $(p_{1},q_{1})$ be any pair of Strichartz admissible exponents and $(p',q')$, $(p'_{1},q'_{1})$ be their Hölder conjugates, respectively. Then, for any $s\in\mathbb{R}$, we have
    \begin{enumerate}[(i)]
        \item  $\Vert S(t) u_{0}\Vert_{L^{q}_{t}W^{s,p}_{x}(\mathbb{R}\times\mathbb{R}^{n})}\lesssim\Vert u_{0}\Vert_{H^{s}_{x}(\mathbb{R}^{n})}$,
        \item  $\left\Vert\int_{\mathbb{R}} S(-t')F(t')dt'\right\Vert_{L^{2}_{x}(\mathbb{R}^{n})}\lesssim\Vert F\Vert_{L^{q'}_{t}L^{p'}_{x}(\mathbb{R}\times\mathbb{R}^{n})}$,
        \item  $\left\Vert\int_{0}^{t} S(t-t')F(t')dt'\right\Vert_{L^{q}_{t}W^{s,p}_{x}(\mathbb{R}\times\mathbb{R}^{n})}\lesssim\Vert F\Vert_{L^{q'_{1}}_{t}W^{s,p'_{1}}_{x}(\mathbb{R}\times\mathbb{R}^{n})}$.
    \end{enumerate}
\end{lemma}
The stochastic Gronwall's inequality is stated as follows, see \cite{Xicheng_Zhang_Gronwall}. 
\begin{lemma}[Stochastic Gronwall's Inequality]\label{gronwall}
    Let $\xi(t)$ and $\eta(t)$ be two nonnegative cadlag $\mathcal{F}_t$-adapted process, $A_t$ a continuous nondecreasing $\mathcal{F}_t$-adapted processes with $A_0 = 0$ and $M_t$ a local martingale with $M_0 = 0$. Suppose that
\begin{equation*}
    \xi(t) \leq \eta(t) + \int_0^t \xi(s) dA_s + M_t
\end{equation*}
for any $t \geq 0$. Then for any $0 < q < p < 1$ and $\tau > 0$, we have
\begin{equation*}
    \left( \mathbb{E}\left[\sup\limits_{0 \leq t \leq \tau} \xi(t)^q \right] \right)^{1/q} \leq \left( \frac{p}{p-q} \right)^{1/q} \big( \mathbb{E}\left[\exp\left( p A_{\tau} / (1-p) \right) \right] \big)^{(1-p)/p} \mathbb{E}\left[\sup\limits_{0 \leq t \leq \tau} \eta(t) \right].
\end{equation*}
\end{lemma}
In addition, we need the following lemma \cite[Lemma 2.17]{saanouni2015remarks}:
\begin{lemma}\label{lemma:uniform_bounded_by_its_higher_powers}
Let $T>0$ and $f\in C([0,T];\mathbb{R}_{+})$, such that
\begin{equation*}
    f(t)\leq a+b f(t)^{\alpha},\quad \text{for $t\in[0,T]$,}
\end{equation*}
where $a,b>0$, $\alpha>1$, $a<(1-\frac{1}{\alpha})(\alpha b)^{-\frac{1}{\alpha-1}}$, and $f(0)\leq (\alpha b)^{-\frac{1}{\alpha-1}}$. Then
\begin{equation*}
    f(t)\leq \frac{\alpha}{\alpha-1}a,\quad\text{for all $t\in[0,T]$}.
\end{equation*}
\end{lemma}
\subsection{Estimate on the tail of the stocastic convolution}
Our goal in this part is to draw inferences on the asymptotic time-decay rate of
the tail of the stochastic convolution. In this regard, we first write the solution $u(t)$ to \eqref{eq:additive_SNLS} in the Duhamel formulation
\begin{equation*}
    u(t) = S(t)u_0 + i \int_0^t S(t-s) |u|^{2\sigma} u ds +i \int_0^t S(t-s) dW(s).
\end{equation*}
For convenience, let us denote the stochastic convolution in the above formulation as 
\begin{equation*}
    z(t) = i \int_0^t S(t-s) dW(s).
\end{equation*}
The local-in-time estimates for $z(t)$ will be demonstrated in the next section to establish the local and global well-posedness of the equation; however, our main focus for the moment is to determine the asymptotic time-decay rate of the tail of the stochastic convolution. To this end, let us define
\begin{equation*}
    z_*(t) = -i \int_t^\infty S(t-s) dW(s),
\end{equation*}
where $t \in [T,\infty)$ for a large time parameter $T$. We shall write (as in the Da Prato-Debussche trick) $u(t) = u_*(t) + z_*(t)$. Note that $u_*(t)$ satisfies the following random differential equation
\begin{equation}\label{eq:additive_random_eqn}
    idu_{*} -\Delta u_{*}dt + \vert u_{*}+z_{*}\vert^{2\sigma}(u_{*}+z_{*})=0.
\end{equation}
Further note that $u_*(t)$ and $u(t)$ will have the same asymptotic behavior if we can establish sufficient time decay for $\Vert z_*(t) \Vert_{L^p_x}$, $p \geq 2$. Now let us consider the deterministic NLS
\begin{equation}\label{eq:deterministic_NLS}
    idy - \Delta y dt + |y|^{2\sigma}y dt = 0.
\end{equation}
Let $T$ be a large time and $u_{*}(T)=y(T)$. Simply, for each outcome $\omega\in\Omega$, we take a different $y(T)$, and turn \eqref{eq:deterministic_NLS} into a random differential equation. Set $v:=u_{*}-y$ so that $v(T)=0$ and $v$ satisfies the following random IVP:
\begin{equation}\label{eq:differ_rand_and_det_soln_IVP}
    \begin{cases}
        &idv - \Delta v + \vert v + y + z_{*}\vert^{2\sigma}(v + y + z_{*}) - \vert y\vert^{2\sigma}y = 0\\
        & v(T)=0,
    \end{cases}
    \quad (t,x)\in[T,\infty)\times\mathbb{R}^{n}.
\end{equation}
To study $z_*(t)$ in more detail, let us first note that for fixed $t$ and $x$, the stochastic convolution term $M(t,x,T) := \int_t^T S(t-r) \phi(x) g(r) dB(r)$ is a continuous martingale. The martingale property holds since, given an increasing sequence of stopping times $\{\tau_i\}$, $i \in \{1, 2, \dots , n\}$, we have
\begin{align*}
    \mathbb{E}&[M(t,x,\tau_n) \mid M(t,x,\tau_1), \dots, M(t,x,\tau_{n-1})] \\ 
    &= \mathbb{E}[M(t,x,\tau_{n-1}) \mid M(t,x,\tau_1), \dots, M(t,x,\tau_{n-1})] + \mathbb{E}[M(\tau_{n-1},x,\tau_n) \mid M(t,x,\tau_1), \dots, M(t,x,\tau_{n-1})] \\
    &= M(t,x,\tau_{n-1}) + \mathbb{E}[M(\tau_{n-1},x,\tau_n)] = M(t,x,\tau_{n-1}).
\end{align*}
By Theorem 3.1 of \cite{revuz2013continuous}, the martingale $M(t, x, T)$ is bounded in $L^{2}$, hence converges in $L^{2}$ and almost surely as $T \rightarrow \infty$. Moreover, Itô isometry yields
\begin{equation*}
    \mathbb{E}\left[\sup\limits_{T \geq t}|M(t, x, T)|^{2}\right] = \mathbb{E}\left[ \int_{t}^{T}|S(t-r)\phi(x)|^{2}|g(r)|^{2} dr\right] \leq \sup\limits_{t, r, x}|S(t-r)\phi(x)|^2 \int_{t}^{\infty} \mathbb{E}\left[|g(r)|^{2}\right] dr < \infty.
\end{equation*}
Now let $p>2$ and $\gamma=n\left(\frac{1}{2}-\frac{1}{p}\right)$. Then with an application of Sobolev embedding, we obtain
\begin{align*}
\mathbb{E}\left[ \sup _{s \geq t} t^{\beta}\left\|z_{*}(s)\right\|_{W_{x}^{1,p}}^{p} \right] &= t^{\beta} \mathbb{E}\left[ \sup _{s \geq t}\left\|z_{*}(s)\right\|_{W_{x}^{1,p}}^{p} \right] \lesssim t^{\beta} \mathbb{E}\left[ \sup _{s \geq t}\left\|z_{*}(s)\right\|_{H^{1+\gamma}}^{p}\right] \\ 
&=t^{\beta} \mathbb{E} \left[ \sup\limits_{s \geq t}\left\|\int_{s}^{\infty} S(s-r) d W(r)\right\|_{H^{1+\gamma}}^{p} \right] =t^{\beta} \mathbb{E}\left[ \sup _{s \geq t}\left\|\int_{s}^{\infty} S(-r) d W(r)\right\|_{H^{1+\gamma}}^{p} \right].
\end{align*}
We use triangle inequality 
\begin{equation*}
    \left|\int_{s}^{\infty} S(-r) d W(r)\right| \leq\left|\int_{t}^{\infty} S(-r) d W(r)\right|+\left|\int_{t}^{s} S(-r) d W(r)\right|
\end{equation*}
which implies
\begin{equation*}
    \left\|\int_{s}^{\infty} S(-r) d W(r)\right\|_{H_{x}^{1+\gamma}}^{p} \lesssim\left\|\int_{t}^{\infty} S(-r) d W(r)\right\|_{H_{x}^{1+\gamma}}^{p}+\left\|\int_{t}^{s} S(-r) d W(r)\right\|_{H_{x}^{1+\gamma}}^{p}
\end{equation*}
and hence
\begin{equation*}
    t^{\beta} \mathbb{E} \left[ \sup_{s \geq t}\left\|\int_{s}^{\infty} S(-r) dW(r) \right\|_{H_{x}^{1+\gamma}}^{p} \right] \lesssim t^{\beta} \mathbb{E}\left[\left\|\int_{t}^{\infty} S(-r) d W(r)\right\|_{H_{x}^{1+\gamma}}^{p}\right] +t^{\beta} \mathbb{E} \left[ \sup_{s \geq t}\left\|\int_{t}^{s} S(-r) d W(r)\right\|_{H_{x}^{1+\gamma}}^{p} \right]
\end{equation*}
holds. Now we can apply Burkholder-Davis-Gundy inequality for Hilbert space-valued martingales (as developed in \cite{dapratozabczyk} and related literature) to obtain the bound
\begin{align*}
t^{\beta} \mathbb{E}\left[\left\|\int_{t}^{\infty} S(-r) d W(r)\right\|_{H_{x}^{1+\gamma}}^{p}\right] &+ t^{\beta} \mathbb{E}\left[\sup _{s \geq t}\left\|\int_{t}^{s} S(-r) d W(r)\right\|_{H_{x}^{1+\gamma}}^{p}\right] \\
&\lesssim t^{\beta} \mathbb{E}\left[\left(\int_{t}^{\infty}\|S(-r) \phi(x)\|_{H_{x}^{1+\gamma}}^{2}|g(r)|^{2} d r\right)^{\frac{p}{2}}\right] \\
&= t^{\beta}\|\phi\|_{H_{x}^{1+\gamma}}^{p} \mathbb{E}\left[\left(\int_{t}^{\infty}|g(r)|^{2} d r\right)^{p / 2}\right] \lesssim t^{\beta+\frac{p}{2}(-2 \alpha+1)}\|\phi\|_{H_{x}^{1+\gamma}}^{p}.
\end{align*}
So if we choose $\beta = \frac{p}{2}(2\alpha-1)$, we see that $\mathbb{E}\left[ \sup\limits_{s \geq t} t^{\beta}\left\|z_{*}(s)\right\|_{L_{x}^{p}}^{p}\right] \leq C$.
Next for the Borel-Cantelli argument, consider the event $\Omega_{n}:=\left\{\sup\limits_{s \geq 2^{n}} 2^{n \beta}\left\|z_{*}(s)\right\|_{L_{x}^{p}}^{p}>\right.$ $\left.n^{2}\right\}$, where $n \in \mathbb{N}$. By Markov's inequality, this event has probability at most $C n^{-2}$. Thus
$$
\sum_{n \in \mathbb{N}} \mathbb{P}\left[\Omega_{n}\right] \leq 2 C.
$$
As a result of the Borel-Cantelli lemma, for a.s $\omega$, there exists $n(\omega)$ such that for $n \geq n(\omega)$,
$$
\sup\limits_{s \geq 2^{n}} 2^{n \beta}\left\|z_{*}(s)\right\|_{W_{x}^{1,p}}^{p} \leq n^{2} \quad a.s.
$$
implying that
$$
\sup\limits_{s \geq 2^{n}}\left\|z_{*}(s)\right\|_{W_{x}^{1,p}}^{p} \leq 2^{-n \beta} n^{2} \quad a.s.
$$
Let $s \geq 2^{n(\omega)}$. Then there exists $n \in \mathbb{N}$ such that $2^{n} \leq s<2^{n+1}$ and, by the above inequalities, we get
$$
\left\|z_{*}(s)\right\|_{W_{x}^{1,p}}^{p} \leq 2^{-n \beta} n^{2} \leq(s / 2)^{-\beta} \log^{2}s \lesssim s^{\frac{p}{2}(1-2\alpha)}\log^{2}s \quad a.s.
$$
In particular, we may write
\begin{equation}\label{eq:almost_sure_time_decay_of_tail}
    \Vert z_{*}(s)\Vert_{W^{1,p}_{x}}\lesssim\langle s\rangle^{-\alpha+\frac{1}{2}+},\quad a.s.,
\end{equation}
whenever
\begin{equation*}
    \begin{cases}
        2\leq p\leq\infty&\text{if $n=1$},\\
        2\leq p<\infty&\text{if $n=2$},\\
        2\leq p\leq\frac{2n}{n-2}&\text{if $n>2$},
    \end{cases}
\end{equation*}
and $g(s)=o(s^{-\alpha})$ as $s\to\infty$ a.s. Now, let $2\leq q\leq\infty$ be chosen so that the pair $(p,q)$ satisfies \eqref{eq:Strichartz_admissible_pair}. Then for any $0<T<\infty$, almost surely we have
\begin{align*}
    \Vert z_{*}\Vert_{L^{q}_{t}W^{1,p}_{x}([T,\infty)\times\mathbb{R}^{n})}=&\left(\int_{T}^{\infty}\Vert z_{*}(s)\Vert_{W^{1,p}_{x}}^{q}\,ds\right)^{\frac{1}{q}}\\
    \lesssim&\left(\int_{T}^{\infty}\langle s\rangle^{\frac{4p}{n(p-2)}(-\alpha+\frac{1}{2}+)}\, ds\right)^{\frac{1}{q}}\\
    \sim&\left(\langle s\rangle^{\frac{4p}{n(p-2)}(-\alpha+\frac{1}{2}+)+1}\bigg|_{s=T}^{\infty}\right)<\infty
\end{align*}
as long as $\frac{4p}{n(p-2)}(-\alpha+\frac{1}{2}+)+1<0$, which is equivalent to
\begin{equation*}
    \alpha-\frac{1}{2}>\frac{n}{4}-\frac{n}{2p}+=\frac{1}{q}+.
\end{equation*}
Since $0\leq\frac{1}{q}\leq\frac{1}{2}$, taking $\alpha>1$ guarantees that
\begin{equation}\label{eq:global_Strichartz_est_for_tail}
    \Vert z_{*}\Vert_{L^{q}_{t}W^{1,p}([T,\infty)\times\mathbb{R}^{n})}<\infty,\quad a.s.
\end{equation}
whenever $(p,q)$ satisfies \eqref{eq:Strichartz_admissible_pair}. We remark that more restrictive asymptotic decay on $g$ is needed to be assumed so as to conclude almost sure scattering in $L^{2}_{x}$ or $\Sigma$, see Theorem \ref{thm:result_L^2_scattering} and \ref{thm:result_Sigma_scattering}.
\subsection{Notations}
Throughout, we use $C$ to denote the universal constant, which may change line by line. Also, the notation $C(\beta)$ is used for a constant depending on $\beta$. For nonnegative quantities $X,Y$, we denote $X\lesssim Y$ if there is a constant $C>0$ independent of $X,Y$ such that $X\leq CY$ and we write $X\lesssim_{a,b}Y$ if the implicit constant depends on parameters $a,b$. We denote $X\sim Y$ if both $X\lesssim Y$ and $Y\lesssim X$ hold. The notation $X\ll Y$ is used when $X<cY$ for a sufficiently small constant $c>0$, depending on the context. Also, the notation $\langle \cdot \rangle$ is used to denote $\langle X\rangle=(1+\vert X\vert^{2})^{\frac{1}{2}}$. We write $X^{\alpha\pm}=X^{\alpha\pm\varepsilon}$ for $0<\varepsilon\ll 1$. 
\section{Well-posedness in $H^1_x$: Proof of Theorem \ref{thm:result_well-posedness}} \label{section:well-posedness}
Since the noise term is given by Schwartz $\phi(x)$ and $g(t, \omega)$ which is decaying in time, local and global well-posedness in $H^1_x$ can be shown to hold in an identical fashion as in \cite{deBouard_H1}. Hence, we will present some additional results in this section: Proposition \ref{proposition:boundedness_of_mass_and_hamiltonian} shows that arbitrary moments of mass and Hamiltonian are almost surely bounded for all finite time, and Proposition \ref{prop_Strichartz_estimate_for_sigma_space} shows that an $H^1_x$ solution of \eqref{eq:additive_SNLS} with initial data in $\Sigma$ stays in $\Sigma$ for all finite time.
\begin{proposition}\label{proposition:boundedness_of_mass_and_hamiltonian}
        Assume $0 < 2\sigma < \frac{4}{n-2}$ if $n \geq 3$ and $0 < \sigma < \infty$ if $n=1,2$. Assume that the initial data $u_0 \in H^{1}_{x}(\mathbb{R}^{n})$ and $0< T < \infty$. Then for any $p \geq 1$, we have 
        \begin{align*}
          \mathbb{E}\left[\sup\limits_{0 \leq t < \infty} [M^p(u(t)) + H^p(u(t))]\right] \leq C(M(u_0), H(u_0), p, \phi, g).  
        \end{align*}
        \end{proposition}

\begin{proof}
A formal application of the Itô formula gives that
\begin{align*}
    H(u(t))&=H(u(0))-\Im\int_{0}^{t}\int_{\mathbb{R}^{n}}(\Delta\overline{u}-\vert u\vert^{2\sigma}\overline{u})\phi g dxdB+\frac{1}{2}\int_{0}^{t}\int_{\mathbb{R}^{n}}\vert\nabla\phi\vert^{2}g^{2}dxdt\\&\quad+\frac{1}{2}\int_{0}^{t}\int_{\mathbb{R}^{n}}\vert u\vert^{2\sigma}\vert\phi\vert^{2}g^{2}dxdt
    +\sigma\int_{0}^{t}\int_{\mathbb{R}^{n}}\vert u\vert^{2\sigma-2}\left(\Im(u\overline{\phi})\right)^{2} g^2 dxdt\\&=:H(u(0))+I+II+III+IV.
\end{align*}
For the term $I$, applying an integration by parts in the spatial integral, using Burkholder-Davis-Gundy inequality, Hölder's inequality and then Young's inequality, we obtain
\begin{align*}
    \mathbb{E}&\left[\sup_{0\leq t<T}\vert I\vert\right]\\&\leq\mathbb{E}\left[\sup_{0\leq t<T}\left\vert\int_{0}^{t}\int_{\mathbb{R}^{n}}\nabla u \cdot \nabla\phi g dxdB \right\vert + \sup_{0\leq t<T} \left\vert \int_{0}^{t}\int_{\mathbb{R}^{n}}\vert u\vert^{2\sigma} \overline{u} \phi g dxdB \right\vert \right]\\
    &\lesssim \mathbb{E}\left[\sup_{0\leq t<T}\left(\int_{0}^{t}\left(\int_{\mathbb{R}^{n}}\vert\nabla u\vert\vert\nabla\phi\vert dx\right)^2 \vert g\vert^2 dt\right)^{1/2}+ \sup_{0\leq t<T} \left(\int_{0}^{t}\left(\int_{\mathbb{R}^{n}}\vert u\vert^{2\sigma+1}\vert\phi\vert dx\right)^2 \vert g\vert^2 dt\right)^{1/2}\right]\\
    &\lesssim\mathbb{E}\left[\sup_{0\leq t<T}(\Vert\nabla u(t)\Vert_{L^{2}_{x}}\Vert\nabla\phi\Vert_{L^{2}_{x}}+\Vert u(t)\Vert_{L^{2\sigma+2}_{x}}^{2\sigma+1}\Vert\phi\Vert_{L^{2\sigma+2}_{x}})\Vert g\Vert_{L^{2}_{t}([0,T))}\right]\\
    &\lesssim\mathbb{E}\left[C(\phi,g,T)\left(\sup_{0\leq t<T}H(u(t))^{\frac{1}{2}}+\sup_{0\leq t<T}H(u(t))^{\frac{2\sigma+1}{2\sigma+2}}\right)\right]\\
    &\leq C(\phi,g,T)+\frac{1}{10}\mathbb{E}\left[\sup_{0\leq t<T}H(u(t)) \right].
\end{align*}
The expectation of the second term $II$ can be controlled by a constant $C(\phi,g,T)$. For the term $III$, using Young's inequality, we have
\begin{align*}
    \mathbb{E}\left[\sup_{0\leq t<T}\vert III\vert\right]&\lesssim\mathbb{E}\left[\sup_{0\leq t<T}\int_{0}^{t}\Vert u(t)\Vert_{L^{2\sigma+2}_{x}}^{2\sigma}\Vert\phi\Vert_{L^{2\sigma+2}_{x}}^{2}\,g^{2}(s)ds\right]\\
    &\leq C(\phi,g,T)+\frac{1}{10}\mathbb{E}\left[\sup_{0\leq t<T}H(u(t))\right].
\end{align*}
Notice that the expectation of $\vert IV\vert$ can be controlled by the same upper bound of the term $III$. Therefore, by combining all of the above estimates, we obtain
\begin{equation*}
    \mathbb{E}\left[\sup_{0\leq t<T} H(u(t))\right]\leq \frac{10}{7}\mathbb{E}\left[H(u(0))\right]+C(\phi,g,T),
\end{equation*}
which yields that the first moment of the Hamiltonian is bounded for all finite time. Also, in relation to the mass $M(u(t)) = \int_{\mathbb{R}^n} |u(t)|^2 dx$, Itô's lemma gives
\begin{equation*}
    M(u(T)) = M(u(0)) + 2\Im\int_0^T \int_{\mathbb{R}^n} u(t)\overline{\phi}(x) g(t) dx dB(t) + \Vert \phi \Vert_{L^2_x}^2 \int_0^T g^2(t) dt.
\end{equation*} 
Proceeding as we did for $H(u(t))$, it is easy to conclude that the mass is bounded for all finite time. Our aim now is to show that arbitrary moments of $M(u(t))+H(u(t))$ are bounded almost surely as well. Let us first start with the Itô formula for $M(u(t))+H(u(t))$:
\begin{align*}
     H(u(T)) + M(u(T)) &= H(u(0)) + M(u(0)) + 2\Im\int_0^T \int_{\mathbb{R}^n} u(t)\overline{\phi}(x) g(t) dx dB(t) + \Vert \phi \Vert_{L^2_x}^2 \int_0^T g^2(t) dt \\ 
     &\quad- \Im\int_{0}^{T}\int_{\mathbb{R}^{n}}(\Delta\overline{u}-\vert u\vert^{2\sigma}\overline{u})\phi g dxdB+\frac{1}{2}\int_{0}^{T}\int_{\mathbb{R}^{n}}\vert\nabla\phi\vert^{2}g^{2}dxdt \\
     &\quad+\frac{1}{2}\int_{0}^{T}\int_{\mathbb{R}^{n}}\vert u\vert^{2\sigma}\vert\phi\vert^{2}g^{2}dxdt + \sigma\int_{0}^{T}\int_{\mathbb{R}^{n}}\vert u\vert^{2\sigma-2}\left(\Im(u\overline{\phi})\right)^{2} g^2 dxdt
\end{align*}
Instead of applying Ito's formula for $(H(u(T))+M(u(T)))^k$, we will instead consider $(K(\phi) + H(u(T)) + M(u(T)))^k$ for some constant $K(\phi)$ depending on $\phi$. Showing that the expectation of this quantity is bounded will imply that the expectation of $(H(u(T))+M(u(T)))^k$ is also bounded. We have:
\begin{align*}
    (K(\phi) + H(u(T))+M(u(T)))^k &= \left(K(\phi) + H(u(0)) + M(u(0))\right)^k \\
    &\quad+ k \int_0^T (K(\phi) + H(u(t))+M(u(t)))^{k-1} F_1(u(t)) dt \\ 
    &\quad +\frac{k(k-1)}{2} \int_0^T (K(\phi) + H(u(t))+M(u(t)))^{k-2} F_2(u(t))^2 dt \\
    &\quad+ k\int_0^T (K(\phi) + H(u(t))+M(u(t)))^{k-1} F_2(u(t)) dB(t),
\end{align*}
where $F_1(u(t))$ and $F_2(u(t))$ are given by
\begin{align*}
    F_1(u(t)) &= \Vert \phi \Vert_{L^2_x}^2 g^2(t) + \frac{1}{2}\int_{\mathbb{R}^{n}}\vert\nabla\phi\vert^{2}g^{2}dx + \frac{1}{2}\int_{\mathbb{R}^{n}}\vert u\vert^{2\sigma}\vert\phi\vert^{2}g^{2}dx + \sigma\int_{\mathbb{R}^{n}}\vert u\vert^{2\sigma-2}\left(\Im(u\overline{\phi})\right)^{2} g^2 dx \\& =: I + II + III + IV,
\end{align*}
and
\begin{equation*}
    F_2(u(t)) = 2\Re \int_{\mathbb{R}^n} u(t)\overline{\phi}(x) g(t) dx - \Im\int_{\mathbb{R}^{n}}(\Delta\overline{u}-\vert u\vert^{2\sigma}\overline{u})\phi g dx =: V + VI.
\end{equation*}
Note that $I + II \leq C(\phi) g^2$. For the term $III$, Hölder's inequality yields $III \leq \frac{g^2}{2} \Vert \phi \Vert_{L^{2\sigma + 2}_x}^2 \Vert u(t) \Vert_{L^{2\sigma +2}_{x}}^{2\sigma}$. Then, applying Young's inequality, we obtain $III \leq C(\phi)g^2 + \Vert u(t) \Vert_{L^{2\sigma +2}_x}^{2\sigma + 2} g^2 \leq C(\phi)g^2 + \frac{1}{2}(K(\phi) + H(u(t)) + M(u(t)))g^2$. Similarly, $IV \leq C(\phi)g^2 + \frac{1}{2}(K(\phi) + H(u(t)) + M(u(t)))g^2$. Therefore, we have
\begin{equation*}
    F_1(u(t)) \leq C(\phi)g^2 + (K(\phi) + H(u(t)) + M(u(t)))g^2.
\end{equation*}
For the term $V$, we can apply the Cauchy-Schwarz inequality in the $x$ variable to obtain $V \leq C(\phi) |g| \Vert u \Vert_{L^{2}_{x}}$, then Young's inequality gives us that $V \leq C(\phi)|g| + \frac{1}{2}(K(\phi) + H(u(t))+M(u(t))) |g|$. As for $VI$, the term $|u|^{2\sigma}\overline{u}$ can be dealt with analogously, and shown to be bounded by $ C(\phi)|g| + \frac{1}{4}(K(\phi) + H(u(t))+M(u(t))) |g|$. Regarding the term $\Delta \overline{u}$, integrating by parts twice and using Hölder's and Young's inequalities reveals the bound $ C(\phi)|g| + \frac{1}{4}(K(\phi) + H(u(t))+M(u(t))) |g|$ for this term. Therefore, we eventually get 
\begin{equation*}
    F_2(u(t)) \leq C(\phi)\vert g \vert + (K(\phi) + H(u(t))+M(u(t))) \vert g \vert.
\end{equation*} 
Note that $ C(\phi) k \int_0^T (K(\phi) + H(u(t))+M(u(t)))^{k-1} F_2(u(t)) dB(t) =: M_T$ is a continuous martingale. Combining these observations, we arrive at:
\begin{align*}
    (K(\phi) + H(u(T))+M(u(T)))^k &\leq (k^2-k) C(\phi) \int_0^T (K(\phi) + H(u(t))+M(u(t)))^{k} g^2 dt \\
    &\quad+ C(H(u_0), M(u_0), \phi) + M_T.
\end{align*}
Finally we can use the stochastic Gronwall's inequality to obtain
\begin{align*}
    \left(\mathbb{E}\left[ \sup\limits_{0 \leq t \leq \tau} (K(\phi) + H(u(t)) + M(u(t)))^{kq} \right] \right)^{1/q} &\leq \left( \frac{p}{p-q} \right)^{1/q} C(H(u_0), M(u_0), \phi) \\
    &\quad\times \left( \mathbb{E}\left[\exp\left( \frac{p}{1-p} C(\phi) (k^2-k) \int_0^\tau g^2(t) dt \right) \right] \right)^{(1-p)/p}
\end{align*}
for an arbitrary stopping time $\tau$. By our assumption on $g$, we see that the RHS of the above inequality is independent of $\tau$. Since $k$ is any integer and $0<q<p<1$ are arbitrary as well, this implies that $$\mathbb{E}\left[\sup\limits_{0 \leq t \leq \tau} (K(\phi) + H(u(t)) + M(u(t)))^{r} \right]$$ is bounded for any $0 < r < \infty$.
\end{proof}
\begin{proposition}\label{prop_Strichartz_estimate_for_sigma_space}
    Assume $0 < 2\sigma < \frac{4}{n-2}$ if $n \geq 3$ and $0 < 2\sigma < \infty$ if $n=1,2$. Let $u_0 \in \Sigma$. Then for the corresponding $H^{1}_{x}$ solution $u$ to \eqref{eq:additive_SNLS}, we have  $xu\in L^{\frac{4\sigma + 4}{n\sigma}}_t L^{2\sigma + 2}_x([0,T)\times\mathbb{R}^{n})$ a.s. for any $T>0$.
\end{proposition}
\begin{proof}
    We will only present the proof for $n=1,2$. For $n \geq 3$, choosing $(p,q) = \left( \frac{n(2\sigma+2)}{n+2\sigma}, \frac{4\sigma+4}{(n-2)\sigma} \right)$ and proceeding as in \cite[Lemma 2.1]{Herr_2019} yields the desired result. Since $u \in L^q_tW^{1,r}_x([0,T)\times\mathbb{R}^{n})$ a.s. for any admissible pair $(p,q)$, we can take a finite partition $\{I_k\} := \{[t_k, t_{k+1}]\}$ of $[0,T]$ satisfying $\Vert u \Vert_{L^q_t W^{1,r}_x (I_k)} < \varepsilon$ for any admissible pair $(p,q) \neq (2,\infty)$. Using the commutation relation $x_j S(t) = S(t) (x_j + 2it \partial_{x_j})$, we have
\begin{equation} \label{xju}   
\begin{aligned}
    x_j u(t) &= S(t) (x_j u_0 + 2it\partial_{x_j} u_0) + i \int_0^t S(t-s) (x_j + 2i(t-s)\partial_{x_j})(\vert u \vert^{2\sigma}u(s))ds \\&\quad
    + i \int_0^t S(t-s) (x_j + 2i(t-s)\partial_{x_j}) \phi(x) g(s) dB(s).
\end{aligned}
\end{equation}
For the first term, we have
\begin{equation*}
    \Vert S(t) x_j u_0 \Vert_{L^{\frac{4\sigma + 4}{n\sigma}}_t L^{2\sigma + 2}_x} \lesssim \Vert x_j u_0 \Vert_{L^2_x} \lesssim \Vert u_0 \Vert_{\Sigma},
\end{equation*}
which is finite since $u_0 \in \Sigma$. Also, we have
\begin{equation*}
    \Vert S(t) (t \partial_{x_j} u_0) \Vert_{L^{\frac{4\sigma + 4}{n\sigma}}_t L^{2\sigma + 2}_x} \lesssim T \Vert \partial_{x_j} u_0 \Vert_{L^2_x} \leq C_T \Vert u_0 \Vert_{H^1_x}.
\end{equation*}
For the second term, we apply Strichartz estimates, Hölder's inequality, Sobolev embedding and make use of our partition to obtain
\begin{align*}
    \left \Vert \int_0^t S(t-s) x_j \vert u \vert^{2\sigma} u(s) ds \right \Vert&_{L^{\frac{4\sigma + 4}{n\sigma}}_t L^{2\sigma + 2}_x (I_k)} \\&\lesssim \Vert x_j \vert u \vert^{2\sigma} u \Vert_{L^{\frac{4\sigma + 4}{4\sigma + 4 - n\sigma}}_t L^{\frac{2\sigma + 2}{2\sigma + 1}}_x (I_k)} \\
    &\lesssim \Vert x_j u \Vert_{L^{\frac{4\sigma + 4}{n\sigma}}_t L^{2\sigma +2}_x (I_k)} \Vert u \Vert_{L^{\frac{2\sigma(4\sigma+4)}{4\sigma + 4 - 2n\sigma}}_t L^{2\sigma + 2}_x (I_k)}^{2\sigma} \\&\lesssim \Vert x_j u \Vert_{L^{\frac{4\sigma + 4}{n\sigma}}_t L^{2\sigma +2}_x (I_k)} \Vert u \Vert_{L^{\frac{2\sigma(4\sigma+4)}{4\sigma + 4 - 2n\sigma}}_t W^{s,2+\theta}_x (I_k)}^{2\sigma} \\
    &\leq C_T \Vert u \Vert_{L^{\frac{8+4\theta}{n\theta}}_t W^{1,2+\theta}_x (I_k)}^{2\sigma} \Vert x_j u \Vert_{L^{\frac{4\sigma + 4}{n\sigma}}_t L^{2\sigma +2}_x (I_k)} \\&\leq C_T \varepsilon^{2\sigma} \Vert x_j u \Vert_{L^{\frac{4\sigma + 4}{n\sigma}}_t L^{2\sigma +2}_x (I_k)}
\end{align*}
where $\theta>0$ is small enough to satisfy $\frac{8+4\theta}{n\theta} > \frac{4\sigma(\sigma+1)}{2\sigma+2-n\sigma}$ and $s = n\left( \frac{1}{2+\theta} - \frac{1}{2\sigma+2} \right)$, which is always less than 1 for any $\theta > 0$. Similarly, we have
\begin{align*}
    &\left \Vert \int_0^t S(t-s) 2i(t-s) \partial_{x_j} \vert u \vert^{2\sigma} u(s) ds \right \Vert_{L^{\frac{4\sigma + 4}{n\sigma}}_t L^{2\sigma + 2}_x (I_k)} \\
    &\quad \leq 2T\left \Vert \int_0^t S(t-s) \partial_{x_j} \vert u \vert^{2\sigma} u(s) ds \right \Vert_{L^{\frac{4\sigma + 4}{n\sigma}}_t L^{2\sigma + 2}_x (I_k)} + \left \Vert \int_0^t S(t-s) (2s) \partial_{x_j} \vert u \vert^{2\sigma} u(s) ds \right \Vert_{L^{\frac{4\sigma + 4}{n\sigma}}_t L^{2\sigma + 2}_x (I_k)} \\
    &\quad\leq C_T \Vert \partial_{x_j} \vert u \vert^{2\sigma} u \Vert_{L^{\frac{4\sigma + 4}{4\sigma + 4 - n\sigma}}_t L^{\frac{2\sigma + 2}{2\sigma + 1}}_x (I_k)} + \Vert 2t \partial_{x_j} \vert u \vert^{2\sigma} u \Vert_{L^{\frac{4\sigma + 4}{4\sigma + 4 - n\sigma}}_t L^{\frac{2\sigma + 2}{2\sigma + 1}}_x (I_k)} \\& \quad
    \leq C_T \Vert \partial_{x_j} u \Vert_{L^{\frac{4\sigma + 4}{n\sigma}}_t L^{2\sigma +2}_x (I_k)} \Vert u \Vert_{L^{\frac{2\sigma(4\sigma+4)}{4\sigma + 4 - 2n\sigma}}_t W^{s,2+\theta}_x (I_k)}^{2\sigma} \\&\quad\leq C_T \varepsilon^{2\sigma + 1}.
\end{align*}
For the last term in \eqref{xju}, using the Burkholder-Davis-Gundy inequality, applying Minkowski's inequality twice (since $\frac{4\sigma+4}{n\sigma} \geq 2$ and $2\sigma+2 \geq 2$, we can apply Minkowski's inequality) and using Strichartz estimates, we obtain
\begin{align*}
    &\left\Vert \mathbb{E}\left[\int_0^t S(t-s) (x_j + 2i(t-s) \partial_{x_j})\phi(x)g(s)dB(s) \right] \right\Vert_{L^{\frac{4\sigma + 4}{n\sigma}}_t L^{2\sigma + 2}_x} \\ 
    &\lesssim \left \Vert \mathbb{E}\left[\left( \int_0^t \vert S(t-s) (x_j + 2i(t-s) \partial_{x_j})\phi(x)g(s)\vert^2 ds \right)^{1/2} \right] \right \Vert_{L^{\frac{4\sigma + 4}{n\sigma}}_t L^{2\sigma + 2}_x} \\
    &\lesssim \left \Vert  S(t-s) (x_j + 2i(t-s) \partial_{x_j})\phi(x) (\mathbb{E}\left[g(s)^2 \right])^{1/2} \right \Vert_{L^2_s L^{\frac{4\sigma + 4}{n\sigma}}_t L^{2\sigma + 2}_x} \\
    &\leq C_T \left \Vert (x_j - 2is \partial_{x_j})\phi(x) (\mathbb{E}\left[g(s)^2 \right])^{1/2} \right \Vert_{L^2_s L^2_x} + C_T \left \Vert \partial_{x_j}\phi(x) (\mathbb{E}\left[g(s)^2 \right])^{1/2} \right \Vert_{L^2_s L^2_x} \\
    &\leq C(T, g,\phi).
\end{align*}
Thus, combining the above estimates leads to
\begin{equation*}
    \Vert x_j u(t) \Vert_{L^{\frac{4\sigma + 4}{n\sigma}}_t L^{2\sigma + 2}_x (I_k)} \leq C(T, g, \phi) (\Vert u_0 \Vert_{\Sigma}+1) + C_T \varepsilon^{2\sigma} \Vert x_j u(t) \Vert_{L^{\frac{4\sigma + 4}{n\sigma}}_t L^{2\sigma + 2}_x (I_k)}+C_{T}\varepsilon^{2\sigma+1}.
\end{equation*}
As a result, taking sufficiently small $\varepsilon$ and summing over $k$, we see that $\Vert x_j u(t) \Vert_{L^\frac{4\sigma+4}{n\sigma}_t L^{2\sigma+2}_x([0,T])}<\infty$ for any $0<T<\infty$.

\end{proof}

\section{Pseudo-conformal energy estimates}\label{section:pseudo-conformal_energy}
We define the pseudo-conformal energy functional as
\begin{equation}\label{eq:pseudo-conf_energy}
    E(u(s))=\int_{\mathbb{R}^{n}}\vert (x+2i(1+s)\nabla)u(s)\vert^{2}dx+\frac{4}{\sigma+1}(1+s)^{2}\Vert u(s)\Vert_{L^{2\sigma+2}_{x}}^{2\sigma+2}\, ,\quad s>0.
\end{equation}
This functional can also be expressed as
\begin{equation*}
    E(u(s))=V(u(s))-4(1+s)G(u(s))+8(1+s)^{2}H(u(s))
\end{equation*}
where $V$ is the virial (or variance) functional $\int_{\mathbb{R}^n} |x|^2 u \overline{u} dx$, $G$ is the momentum-like functional given by $\Im \int_{\mathbb{R}^n} \overline{u} x \cdot \nabla u  dx$, and $H$ is the Hamiltonian. We shall apply the Itô formula to \eqref{eq:pseudo-conf_energy} to understand the growth rate and the behavior of the differential of this quantity, which will play an important role in establishing the scattering in $L^{2}_{x}$ and $\Sigma$.
\begin{lemma}\label{T12lemma}
    Let $0<2\sigma<\frac{4}{n-2}$ if $n\geq 3$, $0<2\sigma<\infty$ if $n=1,2$. Assume also that $\langle\cdot\rangle^{2}g\in L^{2}_{t}([0,\infty)),\, a.s.$ and $u$ is a solution to \eqref{eq:additive_SNLS} with the initial data $u_{0}\in\Sigma$. Then we have
    \begin{equation*}
        E(u(T))=E(u(0))+\frac{4(2-n\sigma)}{\sigma+1}\int_{0}^{T}\int_{\mathbb{R}^{n}}(1+s)\vert u\vert^{2\sigma+2}dxds+\int_{0}^{T}T_{1}(s)ds+\int_{0}^{T}T_{2}(s)dB
    \end{equation*}
    a.s. for all $T>0$ where $T_{1},T_{2}$ are continuous $\{\mathcal{F}_{t}\}$-adapted  processes such that
    \begin{equation*}
         \int_{0}^{\infty}\vert T_{1}(s)\vert ds+\sup_{0\leq T<\infty}\left\vert\int_{0}^{T}T_{2}(s)dB(s)\right\vert<\infty,\quad a.s.
    \end{equation*}
\end{lemma}
\begin{proof}
We begin by noting that it suffices to find the differential of $V$, $G$, and $H$. The below calculations can be justified by approximating the solution $u$ via smooth approximate identities and limiting arguments, see \cite{deBouard_H1, Zhang_PhD} for details. We will work with the corresponding densities and let us start with the virial functional.
\begin{align*}
    dV&=\partial_{u}Vdu+\partial_{\overline{u}}Vd\overline{u}+\partial_{u}\partial_{\overline{u}}Vdud\overline{u}\\
    &=2\vert x\vert^{2}\Im(\overline{u}\Delta u)dt+2\vert x\vert^{2}\Im(\overline{u}\phi)dB+\vert x\vert^{2}\vert\phi\vert^{2}g^{2}dt.
\end{align*}
Taking the space-time integral of $dV$ and applying integration by parts in the space integral to the Laplacian part, we get
\begin{equation*}
  V(u(T))=V(u(0))+4\int_{0}^{T}G(u(s))ds+2\Im\int_{0}^{T}\int_{\mathbb{R}^{n}}\vert x\vert^{2}\overline{u}\phi gdB+\int_{0}^{T}\int_{\mathbb{R}^{n}}\vert x\vert^{2}\vert\phi\vert^{2}g^{2}(s)ds.
\end{equation*}
We now apply Itô formula to the momentum-like operator $G(u(t))=\Im\int_{\mathbb{R}^{n}}u\overline{\nabla u}\cdot xdx$. 
\begin{align*}
    dG&=\partial_{\overline{u}}Gd\overline{u}+\partial_{\nabla u}Gd(\nabla u)+\partial_{\overline{u}}\partial_{\nabla u}G d\overline{u}d(\nabla u)\\
    &= -2\Im\left[i x\cdot\nabla u(\Delta \overline{u}-\vert u\vert^{2\sigma}\overline{u})dt+ix\cdot\nabla u\overline{\phi}gdB\right]+n\Im\left[-i\overline{u}(\Delta u-\vert u\vert^{2\sigma}u)dt-i\overline{u}\phi gdB\right]\\
    &\quad-\Im(\overline{\phi} g^2 x\cdot \nabla\phi dt)=:I+II+III.
\end{align*}
Let us consider $I$ first, and write
\begin{align*}
    I 
    &=-2\Re(x\cdot\nabla u\Delta\overline{u})dt+2\Re(x\cdot\vert u\vert^{2\sigma}\overline{u}\nabla u)dt-2\Re(x\cdot\nabla u\overline{\phi}g)dB=:A+B+C.
\end{align*}
For the term $A$, integration by parts yields
\begin{align*}
    \int_{\mathbb{R}^n} x\cdot\nabla u\Delta \overline{u} dx&=\int_{\mathbb{R}^n} \Delta(x\cdot\nabla u)\overline{u}dx=2\int_{\mathbb{R}^n} \overline{u}\Delta u dx +\sum_{j=1}^{n}\int_{\mathbb{R}^n} \overline{u}x_{j}\partial_{j}\Delta u dx\\&
    = (n-2)\int_{\mathbb{R}^n} \vert\nabla u\vert^{2}dx -\int_{\mathbb{R}^n} x\cdot\nabla\overline{u}\Delta u dx,
\end{align*}
which amounts to writing that
\begin{equation*}
    -2\Re \int_{\mathbb{R}^n} x\nabla u\Delta\overline{u} dx= (2-n)\int_{\mathbb{R}^n} \vert\nabla u\vert^{2} dx. 
\end{equation*}
In a similar spirit, regarding the term $B$, we get
\begin{equation*}
    2\Re\int_{\mathbb{R}^n} x\cdot \vert u\vert^{2\sigma}\overline{u}\nabla udx=\frac{1}{\sigma+1} \int_{\mathbb{R}^n} x\cdot\nabla(\vert u\vert^{2\sigma+2})dx=-\frac{n}{\sigma+1}\int_{\mathbb{R}^n} \vert u\vert^{2\sigma+2}dx.
\end{equation*}
For the term $C$, application of integration by parts as above yields
\begin{equation*}
    -2\Re\int_{\mathbb{R}^n} x\cdot\nabla u\overline{\phi}g dx=2n\Re\int_{\mathbb{R}^n} u\overline{\phi}g dx + 2\Re\int_{\mathbb{R}^n} ux\cdot\nabla\overline{\phi}g dx.
\end{equation*}
Consequently, we have
\begin{equation*}
    \int_{\mathbb{R}^n} I dx = (2-n)\int_{\mathbb{R}^n} \vert\nabla u\vert^{2} dxdt -\frac{n}{\sigma+1}\int_{\mathbb{R}^n} \vert u\vert^{2\sigma+2}dx+2n\Re\int_{\mathbb{R}^n} u\overline{\phi}g dxdB+2\Re\int_{\mathbb{R}^n} ux\cdot\nabla\overline{\phi}g dxdB.
\end{equation*}
Handling the second term by the integration by parts again we obtain
\begin{align*}
    II 
    &=n\vert\nabla u\vert^{2}dt+n\vert u\vert^{2\sigma+2}dt -n\Re(\overline{u}\phi) g dB.
\end{align*}
Taking the space-time integrals and combining all of the terms above, we finally get that
\begin{align*}
G(u(T))&=G(u(0))+2\int_{0}^{T}\int_{\mathbb{R}^n}\vert\nabla u\vert^{2} dxdt - n\left(\frac{1}{\sigma+1}-1\right)\int_{0}^{T}\int_{\mathbb{R}^n} \vert u\vert^{2\sigma+2}dxdt + n\Re\int_{0}^{T}\int_{\mathbb{R}^n} u\overline{\phi}g dxdB\\&\quad
    +2\Re\int_{0}^{T}\int_{\mathbb{R}^n} ux\cdot\nabla\overline{\phi}g dxdB-\Im\int_{0}^{T}\int_{\mathbb{R}^n} x\overline{\phi}\cdot\nabla\phi g^{2}dxdt.
\end{align*}
Using the definition of the Hamiltonian, equivalently we write 
\begin{equation*}
\begin{aligned}
G(u(T))&=G(u(0))+4\int_{0}^{T}H(u(s))ds-\frac{2-n\sigma}{\sigma+1}\int_{0}^{T}\int_{\mathbb{R}^n}\vert u\vert^{2\sigma+2}dxds+n\Re\int_{0}^{T}\int_{\mathbb{R}^n} u\overline{\phi}g dxdB\\&\quad
    +2\Re\int_{0}^{T}\int_{\mathbb{R}^n} ux\cdot\nabla\overline{\phi}g dxdB-\Im\int_{0}^{T}\int_{\mathbb{R}^n} \overline{\phi} x\cdot\nabla\phi g^{2}dxds.
\end{aligned}
\end{equation*}
Now, we are ready to apply the Itô formula to the pseudo-conformal energy \eqref{eq:pseudo-conf_energy}:
\begin{align*}
    dE&=d\left(V(u(s))-4(1+s)G(u(s))+8(1+s)^{2}H(u(s))\right)\\
    &=dV-4G(u(s))ds - 4(1+s)dG+16(1+s)H(u(s))ds+8(1+s)^{2}dH.
\end{align*}
Taking the space-time integral, we obtain
\begin{align*}
    &E(u(T))\\&=E(u(0))+4\int_{0}^{T}G(u(s))ds+2\Im\int_{0}^{T}\int_{\mathbb{R}^{n}}\vert x\vert^{2}\overline{u}\phi gdxdB+\int_{0}^{T}\int_{\mathbb{R}^{n}}\vert x\vert^{2}\vert\phi\vert^{2}g^{2}(s)dxds\\
    &\quad-4\int_{0}^{T}G(u(s))ds-16\int_{0}^{T}(1+s)H(u(s))ds-\frac{4(2-n\sigma)}{\sigma+1}\int_{0}^{T}\int(1+s)\vert u\vert^{2\sigma+2}dxds \\
    &\quad +4n\Re\int_{0}^{T}\int (1+s)u\overline{\phi}g dxdB    +8\Re\int_{0}^{T}\int (1+s)ux\cdot\nabla\overline{\phi}g dxdB-4\Im\int_{0}^{T}\int (1+s)\overline{\phi}x\cdot\nabla\phi g^{2}dxds\\
    &\quad+16\int_{0}^{T}(1+s)H(u(s))ds-8\Im\int_{0}^{T}\int(1+s)^{2}(\Delta\overline{u}-\vert u\vert^{2\sigma}\overline{u})\phi g dxdB \\
    &\quad\left.+4\int_{0}^{T}\int(1+s)^{2} \vert\nabla\phi\vert^{2}g^{2}dxds  +4\int_{0}^{T}\int(1+s)^{2}\vert u\vert^{2\sigma}\vert\phi g\vert^{2}dxds\right. \\
    &\quad+2\sigma\int_{0}^{T}\int(1+s)^{2}\vert u\vert^{2\sigma-2}g^{2}(\Im(u\overline{\phi}))^{2}dxds\\
    &=E(u(0))+\frac{4(2-n\sigma)}{\sigma+1}\int_{0}^{T}\int(1+s)\vert u\vert^{2\sigma+2}dxds+\int_{0}^{T}T_{1}(s)ds+\int_{0}^{T}T_{2}(s)dB,
\end{align*}
where the term $T_{1}$ consists of the remaining terms with differential $ds$, while $T_{2}$ includes the terms with differential $dB$. Our aim now is to show that $T_{1}$, $T_{2}$ are continuous processes adapted to the filtration, and the last two integrals in the above expression are bounded almost surely for any $T>0$, i.e.,
\begin{equation*}
    \int_{0}^{\infty}\vert T_{1}(s)\vert ds+\sup_{0\leq T<\infty}\left\vert\int_{0}^{T}T_{2}(s)dB(s)\right\vert<\infty.
\end{equation*}
Now, we shall consider the integral involving $T_{1}$. Using Hölder inequality in space and time, we get
\begin{align*}
    \int_{0}^{\infty}\vert T_{1}(s)\vert ds&=\int_{0}^{\infty}\left\vert \int_{\mathbb{R}^{n}}\vert x\vert^{2}\vert\phi\vert^{2}g^{2}(s)dx-4\Im\int_{\mathbb{R}^n} (1+s)\overline{\phi} x\cdot\nabla\phi g^{2}dx+4\int_{\mathbb{R}^n}(1+s)^{2} \vert\nabla\phi\vert^{2}g^{2}dx \right.\\
    &\left.\quad +4\int_{\mathbb{R}^n}(1+s)^{2}\vert u\vert^{2\sigma}\vert\phi g\vert^{2}dx+2\sigma\int_{\mathbb{R}^n}(1+s)^{2}\vert u\vert^{2\sigma-2}g^{2}(\Im(u\overline{\phi}))^{2}dx\right\vert ds\\
    &\lesssim \int_{0}^{\infty}g^{2}(s)\int_{\mathbb{R}^n}\vert x\vert^{2}\vert\phi\vert^{2}dxds+\int_{0}^{\infty}(1+s)g^{2}(s)\int_{\mathbb{R}^n}\vert x\cdot\phi\nabla\phi\vert dxds\\
    &\quad+\int_{0}^{\infty}(1+s)^{2}g^{2}(s) \int_{\mathbb{R}^n}\vert\nabla\phi\vert^{2}dxds+\int_{0}^{\infty}(1+s)^{2}g^{2}(s)\int_{\mathbb{R}^n}\vert u\vert^{2\sigma}\vert\phi\vert^{2}dxds\\
    &\lesssim \Vert g\Vert_{L^{2}_{s}}^{2}\Vert\,\vert\cdot\vert\phi\Vert_{L^{2}_{x}}^{2}+\Vert\langle\cdot\rangle^{\frac{1}{2}} g\Vert_{L^{2}_{s}}^{2}\Vert\,\vert\cdot\vert\phi\Vert_{L^{2}_{x}}\Vert\phi\Vert_{\Dot{H}^{1}_{x}}+\Vert\langle\cdot\rangle g\Vert_{L^{2}_{s}}^{2}\Vert\phi\Vert_{\Dot{H}^{1}_{x}}^{2}\\
    &\quad+\Vert\langle\cdot\rangle g\Vert_{L^{2}_{s}}^{2}\Vert u\Vert_{L^{\infty}_{s}L^{2\sigma+2}_{x}}^{2\sigma}\Vert \phi\Vert_{L^{2\sigma+2}_{x}}^{2}\\
    &\lesssim \Vert\langle\cdot\rangle g\Vert_{L^{2}_{s}}^{2}\Vert\phi\Vert_{\Sigma}\big[1+\sup_{0\leq s<\infty}H(u(s))\big].
\end{align*}
For the second integral, we utilize the Burkholder-Davis-Gundy inequality first, then application of integration by parts, Cauchy-Schwarz and Hölder inequalities leads to
\begin{align*}
    \mathbb{E}\left[\sup_{0\leq T<\infty}\left\vert\int_{0}^{T}T_{2}(s)dB(s)\right\vert \right] &\lesssim \mathbb{E}\left[\Bigg(\int_{0}^{\infty}\vert T_{2}(s)\vert^{2}ds\Bigg)^{1/2}\right]\\ &\lesssim\mathbb{E}\left[\Bigg(\int_{0}^{\infty}\left\vert\int_{\mathbb{R}^{n}}\vert x\vert^{2}\overline{u}\phi gdx+\int (1+s)u\overline{\phi}g dx+\int (1+s)ux\cdot\nabla\overline{\phi}g dx \right. \right.\\
    &\left. \left.\quad \quad+\int(1+s)^{2}(\Delta\overline{u}-\vert u\vert^{2\sigma}\overline{u})\phi g dx\right\vert^{2}ds\Bigg)^{1/2}\right] \\&=\mathbb{E}\left[\Bigg(\int_{0}^{\infty}\left\vert\int_{\mathbb{R}^{n}}\vert x\vert^{2}\overline{u}\phi gdx+\int (1+s)u\overline{\phi}g dx+\int (1+s)ux\cdot\nabla\overline{\phi}g dx \right. \right. \\
    &\left. \left.\quad \quad+\int(1+s)^{2}(-\nabla\overline{u}\cdot\nabla\phi-\vert u\vert^{2\sigma}\overline{u}\phi) g dx\right\vert^{2}ds\Bigg)^{1/2}\right]\\
    &\lesssim\mathbb{E}\left[\Bigg(\int_{0}^{\infty}\Vert u(s)\Vert_{L^{2}_{x}}^{2}\Vert\,\vert\cdot\vert^{2}\phi\Vert_{L^{2}_{x}}^{2}g^{2}(s)ds+\int_{0}^{\infty}(1+s)^{2}g^{2}(s)\Vert u(s)\Vert_{L^{2}_{x}}^{2}\Vert\phi\Vert_{L^{2}_{x}}^{2} ds\right.\\
    &\quad\quad+\int_{0}^{\infty}(1+s)^{2}g^{2}(s)\Vert u(s)\Vert_{L^{2}_{x}}^{2}ds\\
    &\left. \quad\quad+\int_{0}^{\infty}(1+s)^{4}g^{2}(s)\big[\Vert u(s)\Vert_{\Dot{H}^{1}_{x}}^{2}\Vert \phi\Vert_{\Dot{H}^{1}_{x}}^{2}+\Vert u(s)\Vert_{L^{2\sigma+2}_{x}}^{4\sigma+2}\Vert\phi\Vert_{L^{2\sigma+2}_{x}}^{2}\big]ds\Bigg)^{1/2}\right]\\    
    &\lesssim C(\phi)\mathbb{E}\left[\Vert\langle\cdot\rangle^2 g\Vert_{L^{2}_{s}}\sup_{0\leq s<\infty}[1 + M(u(s))+H(u(s))]\right].
\end{align*}
Therefore, collecting the above estimates for $T_1$ and $T_2$, we obtain
\begin{align*}
\mathbb{E}&\left[\int_{0}^{\infty}\vert T_{1}(s)\vert ds+\sup_{0\leq T<\infty}\left\vert\int_{0}^{T}T_{2}(s)dB(s)\right\vert\right]\\&\lesssim_{\phi}\mathbb{E}\left[\big(\Vert\langle\cdot\rangle g\Vert_{L^{2}_{s}}^{2} +\Vert\langle\cdot\rangle^2 g\Vert_{L^{2}_{s}}\big)\sup_{0\leq s<\infty}[1+M(u(s))+H(u(s))]\right]<\infty
\end{align*}
since all of the moments of the mass and the Hamiltonian are bounded for all time. Also, continuity of the processes $T_{1}$ and $T_{2}$ follows from the continuity of the solution $u$.
\end{proof}
Next we aim to investigate the growth rate of $E(u(t))$. Instead of direct application of Gronwall's inequality to $E(u(t))$ (because of the presence of the power $q < 1$ in Lemma \ref{gronwall}), we prefer to apply Itô's lemma to the power $E^m(u(t))$ for some integer $m \geq 2$, and set $q = 1/m$ in Lemma \ref{gronwall} to obtain the following growth rate for $E(u(t))$:
\begin{proposition}[Growth of the pseudo-conformal energy]
    Let $E$ be the pseudo-conformal energy defined as in \eqref{eq:pseudo-conf_energy}, $C(\phi)$ be some constant depending on $\Vert\phi\Vert_{\Sigma}$, and $u$ be the solution to \eqref{eq:additive_SNLS} with the initial data $u_{0}\in\Sigma$. Then for any $\tau>0$, if $2\sigma<\frac{4}{n}$, we have
    \begin{equation}\label{eq:growth_of_psedo_energy_mass_sub-critical}
        \mathbb{E}\left[\sup_{0\leq t\leq\tau}(C(\phi)+E(u(t)))\right]\leq C(\Vert u_{0}\Vert_{\Sigma})(1+\tau)^{2-n\sigma}.
    \end{equation}
    Moreover, if $\frac{4}{n}\leq 2\sigma<\frac{4}{n-2}$ when $n=3$, or $\frac{4}{n}\leq 2\sigma<\infty$ when $n=1,2$, then    \begin{equation}\label{eq:growth_of_psedo_energy_mass_super-critical}
        \mathbb{E}\left[\sup_{0\leq t\leq\tau}(C(\phi)+E(u(t)))\right]\leq C(\Vert u_{0}\Vert_{\Sigma})<\infty.
    \end{equation}
\end{proposition}
\begin{proof}
    Let us start with the mass sub-critical case by recalling that
\begin{equation*}
    d(C(\phi)+E)=dE = \frac{4 (2-\sigma n)}{\sigma + 1} (1+s)\Vert u(s) \Vert_{L^{2\sigma+2}_x}^{2\sigma + 2} ds + T_1 ds + T_2 dB(s).
\end{equation*}
Let $m \geq 2$ be an integer. Implementing the Itô formula to $(C(\phi)+E)^{m}$, we have
\begin{align*}
    d(C(\phi)+E)^m &= m (C(\phi)+E)^{m-1}(1+s) \frac{4(2-\sigma n)}{\sigma + 1} \Vert u(s) \Vert_{L^{2\sigma+2}_x}^{2\sigma + 2} ds + m (C(\phi)+E)^{m-1} T_2 dB\\&\quad + m (C(\phi)+E)^{m-1} T_1 ds + \frac{1}{2} m (m-1) (C(\phi)+E)^{m-2} T_2^2 ds,
\end{align*}
which can be expressed in the integral form by the following
\begin{align*}
    (C(\phi)+E(u(t)))^m &= (C(\phi)+E(u(0)))^m + \frac{4m(2-\sigma n)}{\sigma + 1} \int_0^t (C(\phi)+E(u(s)))^{m-1}(1+s) \Vert u(s) \Vert_{L^{2\sigma + 2}_x}^{2\sigma + 2} ds\\
    &\quad+ m \int_0^t (C(\phi)+E(u(s)))^{m-1}T_1(s)ds + \frac{1}{2} m(m-1) \int_0^t (C(\phi)+E(u(s)))^{m-2} T_2^2(s) ds\\
    &\quad+ m \int_0^t (C(\phi)+E(u(s)))^{m-1} T_2(s) dB(s).
\end{align*}
Note that 
\begin{align*}
    (C(\phi)+E(u(t)))^m &\leq (C(\phi)+E(u(0)))^m + m(2-\sigma n) \int_0^t (1+s)^{-1} (C(\phi)+E(u(s)))^m ds\\
    &\quad+ m\int_0^t (C(\phi)+E(u(s)))^{m-1}T_1(s)ds + \frac{1}{2} m(m-1) \int_0^t (C(\phi)+E(u(s)))^{m-2} T_2^2(s)ds\\
    &\quad+ m \int_0^t (C(\phi)+E(u(s)))^{m-1} T_2(s)dB(s).
\end{align*}
Note also that for any $t>0$, the stochastic integral term is a continuous martingale. We shall handle the terms $T_1$ and $T_2$ next. From Lemma \ref{T12lemma},
\begin{align*}
    T_1(s) &= \int_{\mathbb{R}^{n}}\vert x\vert^{2}\vert\phi\vert^{2}g^{2}(s)dx-4\Im\int_{\mathbb{R}^{n}} (1+s)\overline{\phi} x\cdot\nabla\phi g^{2}dx+4\int_{\mathbb{R}^{n}}(1+s)^{2} \vert\nabla\phi\vert^{2}g^{2}dx\\
    &\quad+4\int_{\mathbb{R}^{n}}(1+s)^{2}\vert u\vert^{2\sigma}\vert\phi g\vert^{2}dx+2\sigma\int_{\mathbb{R}^{n}}(1+s)^{2}\vert u\vert^{2\sigma-2}g^{2}(\Im(u\overline{\phi}))^{2}dx \\&\quad=: I + II + III + IV + V.
\end{align*}
By the constraints imposed on $\phi$, we see that $I+|II|+III \leq C(\phi)(1+s)^2 g^2(s)$. For $IV$, Hölder's inequality and Young's inequality yields
\begin{align*}
    IV &\leq 4(1+s)^2 g^2(s) \Vert u(s) \Vert_{L^{2\sigma + 2}_x}^{2\sigma} \Vert \phi \Vert_{L^{2\sigma+2}_x}^2 \leq (1+s)^{2}g^2(s)\left(C(\phi) + \frac{1}{2\sigma + 2} \Vert u(s) \Vert_{L^{2\sigma + 2}_x}^{2\sigma + 2}\right)\\& \leq (1+s)^2 g^2(s)\left( C(\phi)+E(u(s))\right).
\end{align*}
Finally, for $V$, again using Hölder's inequality and Young's inequality yields
\begin{equation*}
    V \leq (1+s)^2 g^2(s)\left( C(\phi)+E(u(s))\right).
\end{equation*}
From these observations, we conclude that $T_1(s) \lesssim (C(\phi)+E(u(s)))(1+s)^2 g^2(s)$. Recalling from Lemma \ref{T12lemma}, we write
\begin{align*}
    T_2(s) &= \int_{\mathbb{R}^{n}}\vert x\vert^{2}\overline{u}\phi gdx+\int (1+s)u\overline{\phi}g dx+\int (1+s)ux\cdot\nabla\overline{\phi}g dx \\
    &\quad+\int(1+s)^{2}\Delta\overline{u}\phi g dx -\int(1+s)^{2}\vert u\vert^{2\sigma}\overline{u}\phi g dx \\&=: I + II + III + IV + V.
\end{align*}
For $IV$, we integrate by parts and use Young's inequality to obtain
\begin{equation*}
    \left\vert \int(1+s)^{2}\Delta\overline{u}\phi g dx \right\vert = \left\vert \int(1+s)^{2}\nabla \overline{u} \cdot \nabla \phi g dx \right\vert \leq (C(\phi) + E(u(s)))(1+s)^{2} \vert g(s) \vert.
\end{equation*}
Similar arguments applied to the remaining terms using Hölder's inequality and Young's inequality also shows that $T_2(s) \lesssim (C(\phi)+E(u(s)))(1+s)^2 \vert g(s)\vert$. Combining the above estimates, we obtain
\begin{align*}
    (C(\phi)+E(u(t)))^m &\leq (C(\phi)+E(u(0))^m + m(2-\sigma n) \int_0^t (1+s)^{-1} (C(\phi)+E(u(s)))^{m}ds\\
    &\quad+ Cm\int_0^t (C(\phi)+E(u(s)))^{m}(1+s)^2 g^2(s)ds\\
    &\quad+C\frac{m(m-1)}{2} \int_0^t (C(\phi)+E(u(s)))^{m}(1+s)^4 g^2(s)ds\\
    &\quad+ Cm \int_0^t (C(\phi)+E(u(s)))^{m-1} T_2(s)dB(s).
\end{align*}
We can rewrite the above inequality as
\begin{align*}
     &(C(\phi)+E(u(t)))^m \\&\leq C(m) \Vert u_0 \Vert_{\Sigma}^{2m(\sigma + 1)} + \int_0^t (C(\phi)+E(u(s)))^{m} \left[ m(2-\sigma n) (1+s)^{-1} + \frac{m(m+1)}{2}(1+s)^4 g^2(s) \right] ds + M_t
\end{align*}
where $M_t$ stands for the stochastic integral we remarked earlier as a continuous martingale. We are now ready to apply the stochastic Gronwall's inequality. Let $0 < q < p <1$. Then we have
\begin{align*}
    &\mathbb{E}\left[\sup\limits_{0 \leq t \leq \tau} (C(\phi)+E(u(t)))^{mq} \right]^{1/q}\\& \leq C(m) \left( \frac{p}{p-q} \right)^{1/q} \Vert u_0 \Vert_{\Sigma}^{2m(\sigma + 1)}  \left( \mathbb{E}\left[\exp\left( \frac{p}{1-p} \int_0^\tau \frac{m(2-\sigma n)}{(1+s)^{-1}}+\frac{m(m+1)}{2}(1+s)^4 g^2(s) ds \right) \right] \right)^{\frac{1-p}{p}} \\
    &\leq C(m) \Vert u_0 \Vert_{\Sigma}^{2m(\sigma + 1)} \left( \frac{p}{p-q} \right)^{1/q} (1+\tau)^{m(2-\sigma n)}.
\end{align*}
Now take $q = \frac{1}{m}$. Taking the $q$th power yields
\begin{equation*}
    \mathbb{E}\left[\sup\limits_{0 \leq t \leq \tau} (C(\phi)+E(u(t))) \right] \leq C(m,p) \Vert u_0 \Vert_{\Sigma}^{2\sigma + 2} (1+\tau)^{2 - n \sigma}.
\end{equation*}
The remaining case, mass-critical or super-critical nonlinearities, can be demonstrated similarly. The only difference is that the coefficient $\frac{4m(2-n\sigma)}{\sigma+1}$ of the integral involving $\Vert u(s)\Vert_{L^{2\sigma+2}_{x}}^{2\sigma+2}$ is nonpositive in this case. Thus, we get
\begin{align*}
    &(C(\phi)+E(u(t)))^m \\&\leq (C(\phi)+E(u(0)))^m + m\int_0^t (C(\phi)+E(u(s)))^{m-1}T_1(s)ds\\
    &\quad+ \frac{1}{2} m(m-1) \int_0^t (C(\phi)+E(u(s)))^{m-2} T_2^2(s)ds + m \int_0^t (C(\phi)+E(u(s)))^{m-1} T_2(s)dB(s),
\end{align*}
then the argument as in the mass sub-critical case shows that
\begin{equation*}
    \mathbb{E}\left[\sup\limits_{0 \leq t \leq \tau} (C(\phi)+E(u(t))) \right] \leq C(m,p) \Vert u_0 \Vert_{\Sigma}^{2\sigma + 2} <\infty.
\end{equation*}
\end{proof}
\begin{Remark}
    As we are interested in the growth rate of $\sup\limits_{0 \leq t \leq \tau} E(u(t))$ for large $\tau$, for any $r\geq 1$, choosing $m>r$ and setting $q=\frac{r}{m}$ in the above proof, the following inequality can be shown to hold for any $\tau>0$ and $2\sigma<\frac{4}{n}$:
    \begin{equation*}
        \mathbb{E}\left[\sup_{0\leq t\leq\tau}E^{r}(u(t))\right]\leq C(\Vert u_{0}\Vert_{\Sigma},\phi)(1+\tau)^{r(2-n\sigma)}.
    \end{equation*}
    Also, for $\frac{4}{n}\leq 2\sigma<\frac{4}{n-2}$ if $n\geq 3$, or $\frac{4}{n}\leq 2\sigma<\infty$ if $n=1,2$, we have
    \begin{equation*}
        \mathbb{E}\left[\sup_{0\leq t\leq\tau}E^{r}(u(t))\right]\leq C(\Vert u_{0}\Vert_{\Sigma},\phi)<\infty.
    \end{equation*}
    
\end{Remark}
We shall now examine the behavior of the solution $u_{*}$ to the random equation \eqref{eq:additive_random_eqn} and $E(u(s))$ under the pseudo-conformal transformation 
\begin{equation*}
    \widetilde{u}(t,x)=(1-t)^{-\frac{n}{2}}u\left(\frac{t}{1-t},\frac{x}{1-t}\right)e^{i\frac{\vert x\vert^{2}}{4(1-t)}}.
\end{equation*}
Applying this transformation to the random equation \eqref{eq:additive_random_eqn}, we see that $\widetilde{u}_{*}$ satisfies:
\begin{equation}\label{eq:random_eqn_after_pseudo-conf_tr}
    i\partial_{t}\widetilde{u}_{*}-\Delta \widetilde{u}_{*}+(1-t)^{\sigma n-2}\vert \widetilde{u}_{*}+\widetilde{z}_{*}\vert^{2\sigma}(\widetilde{u}_{*}+\widetilde{z}_{*})=0,
\end{equation}
or
\begin{equation*}
    (1-t)^{2-\sigma n}(i\partial_{t}\widetilde{u}_{*}-\Delta \widetilde{u}_{*})+\vert \widetilde{u}_{*}+\widetilde{z}_{*}\vert^{2\sigma}(\widetilde{u}_{*}+\widetilde{z}_{*})=0.
\end{equation*}
The global-in-time Strichartz estimates for the solution $\widetilde{u}_{*}$ to \eqref{eq:random_eqn_after_pseudo-conf_tr} are established in the next theorem. The proof relies on exploiting the growth rate of the pseudo-conformal energy.
\begin{theorem}\label{thm:global_Strichartz}
    Let $\sigma(n)<2\sigma<\frac{4}{n-2}$ if $n\geq 3$, and $\sigma(n)<2\sigma<\frac{4}{n}$ or $\frac{4}{n} < 2\sigma < \infty$ if $n=1,2$, where $\sigma(n)$ is as in \eqref{eq:Strauss_exponent}. Assume that the time decay condition on $g$ given in Theorem \ref{thm:result_Sigma_scattering} also holds. Then, for any Strichartz admissible pair $(p,q)$, we have
    \begin{equation}\label{eq:global_Strichartz_for_random_eqn}
        \Vert \widetilde{u}_{*}\Vert_{L^{q}_{t}W^{1,p}_{x}([0,1)\times\mathbb{R}^{n})}+\Vert\vert\cdot\vert \widetilde{u}_{*}\Vert_{L^{q}_{t}L^{p}_{x}([0,1)\times\mathbb{R}^{n})}<\infty,\quad a.s.
    \end{equation}
    Moreover, for $2\sigma = \frac{4}{n}$ when $n=1,2$, the estimate \eqref{eq:global_Strichartz_for_random_eqn} is also valid if we impose the small data assumption \eqref{eq:bound_of_data_for_induction} as well. 
\end{theorem}
\begin{proof}
The local-in-time Strichartz estimates for $\widetilde{u}_{*}$ on the time interval $[0,\widetilde{T})$, for any $0<\widetilde{T}<1$, can be directly established by adapting the proof of Lemma 2.1 from \cite{Herr_2019}. Alternatively, the proof of Proposition \ref{prop_Strichartz_estimate_for_sigma_space} may also be adjusted to obtain the local estimates. Note that in both approaches, the implicit constants in the estimates depend on $\widetilde{T}$ which may tend to $\infty$ as $\widetilde{T}\to 1^-$, due to the asymptotic behavior of the pseudo-conformal transformation of $u_{*}$, which is $\widetilde{u}_{*}(t)$, observed at $t=1$. Thus, for $\widetilde{T} < 1$, we will focus on the time interval $[\widetilde{T},1)$, where we refine the local estimates to derive global space-time estimates. We will prove the result separately for the mass sub-critical case and the remaining range of nonlinearities.

\noindent
\textbf{Case 1:} Let $\sigma(n)<2\sigma<\frac{4}{n}$ for $n\geq 1$. \\ Define the identity
\begin{equation*}
    \widetilde{E}_{1}(\widetilde{u}(t)):=4\Vert\nabla\widetilde{u}(t)\Vert_{L^{2}_{x}}^{2}+(1-t)^{\sigma n-2}\frac{4}{\sigma+1}\Vert\widetilde{u}(t)\Vert_{L^{2\sigma+2}}^{2\sigma+2},\quad 0\leq t<1.
\end{equation*}
Note that
\begin{equation*}
     \widetilde{E}_{1}(\widetilde{u}(t))=E(u(s))\quad\text{for}\,\, s=\frac{t}{1-t}.
\end{equation*}
Thus, writing $T=\frac{\widetilde{T}}{1-\widetilde{T}}$ gives that
\begin{align*}
    \widetilde{E}_{1}(\widetilde{u}(\widetilde{T}))-\widetilde{E}_{1}(\widetilde{u}(0))&=E(u(T))-E(u(0))\\
    &=\frac{4(2-n\sigma)}{\sigma+1}\int_{0}^{T}(1+s)\Vert u(s)\Vert_{L^{2\sigma+2}_{x}}^{2\sigma+2}ds+\int_{0}^{T}T_{1}(s)ds+\int_{0}^{T}T_{2}(s)dB\\
    &=\frac{4(2-n\sigma)}{\sigma+1}\int_{0}^{\widetilde{T}}(1-t)^{\sigma n-3}\Vert\widetilde{u}(t)\Vert_{L^{2\sigma+2}_{x}}^{2\sigma+2}dt+\int_{0}^{T}T_{1}(s)ds+\int_{0}^{T}T_{2}(s)dB,
\end{align*}
where in the last line, we make use of the following identity:
\begin{equation*}
    \Vert u(s)\Vert_{L^{2\sigma+2}_{x}}^{2\sigma+2}=(1-t)^{\sigma n}\Vert\widetilde{u}(t)\Vert_{L^{2\sigma+2}_{x}}^{2\sigma+2}.
\end{equation*}
We shall define another identity as follows
\begin{equation*}
    \widetilde{E}_{2}(\widetilde{u}(t)):=(1-t)^{2-\sigma n}\widetilde{E}_{1}(\widetilde{u}(t))=4(1-t)^{2-\sigma n}\Vert\nabla\widetilde{u}(t)\Vert_{L^{2}_{x}}^{2}+\frac{4}{\sigma+1}\Vert\widetilde{u}(t)\Vert_{L^{2\sigma+2}}^{2\sigma+2}.
\end{equation*}
Then, the following observation 
\begin{equation*}
    d(\widetilde{E}_{2})=(\sigma n-2)(1-t)^{1-\sigma n}\widetilde{E}_{1}dt+(1-t)^{2-\sigma n}d\widetilde{E}_{1}
\end{equation*}
implies that
\begin{align*}
    \widetilde{E}_{2}(\widetilde{u}(\widetilde{T}))&=\widetilde{E}_{2}(\widetilde{u}(0))+(\sigma n-2)\int_{0}^{\widetilde{T}}(1-t)^{1-\sigma n}\widetilde{E}_{1}(\widetilde{u}(t))dt+\frac{4(2-n\sigma)}{\sigma+1}\int_{0}^{\widetilde{T}}(1-t)^{-1}\Vert\widetilde{u}(t)\Vert_{L^{2\sigma+2}_{x}}^{2\sigma+2}dt\\
    &\quad+\int_{0}^{T}(1+s)^{\sigma n-2}T_{1}(s)ds+\int_{0}^{T}(1+s)^{\sigma n-2}T_{2}(s)dB\\
    &=\widetilde{E}_{2}(\widetilde{u}(0))+(4\sigma n-8)\int_{0}^{\widetilde{T}}(1-t)^{1-\sigma n}\Vert\nabla\widetilde{u}(t)\Vert_{L^{2}_{x}}^{2}dt\\&\quad+\int_{0}^{T}(1+s)^{\sigma n-2}T_{1}(s)ds+\int_{0}^{T}(1+s)^{\sigma n-2}T_{2}(s)dB\\
    &\lesssim\Vert u_{0}\Vert_{\Sigma}^2 +\sup_{0\leq s<\infty}\left(\int_{0}^{T}\vert T_{1}(s)\vert ds+\left\vert\int_{0}^{T}T_{2}(s)dB\right\vert\right)<\infty,\quad a.s.
\end{align*}
In the last step, we use the fact that $(4\sigma n-8)<0$ and the integral with coefficient $(4\sigma n-8)$ is nonnegative. Since $\Vert\widetilde{u}(t)\Vert_{L^{2\sigma+2}_{x}}\lesssim\widetilde{E}_{2}(\widetilde{u}(t))$, we have 
\begin{equation*}
    \sup_{0\leq t<1}\Vert\widetilde{u}(t)\Vert_{L^{2\sigma+2}_{x}}<\infty,\quad a.s.
\end{equation*}
Recall that $\widetilde{u}=\widetilde{u}_{*}+\widetilde{z}_{*}$ where $\widetilde{z}_{*}$ is the pseudo-conformal transformation of the tail of the stochastic convolution. Then taking $0<\widetilde{T}<1$, changing the time variable and recalling \eqref{eq:almost_sure_time_decay_of_tail}, the following inequality holds almost surely:
\begin{align*}
    \sup_{\widetilde{T}\leq t<1}\Vert\widetilde{u}_{*}(t)\Vert_{L^{2\sigma+2}_{x}}&\leq\sup_{\widetilde{T}\leq t<1}\Vert\widetilde{u}(t)\Vert_{L^{2\sigma+2}_{x}}+\sup_{\widetilde{T}\leq t<1}\Vert\widetilde{z}_{*}(t)\Vert_{L^{2\sigma+2}_{x}}\\
    &\lesssim 1+\sup_{T\leq s<\infty}(1+s)^{\sigma n}\Vert z_{*}(s)\Vert_{L^{2\sigma+2}_{x}}\\
    &\lesssim 1+\sup_{T\leq s<\infty}(1+s)^{\sigma n}\langle s\rangle^{-\alpha+\frac{1}{2}},
\end{align*}
where $T=\frac{\widetilde{T}}{1-\widetilde{T}}$ and $\vert g(s)\vert\sim\langle s\rangle^{-\alpha}$. Note that  we have assumed $\langle s\rangle^{2}g=o(s^{-1/2})$ a.s., which is the case as long as we have $\alpha>\frac{5}{2}$. Therefore, to make the last term above finite, we need
\begin{equation*}
    -\alpha+n\sigma+\frac{1}{2}\leq 0.
\end{equation*}
Since $\sigma<\frac{2}{n}$, we have
\begin{equation*}
    -\alpha+n\sigma+\frac{1}{2}<-\alpha+\frac{5}{2}<0
\end{equation*}
as claimed, which guarantees that
\begin{equation}\label{eq:bound_potential_energy}
   \sup_{\widetilde{T}\leq t<1}\Vert\widetilde{u}_{*}(t)\Vert_{L^{2\sigma+2}_{x}}<\infty, \quad a.s.
\end{equation}
Next we are to show the following Strichartz estimates for $\widetilde{T}$ close to $1$:
\begin{equation}\label{eq:global_in_time_Strichartz}
    \Vert \widetilde{u}_{*}(t)\Vert_{L^{q}_{t}W^{1,p}_{x}([\widetilde{T},1)\times\mathbb{R}^{n})}+\Vert \vert\cdot\vert \widetilde{u}_{*}(t)\Vert_{L^{q}_{t}L^{p}_{x}([\widetilde{T},1)\times\mathbb{R}^{n})}\leq C(\Vert\widetilde{u}_{*}(\widetilde{T})\Vert_{\Sigma})<\infty.
\end{equation}
For this purpose, recall the Duhamel representation of $\widetilde{u}_{*}(t)$ for $t\geq\widetilde{T}$:
\begin{equation}\label{eq:duhamel_for_pseudo}
    \widetilde{u}_{*}(t)=S(t)\widetilde{u}_{*}(\widetilde{T})+i\int_{\widetilde{T}}^{t}S(t-t')(1-t)^{\sigma n-2}\vert\widetilde{u}_{*}+\widetilde{z}_{*}\vert^{2\sigma}(\widetilde{u}_{*}+\widetilde{z}_{*})(t')dt'.
\end{equation}
In the rest of the proof, the spacetime norms $L^{q}_{t}W^{s,p}_{x}$, for $p,q\geq 1$ and $s\in\mathbb{R}$, will be taken over the set $[\widetilde{T},1)\times\mathbb{R}^{n}$. Below we will take advantage of the Strichartz admisssible pair $(p_{1},q_{1})=\left(2\sigma+2,\frac{4\sigma+4}{n\sigma}\right)$. Thus, for any Strichartz admissible pair $(p,q)$, applying Lemma \ref{stichartzlemma} to the Duhamel formula \eqref{eq:duhamel_for_pseudo} and using Leibniz rule, we have
\begin{equation}\label{LqW1pnorm}
\begin{aligned}
\Vert\widetilde{u}_{*}\Vert_{L^{q}_{t}W^{1,p}_{x}}&\lesssim\Vert \widetilde{u}_{*}(\widetilde{T})\Vert_{H^{1}_{x}}+\left\Vert(1-t)^{\sigma n-2}\langle\nabla\rangle\left(\vert\widetilde{u}_{*}+\widetilde{z}_{*}\vert^{2\sigma}(\widetilde{u}_{*}+\widetilde{z}_{*})\right)\right\Vert_{L^{q_{1}'}_{t}L^{p_{1}'}_{x}}\\
    &\lesssim\Vert \widetilde{u}_{*}(\widetilde{T})\Vert_{H^{1}_{x}}+\Vert(1-t)^{\sigma n-2}\,\widetilde{u}_{*}^{2\sigma}(\langle\nabla\rangle\widetilde{u}_{*})\Vert_{L^{q_{1}'}_{t}L^{p_{1}'}_{x}}+\Vert(1-t)^{\sigma n-2}\,\widetilde{u}_{*}^{2\sigma-1}\widetilde{z}_{*}(\langle\nabla\rangle\widetilde{u}_{*})\Vert_{L^{q_{1}'}_{t}L^{p_{1}'}_{x}}\\
    &\quad+\Vert(1-t)^{\sigma n-2}\,\widetilde{u}_{*}^{2\sigma}(\langle\nabla\rangle\widetilde{z}_{*})\Vert_{L^{q_{1}'}_{t}L^{p_{1}'}_{x}}+\Vert(1-t)^{\sigma n-2}\,\widetilde{z}_{*}^{2\sigma}(\langle\nabla\rangle\widetilde{u}_{*})\Vert_{L^{q_{1}'}_{t}L^{p_{1}'}_{x}}\\
    &\quad+\Vert(1-t)^{\sigma n-2}\,\widetilde{z}_{*}^{2\sigma-1}\widetilde{u}_{*}(\langle\nabla\rangle\widetilde{z}_{*})\Vert_{L^{q_{1}'}_{t}L^{p_{1}'}_{x}}+\Vert(1-t)^{\sigma n-2}\,\widetilde{z}_{*}^{2\sigma}(\langle\nabla\rangle\widetilde{z}_{*})\Vert_{L^{q_{1}'}_{t}L^{p_{1}'}_{x}}.
\end{aligned}
\end{equation}
 Recalling \eqref{eq:almost_sure_time_decay_of_tail} that $\widetilde{z}_{*}(t)$ almost surely obeys the Strichartz estimates, we have
\begin{align*}
    \Vert(1-t)^{\sigma n-2}\,\widetilde{u}_{*}^{2\sigma}(\langle\nabla\rangle\widetilde{u}_{*})\Vert_{L^{q_{1}'}_{t}L^{p_{1}'}_{x}}&\lesssim\Vert(1-t)^{\sigma n-2}\widetilde{u}_{*}^{2\sigma}\Vert_{L^{\frac{2\sigma+2}{2\sigma+2-n\sigma}}_{t}L^{\frac{2\sigma+2}{2\sigma}}_{x}}\Vert\langle\nabla\rangle\widetilde{u}_{*}\Vert_{L^{q_{1}}_{t}L^{p_{1}}_{x}}\\
    &=\Vert(1-t)^{\frac{\sigma n-2}{2\sigma}}\widetilde{u}_{*}\Vert_{L^{\frac{2\sigma(2\sigma+2)}{2\sigma+2-n\sigma}}_{t}L^{2\sigma+2}_{x}}^{2\sigma}\Vert\langle\nabla\rangle\widetilde{u}_{*}\Vert_{L^{q_{1}}_{t}L^{p_{1}}_{x}}\\
    &\leq \Vert(1-t)^{-1}\Vert_{L^{\frac{(2-n\sigma)(2\sigma+2)}{2\sigma+2-n\sigma}}_{t}}^{2-n\sigma}\sup_{\widetilde{T}\leq t<1}\Vert\widetilde{u}_{*}\Vert_{L^{2\sigma+2}_{x}}^{2\sigma}\Vert\langle\nabla\rangle\widetilde{u}_{*}\Vert_{L^{q_{1}}_{t}L^{p_{1}}_{x}}.
\end{align*}
The middle term in the last line is bounded due to \eqref{eq:bound_potential_energy}. Also, the first term is finite as long as
\begin{equation*}
    \frac{(2-n\sigma)(2\sigma+2)}{2\sigma+2-n\sigma}<1,
\end{equation*}
that is,
\begin{equation}\label{eq:derivation_Strauss_exp}
    \frac{(2-n\sigma)(2\sigma+2)}{2\sigma+2-n\sigma}-1=\frac{2+2\sigma-2n\sigma^{2}-n\sigma}{2\sigma+2-n\sigma}<0
\end{equation}
which holds provided that $2\sigma>\frac{2-n+\sqrt{n^2+12n+4}}{2n}=\sigma(n)$. This affirms the requirement of the Strauss exponent \eqref{eq:Strauss_exponent} as the lower bound for $2\sigma$ in Theorem \ref{thm:result_Sigma_scattering}. Note that $\sigma(n) > \frac{2}{n}$ for all $n\geq 1$. Now, together with the energy sub-criticality assumption, we need the following requirement
\begin{equation*}
    \sigma(n)<2\sigma<\frac{4}{n}.
\end{equation*}
With this assumption, the first term can be made arbitrarily small by choosing $\widetilde{T}$ sufficiently close to $1$. The remaining nonlinear terms can be controlled in a similar fashion. Thus, we have
\begin{equation}\label{eq:general_inequality_for_bootstrap}
    \Vert\widetilde{u}_{*}\Vert_{L^{q}_{t}W^{1,p}_{x}}\lesssim\Vert\widetilde{u}_{*}(\widetilde{T})\Vert_{H^{1}_{x}}+\varepsilon(\widetilde{T})(1+\Vert\widetilde{u}_{*}\Vert_{L^{q_{1}}_{t}W^{1,p_{1}}_{x}}),
\end{equation}
where $\varepsilon(\widetilde{T})$ depends on $\Vert \widetilde{z}_{*}\Vert_{L^{2\sigma +2}_x}$ and $\Vert(1-t)^{-1}\Vert_{L^{\frac{(2-n\sigma)(2\sigma+2)}{2\sigma+2-n\sigma}}_{t}([\widetilde{T},1))}^{2-n\sigma}$, and tends to $0$ as $\widetilde{T}\to 1^-$. Then setting $(p,q)=(p_{1},q_{1})$ in \eqref{eq:general_inequality_for_bootstrap}, we get that \begin{align}\label{epsilondep}
\Vert\widetilde{u}_{*}\Vert_{L^{q_{1}}_{t}W^{1,p_{1}}_{x}}\lesssim (1-\varepsilon(\widetilde{T}))^{-1}\big[\Vert\widetilde{u}_{*}(\widetilde{T})\Vert_{H^{1}_{x}}+\varepsilon(\widetilde{T})\big]. 
\end{align} 
Thus, substituting \eqref{epsilondep} into \eqref{eq:general_inequality_for_bootstrap} gives that for any admissible pair $(p,q)$, $\Vert\widetilde{u}_{*}\Vert_{L^{q}_{t}W^{1,p}_{x}}\lesssim C(\Vert\widetilde{u}_{*}(\widetilde{T})\Vert_{H^{1}_{x}})$ for sufficiently small $\varepsilon
(\widetilde{T})$. The remaining dual Strichartz norms involving $\widetilde{z}_{*}$ in \eqref{LqW1pnorm} can be treated in the same way.

Next, we will estimate the variance norm. Recalling the commutation relation $x_j S(t) = S(t) (x_j + 2it \partial_{x_j})$, we write
\begin{align*}
    &\Vert x_{j}\widetilde{u}_{*}\Vert_{L^{q}_{t}L^{p}_{x}}\\&\lesssim\Vert x_{j}\widetilde{u}_{*}(\widetilde{T})\Vert_{L^{2}_{x}}+\widetilde{T}\Vert \widetilde{u}_{*}(\widetilde{T})\Vert_{H^{1}_{x}}+\Vert(1-t)^{\sigma n-2}\,\widetilde{u}_{*}^{2\sigma}(x_{j}\widetilde{u}_{*})\Vert_{L^{q_{1}'}_{t}L^{p_{1}'}_{x}}\\
    &\quad+\Vert(1-t)^{\sigma n-2}\,\widetilde{u}_{*}^{2\sigma-1}\widetilde{z}_{*}(x_{j}\widetilde{u}_{*})\Vert_{L^{q_{1}'}_{t}L^{p_{1}'}_{x}}+\Vert(1-t)^{\sigma n-2}\,\widetilde{z}_{*}^{2\sigma}(x_{j}\widetilde{u}_{*})\Vert_{L^{q_{1}'}_{t}L^{p_{1}'}_{x}}\\
    &\quad+\Vert(1-t)^{\sigma n-2}\,\widetilde{z}_{*}^{2\sigma}(x_{j}\widetilde{z}_{*})\Vert_{L^{q_{1}'}_{t}L^{p_{1}'}_{x}}+\Vert (1-t)^{\sigma n-1}\vert \widetilde{u}_{*}+\widetilde{z}_{*}\vert^{2\sigma}(\widetilde{u}_{*}+\widetilde{z}_{*})\Vert_{L^{q_{1}'}_{t}W^{1,p_{1}'}_{x}} \\
    &\lesssim(1+\widetilde{T})\Vert\widetilde{u}_{*}(\widetilde{T})\Vert_{\Sigma}+\Vert(1-t)^{-1}\Vert_{L^{\frac{(2-n\sigma)(2\sigma+2)}{2\sigma+2-n\sigma}}_{t}}^{2-n\sigma}\sup_{\widetilde{T}\leq t<1}\Vert\widetilde{u}_{*}\Vert_{L^{2\sigma+2}_{x}}^{2\sigma}(\Vert x_{j}\widetilde{u}_{*}\Vert_{L^{q_{1}}_{t}L^{p_{1}}_{x}}+\Vert \widetilde{u}_{*}\Vert_{L^{q_{1}}_{t}W^{1,p_{1}}_{x}})\\
    &\quad+\text{(Remaining terms).}
\end{align*}
Proceeding with the same argument (such as taking $(p,q)=(p_1,q_1)$ as above) used to establish the $L^{q}_{t}W^{1,p}_{x}$ bounds, the estimates for the above terms just follow. The remaining terms can be controlled similarly, so we establish the weighted Strichartz estimates as well.

\noindent
\textbf{Case 2:} Let $\frac{4}{n}\leq 2\sigma<\frac{4}{n-2}$ and $n\geq 3$. \\ We take advantage of $\widetilde{E}_{1}(u(t))$ in this case. Recall that by setting $T=\frac{\widetilde{T}}{1-\widetilde{T}}$, we have the following:
\begin{align*}
    \widetilde{E}_{1}(\widetilde{u}(\widetilde{T}))-\widetilde{E}_{1}(\widetilde{u}(0))&=E(u(T))-E(u(0))\\
    &=\frac{4(2-n\sigma)}{\sigma+1}\int_{0}^{T}(1+s)\Vert u(s)\Vert_{L^{2\sigma+2}_{x}}^{2\sigma+2}ds+\int_{0}^{T}T_{1}(s)ds+\int_{0}^{T}T_{2}(s)dB\\
    &=\frac{4(2-n\sigma)}{\sigma+1}\int_{0}^{\widetilde{T}}(1-t)^{\sigma n-3}\Vert\widetilde{u}(t)\Vert_{L^{2\sigma+2}_{x}}^{2\sigma+2}dt+\int_{0}^{T}T_{1}(s)ds+\int_{0}^{T}T_{2}(s)dB.
\end{align*}
Since $2-n\sigma\leq 0$ for $\frac{4}{n}\leq 2\sigma<\frac{4}{n-2}$, we get that
\begin{align*}
    \widetilde{E}_{1}(\widetilde{u}(\widetilde{T}))\lesssim\Vert u(0)\Vert_{\Sigma}+\sup_{0\leq T<\infty}\left\vert\int_{0}^{T}T_{1}(s)ds+\int_{0}^{T}T_{2}(s)dB\right\vert<\infty,\quad a.s.
\end{align*}
Note that in the mass-critical or mass super-critical case, the pseudo-conformal energy estimate \eqref{eq:growth_of_psedo_energy_mass_super-critical} shows that the quantity $\widetilde{E}_{1}(\widetilde{u}(\widetilde{T}))$ remains bounded uniformly in time. This almost surely gives the following bound for $\Vert\nabla\widetilde{u}(t)\Vert_{L^{2}_{x}}$:
\begin{equation*}
    \Vert\nabla\widetilde{u}(t)\Vert_{L^{2}_{x}}^{2}+(1-t)^{n\sigma-2}\Vert \widetilde{u}(t)\Vert_{L^{2\sigma+2}_{x}}^{2\sigma+2}\lesssim 1.
\end{equation*}
This together with the almost sure boundedness of mass shows that
\begin{equation*}
    \sup_{0\leq t<1}\Vert\widetilde{u}(t)\Vert_{H^{1}_{x}}<\infty\quad a.s.
\end{equation*}
Then, since $\widetilde{u}(t)=\widetilde{u}_{*}(t)+\widetilde{z}_{*}(t)$, for $0<\widetilde{T}<1$, almost surely we obtain
\begin{align*}
    \sup_{\widetilde{T}\leq t<1}\Vert\widetilde{u}_{*}(t)\Vert_{H^{1}_{x}}&\leq\sup_{\widetilde{T}\leq t<1}\Vert\widetilde{u}(t)\Vert_{H^{1}_{x}}+\sup_{\widetilde{T}\leq t<1}\Vert\widetilde{z}_{*}(t)\Vert_{H^{1}_{x}}<\infty
\end{align*}
by assuming that $\langle s\rangle^2 g(s)\in o(s^{-1/2})$ as $s\to\infty$, a.s., where $s=\frac{t}{1-t}$.
Indeed, observing
\begin{equation*}
    \Vert\nabla\widetilde{z}_{*}(t)\Vert_{L^{2}_{x}}\sim\Vert \vert\cdot\vert z_{*}(s)\Vert_{L^{2}_{x}}+(1+s)\Vert\nabla z_{*}(s)\Vert_{L^{2}_{x}},
\end{equation*}
and using \eqref{eq:almost_sure_time_decay_of_tail}, we obtain
\begin{equation*}
    \Vert\widetilde{z}_{*}(t)\Vert_{H^{1}_{x}} \lesssim(1+s)\Vert z_{*}(s)\Vert_{\Sigma} \lesssim(1+s)\langle s\rangle^{-\alpha+\frac{3}{2}},
\end{equation*}
and the supremum over $T\leq s<\infty$, for $T=\frac{\widetilde{T}}{1-\widetilde{T}}$, is finite provided that $\alpha\geq\frac{5}{2}$, which is already the case. This implies almost sure boundedness. Subsequently, applying the same arguments as in the proof of Corollary 2.4 of \cite{Herr_2019} suffices to establish the global-in-time Strichartz estimates. Note that staying in the energy sub-critical regime permits us to justify \eqref{eq:derivation_Strauss_exp}, which is a crucial element in proving the Strichartz estimates.

\noindent
\textbf{Case 3:} Let $\frac{4}{n} < 2\sigma<\infty$ and $n=1,2$. \\ We again have $\sup\limits_{\widetilde{T}\leq t<1}\Vert \widetilde{u}_{*}(t)\Vert_{H^{1}_{x}}<\infty,\, a.s.$ for $0<\widetilde{T}<1$ as in Case 2. Note that for any pair $(p_{1},q_{1})$ satisfying \eqref{eq:Strichartz_admissible_pair}, we introduce a parameter $\theta\in[0,1]$ and find another admissible pair $(p_{2},q_{2})$ depending on $\theta$, so that the following are satisfied:
\begin{equation*}
    \frac{2\sigma\theta}{q_{1}}+\frac{2\sigma(1-\theta)}{q_{2}}+\frac{1}{q_{1}}=\frac{1}{q_{1}'},
\end{equation*}
and
\begin{equation*}
    \frac{1}{p_{2}}=\frac{1}{2}\left(1+\frac{1+\sigma\theta}{\sigma(1-\theta)}\right)-\frac{1}{p_{1}}\frac{1+\sigma\theta}{\sigma(1-\theta)}-\frac{1}{n\sigma(1-\theta)}.
\end{equation*}
Note that the above relations require $\frac{4}{n}\leq 2\sigma$ and we may impose further restrictions on $\theta$, see Section \ref{section:scattering_H^1} for details. Let $I\subseteq[0,1)$ and define the spaces associated with the admissible pairs $(p_{1},q_{1})$ and $(p_{2},q_{2})$
\begin{equation*}
    \begin{aligned}
     Z^{0}_{p_{1},q_{1},n,\theta}(I)&=Z^{0}(I)=L^{q_{1}}_{t}L^{p_{1}}_{x}(I)\cap L^{q_{2}}_{t}L^{p_{2}}_{x}(I)  \\ Z^{1}(I)&=\{f \mid \langle\nabla\rangle f\in Z^{0}(I)\}.
    \end{aligned}
\end{equation*}
Utilizing the Duhamel representation of $\widetilde{u}_{*}(t)$, Hölder's inequality and Sobolev embedding, we get that
\begin{align*}
    \Vert \widetilde{u}_{*}(t)\Vert_{Z^{1}([\widetilde{T},1))}&\lesssim\Vert\widetilde{u}_{*}(\widetilde{T})\Vert_{H^{1}_{x}}+\Vert(1-t)^{n\sigma-2}\vert\widetilde{u}_{*}+\widetilde{z}_{*}\vert^{2\sigma}(\widetilde{u}_{*}+\widetilde{z}_{*})\Vert_{L^{q_{1}'}_{t}W^{1,p_{1}'}_{x}([\widetilde{T},1))}\\
    &\lesssim\Vert\widetilde{u}_{*}(\widetilde{T})\Vert_{H^{1}_{x}}+\Vert(1-t)\Vert_{L^{\infty}_{t}([\widetilde{T},1))}^{n\sigma-2}\Big[\Vert\widetilde{u}_{*}\Vert_{L^{q_{1}}_{t}L^{p_{1}}_{x}([\widetilde{T},1))}^{2\sigma\theta}\Vert\widetilde{u}_{*}\Vert_{L^{q_{2}}_{t}W^{1,p_{2}}_{x}([\widetilde{T},1))}^{2\sigma(1-\theta)}\Vert\widetilde{u}_{*}\Vert_{L^{q_{1}}_{t}W^{1,p_{1}}_{x}([\widetilde{T},1))} .\\
    &\quad+ \text{(Remaining terms)}\Big]\\
    &\lesssim\Vert\widetilde{u}_{*}(\widetilde{T})\Vert_{H^{1}_{x}}+\Vert(1-t)\Vert_{L^{\infty}_{t}([\widetilde{T},1))}^{n\sigma-2}\Big[\Vert\widetilde{u}_{*}\Vert_{Z^{1}([\widetilde{T},1))}^{2\sigma+1}+\text{(Remaining terms)}\Big].
\end{align*}
Notice that $\Vert(1-t)\Vert_{L^{\infty}_{t}([\widetilde{T},t))}^{n\sigma-2}=(1-\widetilde{T})^{n\sigma-2}\to 0$ as $\widetilde{T}\to 1^-$. Therefore, the dual Strichartz norms including the remaining terms can be made arbitrarily small by choosing $\widetilde{T}$ sufficiently close to $1$, leading to $\Vert\widetilde{u}_{*}\Vert_{Z^{1}([\widetilde{T},1))}\lesssim\Vert\widetilde{u}_{*}(\widetilde{T})\Vert_{H^{1}_{x}}$. For the weighted Strichartz norm, note that
\begin{align*}
    \Vert x\widetilde{u}_{*}\Vert_{Z^{0}([\widetilde{T},1))}&\lesssim\Vert \widetilde{u}_{*}(\widetilde{T})\Vert_{\Sigma}+\Vert(1-t)^{n\sigma-2}x\vert\widetilde{u}_{*}+\widetilde{z}_{*}\vert^{2\sigma}(\widetilde{u}_{*}+\widetilde{z}_{*})\Vert_{L^{q_{1}'}_{t}L^{p_{1}'}_{x}([\widetilde{T},1))}\\
    &\quad+\Vert(1-t)^{n\sigma-1}\vert\widetilde{u}_{*}+\widetilde{z}_{*}\vert^{2\sigma}(\widetilde{u}_{*}+\widetilde{z}_{*})\Vert_{L^{q_{1}'}_{t}W^{1,p_{1}'}_{x}([\widetilde{T},1))}\\
    &\lesssim\Vert \widetilde{u}_{*}(\widetilde{T})\Vert_{\Sigma}+\Vert (1-t)\Vert_{L^{\infty}_{t}([\widetilde{T},1))}^{n\sigma-2}\Vert \widetilde{u}_{*}\Vert_{Z^{1}([\widetilde{T},1))}^{2\sigma}\Vert x\widetilde{u}_{*}\Vert_{Z^{0}([\widetilde{T},1))}\\
    &\quad+\Vert (1-t)\Vert_{L^{\infty}_{t}([\widetilde{T},1))}^{n\sigma-1}\Vert \widetilde{u}_{*}\Vert_{Z^{1}([\widetilde{T},1))}^{2\sigma+1}+\text{(Remaining terms)}.
\end{align*}
A similar process as in estimating $\Vert\widetilde{u}_{*}(t)\Vert_{Z^{1}([\widetilde{T},1))}$, but this time using $Z^{0}([\widetilde{T},1))$-norm instead of $Z^{1}([\widetilde{T},1))$-norm, shows that $\Vert x\widetilde{u}_{*}\Vert_{Z^{0}([\widetilde{T},1))}\lesssim\Vert\widetilde{u}_{*}(\widetilde{T})\Vert_{\Sigma}$. This completes the proof in the mass super-critical case.

\noindent
\textbf{Case 4:} Let $2\sigma=\frac{4}{n}$ and $n=1,2$. \\ With the Duhamel formulation for $\widetilde{u}_{*}(s)$ on the time interval $[0,t)$, we write
\begin{align*}
    \Vert\widetilde{u}_{*}\Vert_{C^{0}_{t}H^{1}_{x}([0,t))}+\Vert\widetilde{u}_{*}\Vert_{Z^{1}([0,t))}&\leq C\big[\Vert\widetilde{u}_{*}(0)\Vert_{H^{1}_{x}}+\Vert\widetilde{z}_{*}\Vert_{Z^{1}([0,t))}^{2\sigma+1}+\Vert \widetilde{z}_{*}\Vert_{Z^{1}([0,t))}^{2\sigma}\Vert \widetilde{u}_{*}\Vert_{Z^{1}([0,t))}\\
    &\quad\quad+\Vert \widetilde{z}_{*}\Vert_{Z^{1}([0,t))}\Vert \widetilde{u}_{*}\Vert_{Z^{1}([0,t))}^{2\sigma}+\Vert \widetilde{u}_{*}\Vert_{Z^{1}([0,t))}^{2\sigma+1}\big].
\end{align*}
Recall from \eqref{eq:global_Strichartz_est_for_tail} that $\widetilde{z}_{*}$ satisfies global-in-time Strichartz estimates. Therefore, given $\varepsilon>0$, we may partition $[0,t)=\bigcup_{j=1}^{M}I_{j}$, where $I_{j}=[t_{j},t_{j+1})$ and $M=M(\varepsilon,t)$, so that we have the following inequality for $1\leq j\leq M$:
\begin{equation*}
    \Vert\widetilde{z}_{*}\Vert_{Z^{1}(I_{j})}\leq\varepsilon.
\end{equation*}
Then with the above inequalities we have
\begin{align*}
     \Vert\widetilde{u}_{*}\Vert_{C^{0}_{t}H^{1}_{x}(I_{1})}+\Vert\widetilde{u}_{*}\Vert_{Z^{1}(I_{1})}\leq C(\Vert\widetilde{u}_{*}(0)\Vert_{H^{1}_{x}}+\varepsilon^{2\sigma}\Vert \widetilde{u}_{*}\Vert_{Z^{1}(I_{1})}+\varepsilon\Vert \widetilde{u}_{*}\Vert_{Z^{1}(I_{1})}^{2\sigma}+\Vert\widetilde{u}_{*}\Vert_{Z^{1}(I_{1})}^{2\sigma+1}+\varepsilon^{2\sigma+1}).
\end{align*}
Choosing $\varepsilon>0$ small enough, we get for some constant $C_{1}\geq 1$ that
\begin{align*}
    \Vert\widetilde{u}_{*}\Vert_{C^{0}_{t}H^{1}_{x}(I_{1})}+\Vert\widetilde{u}_{*}\Vert_{Z^{1}(I_{1})}\leq C_{1}\Vert\widetilde{u}_{*}(0)\Vert_{H^{1}_{x}}+C_{1}\Vert\widetilde{u}_{*}\Vert_{Z^{1}(I_{1})}^{2\sigma+1}.
\end{align*}
Now, we shall use Lemma \ref{lemma:uniform_bounded_by_its_higher_powers} by choosing $a=C_{1}\Vert\widetilde{u}_{*}(0)\Vert_{H^{1}_{x}}$ and $b=C_{1}$ together with the condition
\begin{equation*}
    C_{1}\Vert\widetilde{u}_{*}(0)\Vert_{H^{1}_{x}}<\left(\frac{2\sigma}{2\sigma+1}\right)\left((2\sigma+1)C_{1}\right)^{-\frac{1}{2\sigma}}\Rightarrow\Vert\widetilde{u}_{*}(0)\Vert_{H^{1}_{x}}<\left(\frac{2\sigma}{2\sigma+1}\right)(2\sigma+1)^{-\frac{1}{2\sigma}}(C_{1})^{-1-\frac{1}{2\sigma}}.
\end{equation*}
Then we have
\begin{equation*}
    \Vert\widetilde{u}_{*}\Vert_{C^{0}_{t}H^{1}_{x}(I_{1})}+\Vert\widetilde{u}_{*}\Vert_{Z^{1}(I_{1})}\leq C_{1}\left(\frac{2\sigma+1}{2\sigma}\right)\Vert\widetilde{u}_{*}(0)\Vert_{H^{1}_{x}}.
\end{equation*}
For the second subinterval $I_{2}$, note that we have $\Vert\widetilde{u}_{*}(t_{2})\Vert_{H^{1}_{x}}\leq C_{1}\left(\frac{2\sigma+1}{2\sigma}\right)\Vert\widetilde{u}_{*}(0)\Vert_{H^{1}_{x}}$. Thus, pursuing the same procedure for the interval $I_{2}$ gives rise to
\begin{align*}
    \Vert\widetilde{u}_{*}\Vert_{C^{0}_{t}H^{1}_{x}(I_{2})}+\Vert\widetilde{u}_{*}\Vert_{Z^{1}(I_{2})}&\leq C_{1}\Vert\widetilde{u}_{*}(t_{2})\Vert_{H^{1}_{x}}+C_{1}\Vert\widetilde{u}_{*}\Vert_{Z^{1}(I_{2})}^{2\sigma+1}\\
    &\leq C_{2}\Vert\widetilde{u}_{*}(0)\Vert_{H^{1}_{x}}+C_{1}\Vert\widetilde{u}_{*}\Vert_{Z^{1}(I_{2})}^{2\sigma+1}
\end{align*}
where $C_{2}=C_{1}^{2}\left(\frac{2\sigma+1}{2\sigma}\right)$. Here to apply Lemma \ref{lemma:uniform_bounded_by_its_higher_powers}, we need the following condition to be satisfied:
\begin{equation*}
    C_{2}\Vert\widetilde{u}_{*}(0)\Vert_{H^{1}_{x}}<\left(\frac{2\sigma}{2\sigma+1}\right)((2\sigma+1)C_{1})^{-\frac{1}{2\sigma}}  
\end{equation*}
which is also equivalent to demanding that
\begin{equation}\label{eq:step_2_condition}
    \Vert\widetilde{u}_{*}(0)\Vert_{H^{1}_{x}}<\left(\frac{2\sigma}{2\sigma+1}\right)^{2}(2\sigma+1)^{-\frac{1}{2\sigma}}C_{1}^{-2-\frac{1}{2\sigma}}.
\end{equation}
Then assuming \eqref{eq:step_2_condition}, we get that
\begin{equation*}
    \Vert\widetilde{u}_{*}\Vert_{C^{0}_{t}H^{1}_{x}(I_{2})}+\Vert\widetilde{u}_{*}\Vert_{Z^{1}(I_{2})}\leq C_{2}\left(\frac{2\sigma+1}{2\sigma}\right)\Vert\widetilde{u}_{*}(0)\Vert_{H^{1}_{x}}=C_{1}^{2}\left(\frac{2\sigma+1}{2\sigma}\right)^{2}\Vert\widetilde{u}_{*}(0)\Vert_{H^{1}_{x}}.
\end{equation*}
Now, we will implement an induction argument for $1\leq j\leq M$ by assuming that
\begin{equation}\label{eq:bound_of_data_for_induction}
    \Vert\widetilde{u}_{*}(0)\Vert_{H^{1}_{x}}<\left(\frac{2\sigma}{2\sigma+1}\right)^{M}(2\sigma+1)^{-\frac{1}{2\sigma}}C_{1}^{-M-\frac{1}{2\sigma}},
\end{equation}
which gives rise to satisfy the condition in Lemma \ref{lemma:uniform_bounded_by_its_higher_powers} at every step of the induction process. Our claim is that for all $1\leq j\leq M$, the following inequality holds true:
\begin{equation*}
    \Vert\widetilde{u}_{*}\Vert_{C^{0}_{t}H^{1}_{x}(I_{j})}+\Vert\widetilde{u}_{*}\Vert_{Z^{1}(I_{j})}\leq C_{j}\left(\frac{2\sigma+1}{2\sigma}\right)\Vert\widetilde{u}_{*}(0)\Vert_{H^{1}_{x}}=C_{1}^{j}\left(\frac{2\sigma+1}{2\sigma}\right)^{j}\Vert\widetilde{u}_{*}(0)\Vert_{H^{1}_{x}}.
\end{equation*}
Suppose that the claim is true for some $1<k<M$. Then we can infer that $$\Vert\widetilde{u}_{*}(t_{k+1})\Vert_{H^{1}_{x}}\leq C_{1}^{k}\left(\frac{2\sigma+1}{2\sigma}\right)^{k}\Vert\widetilde{u}_{*}(0)\Vert_{H^{1}_{x}},$$ and at the $(k+1)$-st step, we arrive at
\begin{align*}
    \Vert\widetilde{u}_{*}\Vert_{C^{0}_{t}H^{1}_{x}(I_{k+1})}+\Vert\widetilde{u}_{*}\Vert_{Z^{1}(I_{k+1})}&\leq C_{1}\Vert\widetilde{u}_{*}(t_{k+1})\Vert_{H^{1}_{x}}+C_{1}\Vert\widetilde{u}_{*}\Vert_{Z^{1}(I_{k+1})}^{2\sigma+1}\\
    &\leq C_{1}^{k+1}\left(\frac{2\sigma+1}{2\sigma}\right)^{k}\Vert\widetilde{u}_{*}(0)\Vert_{H^{1}_{x}}+C_{1}\Vert\widetilde{u}_{*}\Vert_{Z^{1}(I_{k+1)}}^{2\sigma+1}.
\end{align*}
Now, we wish to apply Lemma \ref{lemma:uniform_bounded_by_its_higher_powers} once the following condition holds:
\begin{equation*}
    C_{1}^{k+1}\left(\frac{2\sigma+1}{2\sigma}\right)^{k}\Vert\widetilde{u}_{*}(0)\Vert_{H^{1}_{x}}<\left(\frac{2\sigma}{2\sigma+1}\right)((2\sigma+1)C_{1})^{-\frac{1}{2\sigma}}
\end{equation*}
which is equivalent to
\begin{equation*}
    \Vert\widetilde{u}_{*}(0)\Vert_{H^{1}_{x}}<\left(\frac{2\sigma}{2\sigma+1}\right)^{k+1}(2\sigma+1)^{-\frac{1}{2\sigma}}C_{1}^{-(k+1)-\frac{1}{2\sigma}}.
\end{equation*}
This inequality is already satisfied due to the assumption \eqref{eq:bound_of_data_for_induction}. Therefore, utilizing Lemma \ref{lemma:uniform_bounded_by_its_higher_powers}, we get that
\begin{equation*}
    \Vert\widetilde{u}_{*}\Vert_{C^{0}_{t}H^{1}_{x}(I_{k+1})}+\Vert\widetilde{u}_{*}\Vert_{Z^{1}(I_{k+1})}\leq \left(\frac{2\sigma+1}{2\sigma}\right)C_{1}^{k+1}\left(\frac{2\sigma+1}{2\sigma}\right)^{k}\Vert\widetilde{u}_{*}(0)\Vert_{H^{1}_{x}}=C_{k+1}\Vert\widetilde{u}_{*}(0)\Vert_{H^{1}_{x}},
\end{equation*}
which proves the claim. Therefore, 
\begin{equation*}
    \Vert\widetilde{u}_{*}\Vert_{Z^{1}([0,t))}\leq \Vert\widetilde{u}_{*}(0)\Vert_{H^{1}_{x}}\sum_{j=1}^{M}C_{1}^{j}\left(\frac{2\sigma+1}{2\sigma}\right)^{j}.
\end{equation*}
Letting $t\to 1^-$ and noting that $M_{1}:=\lim\limits_{t\to 1^-}M(\varepsilon,t)<\infty$, we obtain
\begin{equation*}
    \Vert\widetilde{u}_{*}\Vert_{Z^{1}([0,1))}\leq\Vert\widetilde{u}_{*}(0)\Vert_{H^{1}_{x}}\sum_{j=1}^{M_{1}}C_{1}^{j}\left(\frac{2\sigma+1}{2\sigma}\right)^{j}\leq \Vert\widetilde{u}_{*}(0)\Vert_{H^{1}_{x}}\frac{\left(C_{1}\frac{2\sigma+1}{2\sigma}\right)^{M_{1}}-1}{\ln M_{1}}<\infty.
\end{equation*}
It is left to show the weighted Strichartz estimate. Using the Duhamel representation for $\widetilde{u}_{*}$, it follows that
\begin{align*}
    \Vert x\widetilde{u}_{*}\Vert_{Z^{0}([\widetilde{T},1))}&\lesssim\Vert\widetilde{u}_{*}(\widetilde{T})\Vert_{\Sigma}+\Vert\widetilde{u}_{*}\Vert_{Z^{1}([\widetilde{T},1))}^{2\sigma}\Vert x\widetilde{u}_{*}\Vert_{Z^{0}([\widetilde{T},1))}+\Vert\widetilde{u}_{*}\Vert_{Z^{1}([\widetilde{T},1))}^{2\sigma}\Vert x\widetilde{z}_{*}\Vert_{Z^{0}([\widetilde{T},1))}\\
    &+\Vert\widetilde{z}_{*}\Vert_{Z^{1}([\widetilde{T},1))}^{2\sigma}\Vert x\widetilde{u}_{*}\Vert_{Z^{0}([\widetilde{T},1))}+\Vert\widetilde{z}_{*}\Vert_{Z^{1}([\widetilde{T},1))}^{2\sigma}\Vert x\widetilde{u}_{*}\Vert_{Z^{0}([\widetilde{T},1))}\\
    &+\Vert(1-t)\Vert_{L^{\infty}_{t}([\widetilde{T},1))}^{n\sigma-1}C(\Vert\widetilde{u}_{*}\Vert_{Z^{1}([\widetilde{T},1))},\Vert\widetilde{z}_{*}\Vert_{Z^{1}([\widetilde{T},1))}).
\end{align*}
Note that $n\sigma-1>0$ by assumption and we have
\begin{equation*}
    \Vert(1-t)\Vert_{L^{\infty}_{t}([\widetilde{T},1))}^{n\sigma-1}=(1-\widetilde{T})^{n\sigma-1}\to 0\text{ as $\widetilde{T}\to 1^-$}.
\end{equation*}
Furthermore, $\widetilde{z}_{*}$ satisfies the global Strichartz estimates by \eqref{eq:global_Strichartz_est_for_tail} and $\widetilde{u}_{*}$ belongs to $L^{q}_{t}W^{1,p}_{x}([0,1)\times\mathbb{R}^{n})$ for any admissible pair $(p,q)$. Therefore, there exists $\varepsilon(\widetilde{T})>0$ such that $\Vert\widetilde{u}_{*}\Vert_{Z^{1}([\widetilde{T},1))},\Vert\widetilde{z}_{*}\Vert_{Z^{1}([\widetilde{T},1))}\leq\varepsilon (\widetilde{T})$ and $\varepsilon(\widetilde{T})\to 0$ as $\widetilde{T}\to 1^-$. These imply that for $\widetilde{T}$ taken sufficiently close to $1$, we have
\begin{equation*}
    \Vert x\widetilde{u}_{*}\Vert_{Z^{0}([\widetilde{T},1))}\lesssim\Vert\widetilde{u}_{*}(\widetilde{T})\Vert_{\Sigma}.
\end{equation*}
This completes the proof of the theorem.

\end{proof}
\begin{Remark}
In the next section, we will not require global-in-time Strichartz estimates \eqref{eq:global_Strichartz_for_random_eqn} when proving almost sure $L^{2}_{x}$ scattering. This allows us to demonstrate scattering in the entire short-range case.
\end{Remark}

\section{Scattering in $L^2_x$: Proof of Theorem \ref{thm:result_L^2_scattering}}\label{section:scattering_L^2}
Armed with the observations made in the previous section, we are now ready to show the $L^{2}_{x}$ scattering for $\frac{2}{n}<2\sigma<\frac{4}{n}$. By using the identity $\widetilde{E}_{2}(\widetilde{u}(t))=(1-t)^{(2-n\sigma)}E(u(s))$ and \eqref{eq:growth_of_psedo_energy_mass_sub-critical}, we have that
\begin{equation*}    
   \mathbb{E}\left[\Vert\nabla\widetilde{u}(t)\Vert_{L^{2}_{x}}\right]\lesssim(1-t)^{\frac{n\sigma}{2}-1}\quad\text{and}\quad \mathbb{E}\left[\Vert\widetilde{u}(t)\Vert_{L^{2\sigma+2}_{x}}\right]\lesssim 1,\quad\text{for $0\leq t<1$,}
\end{equation*}
with the implicit constants depending on the $\Sigma$ norm of the data. In the same range of $t$, the same bounds above also almost surely hold for $\widetilde{u}_{*}(t)=\widetilde{u}(t)-\widetilde{z}_{*}(t)$ instead (which have been observed in the proof of Theorem \ref{thm:global_Strichartz}). Recall that the mass of the solution $u$ to \eqref{eq:additive_SNLS} is almost surely bounded. We shall first show that $\widetilde{u}_{*}(t)$ converges weakly in $L^{2}_{x}$ as $t\to 1^-$, and it suffices to show this for Schwartz functions. Thus, let $\psi$ be a Schwartz function and $0<t_{1}<t_{2}<1$. Then integrating by parts and using the Cauchy-Schwarz inequality, we almost surely have
\begin{align*}
\vert\langle\widetilde{u}_{*}(t_{2})-\widetilde{u}_{*}(t_{1}),\psi\rangle\vert&=\left\vert\left\langle\int_{t_{1}}^{t_{2}}\partial_{t}\widetilde{u}_{*}(t)dt,\psi\right\rangle\right\vert\\    &\lesssim\int_{t_{1}}^{t_{2}}\vert\langle\Delta \widetilde{u}_{*},\psi\rangle\vert+(1-t)^{n\sigma-2}\vert\langle\vert\widetilde{u}_{*}+\widetilde{z}_{*}\vert^{2\sigma}(\widetilde{u}_{*}+\widetilde{z}_{*}),\psi\rangle\vert dt\\    &\lesssim\int_{t_{1}}^{t_{2}}\Vert\nabla\widetilde{u}_{*}\Vert_{L^{2}_{x}}\Vert\nabla\psi\Vert_{L^{2}_{x}}+(1-t)^{n\sigma-2}C(\Vert\widetilde{u}_{*}\Vert_{L^{2\sigma+2}_{x}},\Vert\widetilde{z}_{*}\Vert_{L^{2\sigma+2}_{x}},\Vert\psi\Vert_{L^{2\sigma+2}_{x}})\, dt\\
    &\lesssim\int_{t_{1}}^{t_{2}}(1-t)^{\frac{n\sigma}{2}-1}+(1-t)^{n\sigma-2}dt\to 0\quad\text{as $t_{1},t_{2}\to1^-$}
\end{align*}
since $\frac{2}{n}<2\sigma$. This shows the weak convergence, so let us define $\widetilde{u}_{+}$ as the weak limit of $\widetilde{u}_{*}(t)$ as $t\to 1^-$. Next we shall demonstrate the strong convergence by showing that $\Vert\widetilde{u}_{*}(t)-\widetilde{u}_{+}\Vert_{L^{2}_{x}}\to 0$ as $t\to 1^-$. Note that for fixed $t$, we have
\begin{equation*}
    \lim_{s\to 1^-}\langle\widetilde{u}_{*}(t)-\widetilde{u}_{*}(s),\widetilde{u}_{+}\rangle=\langle\widetilde{u}_{*}(t)-\widetilde{u}_{+},\widetilde{u}_{+}\rangle.
\end{equation*}
Then
\begin{equation*}
    \langle\widetilde{u}_{*}(t)-\widetilde{u}_{+},\widetilde{u}_{+}\rangle\to 0\quad\text{as $t\to 1^-$}.
\end{equation*}
Thus, it suffices to show the following:
\begin{equation*}
    \lim_{t\to 1^-}\lim_{s\to 1^-}\langle\widetilde{u}_{*}(t)-\widetilde{u}_{*}(s),\widetilde{u}_{*}(t)\rangle=0.
\end{equation*}
Proceeding similarly as in the weak convergence computation above (using the range of $\sigma$), we almost surely obtain
\begin{align*}
    \vert\langle\widetilde{u}_{*}(t)-\widetilde{u}_{*}(s),\widetilde{u}_{*}(t)\rangle\vert&\lesssim\int_{s}^{t}\Vert\nabla\widetilde{u}_{*}(\tau)\Vert_{L^{2}_{x}}\Vert\nabla\widetilde{u}_{*}(t)\Vert_{L^{2}_{x}}+(1-\tau)^{n\sigma-2}d\tau\\
    &\lesssim\int_{s}^{t}(1-\tau)^{\frac{n\sigma}{2}-1}(1-t)^{\frac{n\sigma}{2}-1}+(1-\tau)^{n\sigma-2}d\tau\to 0\quad \text{as $t,s\to 1^-$}.
\end{align*}
Therefore, $\Vert \widetilde{u}_{*}(t)-u_{+}\Vert_{L^{2}_{x}}\to 0$, a.s., as $t\to 1^-$. To show the $L^{2}_{x}$ scattering, it remains to prove the equivalence of asymptotics between $\widetilde{u}_{*}$ and $u_{*}$ at $1$ and $\infty$, respectively. Define the dilation operator $D_{\beta}$ as
\begin{equation*}
    D_{\beta}f(x)=\beta^{\frac{n}{2}}f(\beta x),\quad\beta>0,
\end{equation*}
and the multiplication $M_{\theta}$ as 
\begin{equation*}
    M_{\theta}f(x)=e^{i\frac{\theta\vert x\vert^{2}}{4}}f(x),\quad\theta\in\mathbb{R}.
\end{equation*}
Denote $S(t)=e^{-it\Delta}$ as the linear propagator associated with the Schrödinger equation \eqref{eq:additive_SNLS}. Then, we may write 
\begin{equation*}
    \widetilde{u}_{*}(t)=M_{\frac{1}{1-t}}D_{\frac{1}{1-t}}u_{*}\left(\frac{t}{1-t}\right),\quad t\in[0,1),
\end{equation*}
which is yet another expression for pseudo-conformal transformation. We have the following identities:
\begin{enumerate}[(i)]
    \item $S(-t)D_{\beta}=D_{\beta}S(-\beta^{2}t)$,\label{commutation_dilation}
    \item $S(-t)M_{\theta}=M_{\frac{\theta}{1+\theta t}}D_{\frac{1}{1+\theta t}}S(-\frac{t}{1+\theta t})$.\label{commutation_multiplication}
\end{enumerate}
Using the identities \eqref{commutation_dilation}, \eqref{commutation_multiplication} together with
\begin{equation*}
    u_{*}(s)=M_{-\frac{1}{1+s}}D_{\frac{1}{1+s}}\widetilde{u}_{*}\left(\frac{s}{1+s}\right),
\end{equation*}
we obtain
\begin{align*}
    S(-s)u_{*}(s)&=S(-s)M_{-\frac{1}{1+s}}D_{\frac{1}{1+s}}\widetilde{u}_{*}\left(\frac{s}{1+s}\right)\\
    &=M_{-1}S\left(-\frac{s}{1+s}\right)\widetilde{u}_{*}\left(\frac{s}{1+s}\right)\\
    &=M_{-1}S(-t)\widetilde{u}_{*}(t).
\end{align*}
By Proposition 7.5.1 of \cite{Cazenave_book}, as $\widetilde{u}_{*}(t)$ has a strong limit in $L^{2}_{x}$ as $t\to 1^-$, $u_{*}(s)$ has a strong limit in $L^{2}_{x}$ as $s\to\infty$ and the following holds
\begin{equation*}
    \lim_{s\to\infty}S(-s)u_{*}(s)=\lim_{t\to 1^-}M_{-1}S(-t)\widetilde{u}_{*}(t)=M_{-1}S(-1)\widetilde{u}_{+}\quad \text{in }L^{2}_{x}.
\end{equation*}

\section{Scattering in $\Sigma$ space: Proof of Theorem \ref{thm:result_Sigma_scattering}}\label{section:scattering_Sigma}
It suffices to show that the global solution $u_{*}(s)$ to the random equation \eqref{eq:additive_random_eqn} has a strong limit in $\Sigma$ after taking the pseudo-conformal transformation (this will be sufficient to establish the scattering result via Proposition 7.5.1 from \cite{Cazenave_book}). 

Recall that
\begin{equation*}
    S(-s)u_{*}(s)=M_{-1}S(-t)\widetilde{u}_{*}(t).
\end{equation*}
Since $\Sigma$ is a Banach space equipped with the $\Sigma$ norm, it is enough to show that $\{S(-t)\widetilde{u}_{*}(t)\}$ is a Cauchy sequence for $t\in[0,1)$. We shall start with the weighted norm. Recall the Duhamel formulation of $\widetilde{u}_{*}(t)$:
\begin{equation*}
\widetilde{u}_{*}(t)=S(t)\widetilde{u}_{*}(0)+i\int_{0}^{t}S(t-t')(1-t')^{n\sigma-2}\vert\widetilde{u}_{*}+\widetilde{z}_{*}\vert^{2\sigma}(\widetilde{u}_{*}+\widetilde{z}_{*})(t')dt'
\end{equation*}
implying that
\begin{equation*}
    S(-t)\widetilde{u}_{*}(t)=\widetilde{u}_{*}(0)+i\int_{0}^{t}S(-t')(1-t')^{n\sigma-2}\vert\widetilde{u}_{*}+\widetilde{z}_{*}\vert^{2\sigma}(\widetilde{u}_{*}+\widetilde{z}_{*})(t')dt'.
\end{equation*}
Set $J(t):= x-2it\nabla$. We have the following identities:
\begin{enumerate}[(I)]
    \item $J(t)=S(-t)xS(t)\Rightarrow$ $S(t)J(t)=xS(t)$ or $J(t)S(-t)=S(-t)x$,\label{first_identity}
    \item $J(t)=M_{-\frac{1}{t}}(-2it\nabla) M_{\frac{1}{t}}$.\label{second_identity}
\end{enumerate}
Let $0<t_{1}<t_{2}<1$. Note that
\begin{align*}
    xS(-t_{2})\widetilde{u}_{*}(t_{2})-xS(-t_{1})\widetilde{u}_{*}(t_{1})=ix\int_{t_{1}}^{t_{2}}S(-t')(1-t')^{n\sigma-2}\vert\widetilde{u}_{*}+\widetilde{z}_{*}\vert^{2\sigma}(\widetilde{u}_{*}+\widetilde{z}_{*})(t')dt'.
\end{align*}
Then using the identity \eqref{first_identity} and the $L^{2}_{x}\to L^{2}_{x}$ isometry of $S(t)$, we get the following:
\begin{align*}
    &\Vert xS(-t_{2})\widetilde{u}_{*}(t_{2})-xS(-t_{1})\widetilde{u}_{*}(t_{1})\Vert_{L^{2}_{x}}\\&=\left\Vert x\int_{t_{1}}^{t_{2}}S(-t')(1-t')^{n\sigma-2}\vert\widetilde{u}_{*}+\widetilde{z}_{*}\vert^{2\sigma}(\widetilde{u}_{*}+\widetilde{z}_{*})(t')dt'\right\Vert_{L^{2}_{x}}\\
    &=\left\Vert S(1)x\int_{t_{1}}^{t_{2}}S(-t')(1-t')^{n\sigma-2}\left(\vert\widetilde{u}_{*}+\widetilde{z}_{*}\vert^{2\sigma}(\widetilde{u}_{*}+\widetilde{z}_{*})(t')\right)dt'\right\Vert_{L^{2}_{x}}\\
    &=\left\Vert (x+2i\nabla)\int_{t_{1}}^{t_{2}}S(1-t')(1-t')^{n\sigma-2}\left(\vert\widetilde{u}_{*}+\widetilde{z}_{*}\vert^{2\sigma}(\widetilde{u}_{*}+\widetilde{z}_{*})(t')\right)dt\right\Vert_{L^{2}_{x}}\\
    &\lesssim\left\Vert\int_{t_{1}}^{t_{2}}S(-t')(1-t')^{n\sigma-2}\nabla\left(\vert\widetilde{u}_{*}+\widetilde{z}_{*}\vert^{2\sigma}(\widetilde{u}_{*}+\widetilde{z}_{*})(t')\right)dt'\right\Vert_{L^{2}_{x}}\\
    &+\left\Vert\int_{t_{1}}^{t_{2}}S(-t')(1-t')^{n\sigma-2}J(1-t')\left(\vert\widetilde{u}_{*}+\widetilde{z}_{*}\vert^{2\sigma}(\widetilde{u}_{*}+\widetilde{z}_{*})(t')\right)dt'\right\Vert_{L^{2}_{x}}\\
    &\lesssim\Vert (1-t)^{n\sigma-2}J(1-t)\left(\vert\widetilde{u}_{*}+\widetilde{z}_{*}\vert^{2\sigma}(\widetilde{u}_{*}+\widetilde{z}_{*})\right)\Vert_{L^{q_{1}'}_{t}L^{p_{1}'}_{x}((t_{1},t_{2})\times\mathbb{R}^{n})}\\
    &+\Vert (1-t)^{n\sigma-2}\vert\widetilde{u}_{*}+\widetilde{z}_{*}\vert^{2\sigma}(\widetilde{u}_{*}+\widetilde{z}_{*})\Vert_{L^{q_{1}'}_{t}W^{1,p_{1}'}_{x}((t_{1},t_{2})\times\mathbb{R}^{n})}
\end{align*}
where $(p_{1},q_{1})$ is an admissible pair. Consider the pair $(2\sigma+2,\gamma)=(2\sigma+2,\frac{2\sigma(2\sigma+2)}{2-\sigma(n-2)})$. Then we may find another admissible pair $(p_{2},q_{2})$ satisfying
\begin{equation*}
    \frac{1}{p_{1}'}=\frac{\sigma}{\sigma+1}+\frac{1}{p_{2}}\quad\text{and}\quad \frac{1}{q_{1}'}=\frac{2\sigma}{\gamma}+\frac{1}{q_{2}}.
\end{equation*}
Indeed, to verify, 
\begin{align*}
    n\left(\frac{1}{2}-\frac{1}{p_{2}}\right)&=n\left(\frac{1}{2}+\frac{\sigma}{\sigma+1}-\frac{1}{p_{1}'}\right) =n\left(\frac{1}{2}+\frac{\sigma}{\sigma+1}-1+\frac{1}{p_{1}}\right) = -\frac{2}{q_{1}}+\frac{n\sigma}{\sigma+1} \\
    &=\frac{2}{q_{1}'}-2+\frac{n\sigma}{\sigma+1} =\frac{2}{q_{1}'}-\frac{2\sigma+2-n\sigma}{\sigma+1} =2\left(\frac{1}{q_{1}'}-\frac{2\sigma}{\theta}\right)=\frac{2}{q_{2}}.
\end{align*}
Note that we only require $2\sigma\leq\gamma$, which is the case for $\sigma(n)<2\sigma<\frac{4}{n-2}$ since we have
\begin{equation*}
    2\sigma-\frac{2\sigma(2\sigma+2)}{2-\sigma(n-2)}=-\frac{n\sigma}{2-\sigma(n-2)}\leq 0.
\end{equation*}
Now using the identity \eqref{second_identity}, Leibniz rule and triangle inequality, we have
\begin{align*}
    &\left\Vert (1-t)^{n\sigma-2}J(1-t)\left(\vert\widetilde{u}_{*}+\widetilde{z}_{*}\vert^{2\sigma}(\widetilde{u}_{*}+\widetilde{z}_{*})\right)\right\Vert_{L^{q_{1}'}_{t}L^{p_{1}'}_{x}((t_{1},t_{2})\times\mathbb{R}^{n})}\\
    &\lesssim \left\Vert (1-t)^{n\sigma-2}(-2i(1-t)\nabla)\left(\vert M_{\frac{1}{1-t}}\widetilde{u}_{*}\vert^{2\sigma}(M_{\frac{1}{1-t}}\widetilde{u}_{*})\right)\right\Vert_{L^{q_{1}'}_{t}L^{p_{1}'}_{x}((t_{1},t_{2})\times\mathbb{R}^{n})}+\text{(Remaining terms)}.
\end{align*}
Observe that
\begin{align*}
    (-2i(1-t)\nabla)M_{\frac{1}{1-t}}\widetilde{u}_{*}&=(-2i(1-t)\nabla)e^{\frac{i\vert x\vert^{2}}{4(1-t)}}\widetilde{u}_{*}(x,t)\\
    &=-2i(1-t) e^{\frac{i\vert x\vert^{2}}{4(1-t)}}\left(\frac{i}{2(1-t)}x+\nabla\right)\widetilde{u}_{*}(x,t)=M_{\frac{1}{1-t}}J(1-t)\widetilde{u}_{*}(x,t).
\end{align*}
Then by the Leibniz rule and above equality
\begin{align*}
    &\left\Vert (1-t)^{n\sigma-2}(-2i(1-t)\nabla)\left(\vert M_{\frac{1}{1-t}}\widetilde{u}_{*}\vert^{2\sigma}(M_{\frac{1}{1-t}}\widetilde{u}_{*})\right)\right\Vert_{L^{q_{1}'}_{t}L^{p_{1}'}_{x}((t_{1},t_{2})\times\mathbb{R}^{n})}\\
    &\lesssim\left\Vert (1-t)^{n\sigma-2}\,\widetilde{u}_{*}^{2\sigma}J(1-t)\widetilde{u}_{*}\right\Vert_{L^{q_{1}'}_{t}L^{p_{1}'}_{x}((t_{1},t_{2})\times\mathbb{R}^{n})}\\    
    &\lesssim\Vert(1-t)^{n\sigma-2}\widetilde{u}_{*}^{2\sigma}\Vert_{L^{\frac{\gamma}{2\sigma}}_{t}L^{\frac{2\sigma+2}{2\sigma}}_{x}((t_{1},t_{2})\times\mathbb{R}^{n})}\Vert J(1-t)\widetilde{u}_{*}\Vert_{L^{q_{2}}_{t}L^{p_{2}}_{x}((t_{1},t_{2})\times\mathbb{R}^{n})}\\
    &\lesssim \Vert(1-t)^{\frac{n\sigma-2}{2\sigma}}\widetilde{u}_{*}\Vert_{L^{\gamma}_{t}L^{2\sigma+2}_{x}((t_{1},t_{2})\times\mathbb{R}^{n})}^{2\sigma}\Vert J(1-t)\widetilde{u}_{*}\Vert_{L^{q_{2}}_{t}L^{p_{2}}_{x}((t_{1},t_{2})\times\mathbb{R}^{n})}\\    
    &\lesssim\Vert(1-t)^{\frac{n\sigma-2}{2\sigma}}\Vert_{L^{\gamma}_{t}((t_{1},t_{2}))}\sup_{t_{1}<t<t_{2}}\Vert\widetilde{u}_{*}(t)\Vert_{L^{2\sigma+2}_{x}}^{2\sigma}\left(\Vert  x\widetilde{u}_{*}\Vert_{L^{q_{2}}_{t}L^{p_{2}}_{x}((t_{1},t_{2})\times\mathbb{R}^{n})}+\Vert (1-t)\nabla\widetilde{u}_{*}\Vert_{L^{q_{2}}_{t}L^{p_{2}}_{x}((t_{1},t_{2})\times\mathbb{R}^{n})}\right).
\end{align*}
Notice that the first factor above is finite for all $0<t_{1}<t_{2}<1$ if $\sigma(n)<2\sigma<\frac{4}{n}$. The second factor is also finite for the same range of $2\sigma$. The remaining two terms in the third factor are controlled by the Strichartz norms of $\widetilde{u}_{*}(t)$, which have shown to be finite on the time interval $[0,1)$. The nonlinear terms involving $\widetilde{z}_{*}(t)$ can be shown to tend to $0$ as $t_{1},t_{2}\to 1^-$ by implementing the same arguments and using the decay properties of $g(s)$.

The second term of the upper bound of the weighted $L^{2}_{x}$-norm of the differences of $S(-t)u(t)$ can be treated similarly, but in this case, combining Leibniz rule, Hölder inequalities with the same spacetime exponents as in the previous estimation and the global-in-time Strichartz estimates, we deduce that this term also tends to zero as $t_{1},t_{2}\to 1^{-}$. Hence, this demonstrates that
\begin{equation*}
     \Vert xS(-t_{2})\widetilde{u}_{*}(t_{2})-xS(-t_{1})\widetilde{u}_{*}(t_{1})\Vert_{L^{2}_{x}}\to 0\quad\text{as $t_{1},t_{2}\to 1^-$}.
\end{equation*}
It remains to show that
\begin{equation*}
    \Vert S(-t_{2})\widetilde{u}_{*}(t_{2})-S(-t_{1})\widetilde{u}_{*}(t_{1})\Vert_{H^{1}_{x}}\to 0\quad\text{as $t_{1},t_{2}\to 1^-$}.
\end{equation*}
Again utilizing the Duhamel representation for $\widetilde{u}_{*}$, using Leibniz rule and Hölder's inequality, we obtain
\begin{align*}
    \Vert S(-t_{2})\widetilde{u}_{*}&(t_{2})-S(-t_{1})\widetilde{u}_{*}(t_{1})\Vert_{H^{1}_{x}}=\left\Vert\int_{t_{1}}^{t_{2}}S(-t')(1-t)^{n\sigma-2}\langle\nabla\rangle\left(\vert\widetilde{u}_{*}+\widetilde{z}_{*}\vert^{2\sigma}(\widetilde{u}_{*}+\widetilde{z}_{*})(t')\right)dt'\right\Vert_{L^{2}_{x}}\\    
    &\lesssim\Vert(1-t)^{n\sigma-2}\langle\nabla\rangle\left(\vert\widetilde{u}_{*}+\widetilde{z}_{*}\vert^{2\sigma}(\widetilde{u}_{*}+\widetilde{z}_{*})\right)\Vert_{L^{q_{1}'}_{t}L^{p_{1}'}_{x}((t_{1},t_{2})\times\mathbb{R}^{n})}\\
    &\lesssim\Vert(1-t)^{n\sigma-2}\widetilde{u}_{*}^{2\sigma}\Vert_{L^{\gamma}_{t}L^{\frac{2\sigma+2}{2\sigma}}_{x}((t_{1},t_{2})\times\mathbb{R}^{n})}\Vert\langle\nabla\rangle\widetilde{u}_{*}\Vert_{L^{q_{2}}_{t}L^{p_{2}}_{x}((t_{1},t_{2})\times\mathbb{R}^{n})}+R(\widetilde{u}_{*},\widetilde{z}_{*})\\
    &\lesssim\Vert(1-t)^{\frac{n\sigma-2}{2\sigma}}\Vert_{L^{\gamma}_{t}((t_{1},t_{2}))}\sup_{t_{1}<t<t_{2}}\Vert\widetilde{u}_{*}\Vert_{L^{2\sigma+2}_{x}}\Vert\widetilde{u}_{*}\Vert_{L^{q_{2}}_{t}W^{1,p_{2}}_{x}((t_{1},t_{2})\times\mathbb{R}^{n})}+R(\widetilde{u}_{*},\widetilde{z}_{*}),
\end{align*}
where $(p_{1},q_{1})$, $(p_{2},q_{2})$, and $(2\sigma+2,\gamma)$ are taken as above. The remaining terms $R(\widetilde{u}_{*},\widetilde{z}_{*})$ can be dealt with in the same way. Therefore, the resulting upper bound tends to $0$ as $t_{1},t_{2}\to 1^-$. In the case $\frac{4}{n} \leq 2\sigma<\frac{4}{n-2}$ if $n\geq 3$, or $\frac{4}{n}\leq 2\sigma<\infty$ if $n=1,2$, implementing the same arguments as in the proof of Theorem \ref{thm:global_Strichartz}, the same can be demonstrated, with the additional small data assumption \eqref{eq:bound_of_data_for_induction} in the $2\sigma = \frac{4}{n}$, $n=1,2$ case (we note that all of the above estimates hold almost surely in probability).

As a result, the above discussion yields that $\{S(-t)\widetilde{u}_{*}(t)\}$ is almost surely a Cauchy sequence in $\Sigma$; indeed, the sequence has an almost sure strong limit in the pseudo-conformal space. Let us denote this strong limit as $S(-1)\widetilde{u}_{+}$:
\begin{equation*}
    \lim_{t\to 1^-}S(-t)\widetilde{u}_{*}(t)=S(-1)\widetilde{u}_{+},\quad a.s.
\end{equation*}
By Proposition 7.5.1 from \cite{Cazenave_book}, we then have
\begin{equation*}
    \lim_{s\to\infty}S(-s)u_{*}(s)=M_{-1}S(-1)\widetilde{u}_{+}\quad\text{a.s. in $\Sigma$}.
\end{equation*}
This proves the almost sure scattering in $\Sigma$ when $\sigma(n)<2\sigma<\frac{4}{n-2}$ if $n\geq 3$ or $\sigma(n)<2\sigma<\infty$ if $n=1,2$, with the additional small data assumption \eqref{eq:bound_of_data_for_induction} in the $2\sigma = \frac{4}{n}$, $n=1,2$ case. 

\section{Scattering in $H^1_x$: Proof of Theorem \ref{thm:result_H^1_scattering}}\label{section:scattering_H^1}

In this section, we aim to establish scattering concerning the mass-critical and inter-critical nonlinearities for $H^1_x$ initial data. We first intend to obtain a uniform-in-time bound for the solution $u_{*}$ to \eqref{eq:additive_random_eqn}:
\begin{equation}\label{eq:uniform_bound_of_rand_soln}
    \sup_{T\leq t<\infty}\Vert u_{*}(t)\Vert_{H^{1}_{x}}<\infty,\quad a.s.\,\,\text{for}\,\,0<T<\infty.
\end{equation}
To this end, recall that $u_{*}(t)=u(t)-z_{*}(t)$ and all the moments of the mass and Hamiltonian of the solution $u$ to \eqref{eq:additive_SNLS} are bounded uniformly-in-time. Using \eqref{eq:almost_sure_time_decay_of_tail}, we have (for $0<T<\infty$)
\begin{align*}
    \sup_{T\leq t<\infty}\Vert u_{*}(t)\Vert_{H^{1}_{x}}&\leq\sup_{T\leq t<\infty}\Vert u(t)\Vert_{H^{1}_{x}}+\sup_{T\leq t<\infty}\Vert z_{*}(t)\Vert_{H^{1}_{x}}\\
    &\lesssim 1+\sup_{T\leq t<\infty}\langle t\rangle^{-\alpha+\frac{1}{2}}<\infty\,\,\,\text{a.s. for}\,\,\alpha\geq\frac{1}{2},
\end{align*}
which implies almost sure boundedness of $\sup_{0\leq t<\infty}\Vert u_{*}(t)\Vert_{H^{1}_{x}}$ as well. Note that the scattering in $H^{1}_{x}$ is to be shown via a different approach than that of $L^{2}_{x}$ or $\Sigma$ (recall that for $L^{2}_{x}$ and $\Sigma$ we exploited the pseudo-conformal transformation and the equivalence of the asymptotics between the original solution and its pseudo-conformal transform in Section \ref{section:scattering_L^2} and \ref{section:scattering_Sigma}), for which the global-in-time Strichartz estimates were sufficient to obtain the desired results. In the rest, we aim to observe the same asymptotic behavior between the solution $u_{*}$ to the random equation \eqref{eq:additive_random_eqn} and the solution $y$ to the deterministic NLS \eqref{eq:deterministic_NLS}, assuming that $u_{*}(T)=y(T),\, a.s.$, for some $T>0$. More precisely, we begin by observing the following:
\begin{proposition}\label{lemma_deterministic_scattering}
    Let $\frac{4}{n}\leq 2\sigma<\frac{4}{n-2}$ if $n\geq 3$, and $\frac{4}{n}\leq 2\sigma<\infty$ if $n=1,2$. Then if $u_{*}(T)=y(T)\in H^{1}_{x}(\mathbb{R}^{n})$ for some $T>0$, there exists a unique global $H^{1}_{x}$ solution $y$, depending on T, to the deterministic NLS \eqref{eq:deterministic_NLS} such that $y$ scatters at infinity in $H^{1}_{x}$ and for any Strichartz pair $(p,q)$, we have
    \begin{equation}\label{eq:global_Strichartz_y}
        \Vert y\Vert_{L^{q}_{t}W^{1,p}_{x}([0,\infty)\times\mathbb{R}^{n})}<\infty.
    \end{equation}
\end{proposition}
\begin{Remark}
    The energy scattering for the mass-critical deterministic NLS  ($2\sigma=\frac{4}{n}$) was solved for any spatial dimensions by utilizing the global space-time estimates established in \cite{Dodson_dimension_3,Dodson_dimension_1,Dodson_dimension_2}. As for the inter-critical nonlinearities, the energy scattering for $n\geq 1$  was established in \cite{Ginibre_Scattering_H1,Nakanishi_scattering_H1}.
   
\end{Remark}
\noindent
We leave the proof of Proposition \ref{lemma_deterministic_scattering} to the appendix. For our purpose, it suffices to show that the difference $u_{*}-y$ belongs to all Strichartz spaces $L^{q}_{t}W^{1,p}_{x}([T,\infty)\times\mathbb{R}^{n})$, $a.s.$, where
\begin{align*}
   u_{*}(t)&=S(t)u_{*}(T)+i\int_{T}^{t}S(t-t')\vert u_{*}+z_{*}\vert^{2\sigma}(u_{*}+z_{*})(t')dt',\\
    y(t)&=S(t)y(T)+i\int_{T}^{t}S(t-t')\vert y\vert^{2\sigma}y(t')dt'.
\end{align*}
Thus, as $u_{*}(T)=y(T)$, we write
\begin{equation*}
    u_{*}(t)-y(t)=i\int_{T}^{t}S(t-t')(F(u_{*}+z_{*})-F(y))dt'
\end{equation*}
and from \eqref{eq:differ_rand_and_det_soln_IVP}
\begin{equation*}
    v(t)=i\int_{T}^{t}S(t-t')(F(v+y+z_{*})-F(y))dt',
\end{equation*}
where $F(u):=\vert u\vert^{2\sigma}u$. The notations and estimates for $F$ in the following are taken from \cite{Visan_I_Method}, so writing $F'(u)=(\partial_{u}F,\partial_{\overline{u}}F)$ and $\nabla u=(\nabla u,\nabla\overline{u})$, we have $\nabla F(u)=\nabla F(u,\overline{u})=\nabla u\cdot F'(u)$ and
\begin{align*}
    \vert F(u)-F(v)\vert&\lesssim (\vert u\vert^{2\sigma}+\vert v\vert^{2\sigma})\vert u-v\vert,\\
    \vert F'(u)\vert &\lesssim\vert u\vert^{2\sigma},\\
    \vert F'(u)-F'(v)\vert &\lesssim\vert u-v\vert^{\min(1,2\sigma)}\left(\vert u\vert+\vert v\vert\right)^{2\sigma-\min(1,2\sigma)}.
    \end{align*}
Therefore, combining these, we arrive at
\begin{align*}
    \vert\nabla\left(F(u)-F(v)\right)\vert &=\vert\nabla u\cdot F'(u)-\nabla v\cdot F'(v)\vert\\
    &=\vert\nabla u\cdot F'(u)-\nabla v\cdot F'(u)+\nabla v\cdot F'(u)-\nabla v\cdot F'(v)\vert\\
    &\leq\vert\nabla(u-v)\vert\vert F'(u)\vert+\vert\nabla v\vert\vert F'(u)-F'(v)\vert\\
    &\lesssim\vert u\vert^{2\sigma}\vert\nabla(u-v)\vert+\vert\nabla v\vert(\vert u\vert+\vert v\vert)^{2\sigma-\min(1,2\sigma)}\vert u-v\vert^{\min(1,2\sigma)}.
\end{align*}
Using the admissible pair $(2+\frac{4}{n},2+\frac{4}{n})$, we obtain
\begin{align*}
    \Vert \vert u\vert^{2\sigma}v\Vert_{L^{\frac{2n+4}{n+4}}_{t,x}}&\leq\Vert u^{2\sigma\theta}\Vert_{L^{\frac{2n+4}{2n\sigma\theta}}_{t,x}}\Vert u^{2\sigma(1-\theta)}\Vert_{L^{p}_{t,x}}\Vert v\Vert_{L^{2+\frac{4}{n}}_{t,x}}\\
    &\leq \Vert u\Vert_{L^{2+\frac{4}{n}}_{t,x}}^{2\sigma\theta}\Vert u\Vert_{L^{2\sigma(1-\theta)p}_{t,x}}^{2\sigma(1-\theta)}\Vert v\Vert_{L^{2+\frac{4}{n}}_{t,x}},
\end{align*}
where
\begin{equation*}
    \frac{1}{p}+\frac{2n\sigma\theta}{2n+4}+\frac{n}{2n+4}=\frac{n+4}{2n+4}\implies \frac{1}{p}=\frac{2-n\sigma\theta}{n+2}.
\end{equation*}
Note that
\begin{align*}
    \frac{2}{2\sigma(1-\theta)p}=\frac{2-n\sigma\theta}{\sigma(1-\theta)(n+2)}=n\left(\frac{1}{2}-\frac{1}{r}\right)\implies \frac{1}{r}=\frac{n\sigma(1-\theta)(n+2)-2(2-n\sigma\theta)}{2n\sigma(1-\theta)(n+2)}
\end{align*}
and the pair $(r,2\sigma(1-\theta)p)$ is admissible. Observe that
\begin{align*}
    \frac{1}{r}-\frac{1}{2\sigma(1-\theta)p}&=\frac{n\sigma(1-\theta)(n+2)-2(2-n\sigma\theta)}{2n\sigma(1-\theta)(n+2)}-\frac{2-n\sigma\theta}{2\sigma(1-\theta)(n+2)}\\
    &=\frac{n\sigma-2}{2n\sigma(1-\theta)}\geq 0,
\end{align*}
provided that $2\sigma\geq \frac{4}{n}$, which holds true in our case. Moreover, by applying a Sobolev embedding $W^{1,r}_{x}\hookrightarrow L^{2\sigma(1-\theta)p}_{x}$ with the relation
\begin{equation}\label{thelastineq1/r}
    \frac{1}{r}\leq\frac{1}{2\sigma(1-\theta)p}+\frac{1}{n},
\end{equation}
we may bound the middle term, $L^{2\sigma(1-\theta)p}_{t,x}$-norm of $u$, by a Strichartz norm. The inequality \eqref{thelastineq1/r} can be rewritten as follows
\begin{align*}
    \frac{1}{r}\leq\frac{1}{2\sigma(1-\theta)p}+\frac{1}{n}\iff\frac{(n-2)\sigma-2+2\sigma\theta}{2n\sigma(1-\theta)}\leq 0\iff 2\sigma\leq\frac{4}{n-2+2\theta}.
\end{align*}
So, we may achieve any inter-critical nonlinearity by a suitable choice of $0<\theta<1$, but when $n=1,2$ we should be more careful:
\begin{itemize}
    \item If $n=2$, then since $\frac{2}{n}=1<\sigma$, we require $$2\sigma\theta-2\leq 0\Rightarrow \theta\leq\frac{1}{\sigma}<1.$$
    \item If $n=1$, since $\sigma>2$, we require $$(2\theta-1)\sigma\leq 2\Rightarrow\theta\leq\frac{1}{2}+\frac{1}{\sigma}<1.$$
\end{itemize}
As a result, applying the Sobolev embedding mentioned above, we obtain the following estimate:
\begin{align*}
    \Vert \vert u\vert^{2\sigma}v\Vert_{L^{\frac{2n+4}{n+4}}_{t,x}}&\leq \Vert u\Vert_{L^{2+\frac{4}{n}}_{t,x}}^{2\sigma\theta}\Vert u\Vert_{L^{2\sigma(1-\theta)p}_{t,x}}^{2\sigma(1-\theta)}\Vert v\Vert_{L^{2+\frac{4}{n}}_{t,x}}\\
    &\lesssim\Vert u\Vert_{L^{2+\frac{4}{n}}_{t,x}}^{2\sigma\theta}\Vert u\Vert_{L^{2\sigma(1-\theta)p}_{t}W^{1,r}_{x}}^{2\sigma(1-\theta)}\Vert v\Vert_{L^{2+\frac{4}{n}}_{t,x}}.
\end{align*}
We now introduce the space $$X^{1}_{n,\sigma,\theta}([T,\infty))=X^{1}([T,\infty)):=\left(L^{2+\frac{4}{n}}_{t}W^{1,2+\frac{4}{n}}_{x}\cap L^{2\sigma(1-\theta)p}_{t}W^{1,r}_{x}\right)([T,\infty)).$$
Let $(\alpha,\beta)$ be either $(2+\frac{4}{n},2+\frac{4}{n})$ or $(r,2\sigma(1-\theta)p)$. There are two cases to consider: $2\sigma \leq 1$ or $2\sigma > 1$.

\noindent
\textbf{Case 1.} Let $2\sigma\leq1$. Then we have
\begin{align*}
    \Vert v\Vert_{L^{\beta}_{t}W^{1,\alpha}_{x}([T,\infty))}&\lesssim\Vert (\vert v+y+z_{*}\vert^{2\sigma}+\vert y\vert^{2\sigma})\vert v+z_{*}\vert\Vert_{L^{\frac{2n+4}{n+4}}_{t,x}([T,\infty))}+\Vert\vert v+y+z_{*}\vert^{2\sigma}\vert\nabla(v+z_{*})\vert\Vert_{L^{\frac{2n+4}{n+4}}_{t,x}([T,\infty))}\\
    &\quad+\Vert \vert\nabla y\vert\vert v+z_{*}\vert^{2\sigma}\Vert_{L^{\frac{2n+4}{n+4}}_{t,x}([T,\infty))}\\
    &\lesssim\left(\Vert v\Vert_{L^{2+\frac{4}{n}}_{t,x}}^{2\sigma\theta}\Vert v\Vert_{L^{2\sigma(1-\theta)p}_{t}W^{1,r}_{x}}^{2\sigma(1-\theta)}+\Vert y\Vert_{L^{2+\frac{4}{n}}_{t,x}}^{2\sigma\theta}\Vert y\Vert_{L^{2\sigma(1-\theta)p}_{t}W^{1,r}_{x}}^{2\sigma(1-\theta)}+\Vert z_{*}\Vert_{L^{2+\frac{4}{n}}_{t,x}}^{2\sigma\theta}\Vert z_{*}\Vert_{L^{2\sigma(1-\theta)p}_{t}W^{1,r}_{x}}^{2\sigma(1-\theta)}\right)\\
    &\quad\quad\times\Vert v+z_{*}\Vert_{L^{2+\frac{4}{n}}_{t}W^{1,2+\frac{4}{n}}_{x}}+\Vert y\Vert_{L^{2+\frac{4}{n}}_{t}W^{1,2+\frac{4}{n}}_{x}}\Vert v+z_{*}\Vert_{L^{2+\frac{4}{n}}_{t,x}}^{2\sigma\theta}\Vert v+z_{*}\Vert_{L^{2\sigma(1-\theta)p}_{t}W^{1,r}_{x}}^{2\sigma(1-\theta)}
\end{align*}
where in the second estimate we apply Hölder's inequality. Therefore, the $X^1$ norm of $v$ can be controlled as follows:
\begin{align}
    \Vert v\Vert_{X^{1}([T,\infty))}&\lesssim\Vert v\Vert_{X^{1}([T,\infty))}^{2\sigma+1}+\Vert v\Vert_{X^{1}([T,\infty))}^{2\sigma} \left( \Vert z_{*}\Vert_{X^{1}([T,\infty))} + \Vert y\Vert_{X^{1}([T,\infty))} \right)\nonumber \\
    &\quad+\Vert v\Vert_{X^{1}([T,\infty))}(\Vert z_{*}\Vert_{X^{1}([T,\infty))}^{2\sigma}+\Vert y\Vert_{X^{1}([T,\infty))}^{2\sigma}) +C(\Vert z_{*}\Vert_{X^{1}([T,\infty))},\Vert y\Vert_{X^{1}([T,\infty))})\nonumber \\
    &\lesssim\Vert v\Vert_{X^{1}([T,\infty))}^{2\sigma+1}+ C(\Vert z_{*}\Vert_{X^{1}([T,\infty))},\Vert y\Vert_{X^{1}([T,\infty))})(\Vert v\Vert_{X^{1}([T,\infty))}^{2\sigma}+\Vert v\Vert_{X^{1}([T,\infty))})\nonumber \\
    &\quad+ C(\Vert z_{*}\Vert_{X^{1}([T,\infty))},\Vert y\Vert_{X^{1}([T,\infty))})\nonumber \\
    \label{eq:base_for_induction_for_v}&\lesssim C(\Vert z_{*}\Vert_{X^{1}([T,\infty))},\Vert y\Vert_{X^{1}([T,\infty))})+\left(1+C(\Vert z_{*}\Vert_{X^{1}([T,\infty))},\Vert y\Vert_{X^{1}([T,\infty))})\right)\Vert v\Vert_{X^{1}([T,\infty))}^{2\sigma+1}.
\end{align}
 Note that replacing $\Vert v\Vert_{X^{1}([T,\infty))}$ with $\Vert v\Vert_{C_{t}H^{1}_{x}}$ above, the estimate \eqref{eq:base_for_induction_for_v} still remains valid. Note also that in arriving at \eqref{eq:base_for_induction_for_v}, we have used the fact of boundedness of $X^{1}$-norm of $v$ with lower exponents by the constant depending on $y$ and $z_{*}$ whenever $\Vert v\Vert_{X^{1}([T,\infty))}\ll 1$. Otherwise, the lower order terms involving $v$ need to be controlled by $\left(1+C(\Vert z_{*}\Vert_{X^{1}([T,\infty))},\Vert y\Vert_{X^{1}([T,\infty))})\right)\Vert v\Vert_{X^{1}([T,\infty))}^{2\sigma+1}$. Recall that $y$ and $z_{*}$ obey the global-in-time Strichartz estimates via \eqref{eq:global_Strichartz_y} and \eqref{eq:global_Strichartz_est_for_tail} implying that
\begin{equation*}
    \Vert y\Vert_{X^{1}([T,\infty))},\Vert z_{*}\Vert_{X^{1}([T,\infty))}<\infty
\end{equation*}
uniformly in time. For sufficiently large $T$ we can find $\varepsilon(T)>0$ so that
\begin{equation*}
     \Vert y\Vert_{X^{1}([T,\infty))},\Vert z_{*}\Vert_{X^{1}([T,\infty))}\lesssim\varepsilon(T)
\end{equation*}
where $\varepsilon(T)\to 0$ as $T\to\infty$. Note that we have
\begin{equation*}
    \Vert v\Vert_{C_{t}H^{1}_{x}([T,\infty))}+\Vert v\Vert_{X^{1}([T,\infty))}\leq C(\varepsilon(T))+(C+C(\varepsilon(T)))\Vert v\Vert_{X^{1}([T,\infty))}^{2\sigma+1}
\end{equation*}
for some constant $C>0$. Applying Lemma \ref{lemma:uniform_bounded_by_its_higher_powers} with $a=C(\varepsilon(T))$, $b=(C+C(\varepsilon(T)))$, $\alpha=2\sigma+1$ and the condition
\begin{equation*}
    C(\varepsilon(T))<\left(\frac{2\sigma}{2\sigma+1}\right)\left((2\sigma+1)(C+C(\varepsilon(T)))\right)^{-\frac{1}{2\sigma}},
\end{equation*}
 we arrive at
\begin{equation*}
     \Vert v\Vert_{C_{t}H^{1}_{x}([T,\infty))}+\Vert v\Vert_{X^{1}([T,\infty))}\lesssim \frac{2\sigma+1}{2\sigma}C(\varepsilon(T)).
\end{equation*}
The same procedure also shows that $\Vert v\Vert_{S^{1}([T,\infty))}<\infty$ a.s. (where $S^{1}$ is defined as in \eqref{S^1space}).

\noindent
\textbf{Case 2.} Let $2\sigma>1$. This case can be handled with the same arguments as in the previous case yet with minor differences appeared in selection of admissible pairs while bounding the difference of Duhamel terms. It remains to show that $\{S(-t)u_{*}(t):t\in[T,\infty)\}$ is Cauchy in $H^{1}_{x}$. Let $T<t_{1}<t_{2}<\infty$. Then
\begin{align*}
    \Vert S(-t_{2})u_{*}(t_{2})-S(-t_{1})u_{*}(t_{1})\Vert_{H^{1}_{x}}&\leq\Vert S(-t_{2})(u_{*}(t_{2})-y(t_{2}))\Vert_{H^{1}_{x}}+\Vert S(-t_{2})y(t_{2})-S(-t_{1})y(t_{1})\Vert_{H^{1}_{x}}\\
    &+\Vert S(-t_{1})(u_{*}(t_{1})-y(t_{1}))\Vert_{H^{1}_{x}}\\
    &\lesssim\sup_{t\in[T,\infty)}\Vert v(t)\Vert_{H^{1}_{x}}+ \Vert S(-t_{2})y(t_{2})-S(-t_{1})y(t_{1})\Vert_{H^{1}_{x}}\\
    &\lesssim C(\varepsilon(T))+\Vert S(-t_{2})y(t_{2})-S(-t_{1})y(t_{1})\Vert_{H^{1}_{x}}\to 0\quad \text{as $T\to\infty$,}
\end{align*}
as $y$ scatters in $H^{1}_{x}$ and $v(t)\to 0$ in $H^{1}_{x}$ as $t\to\infty$. This shows that $u_{*}(t)$ almost surely converges to the same strong limit of $y(t)$ in $H^{1}_{x}$ as $t\to\infty$, i.e.,
\begin{equation*}
    \lim_{t\to\infty}\Vert S(-t)u_{*}(t)-y_{+}\Vert_{H^{1}_{x}}=0,\quad a.s.
\end{equation*}

\section{Appendix}

\begin{proof}[Proof of Proposition \ref{lemma_deterministic_scattering}]
    We will divide the proof into two: the mass critical and the mass supercritical cases. Throughout the proof, the spacetime norms $L^{q}_{t}L^{p}_{x}$ are taken on the set $[0,\infty)\times\mathbb{R}^{n}$ unless otherwise specified. We will only provide the proof for $n\geq 3$ to avoid any technicalities stemming from deterministic NLS problems. When $n=1,2$, see \cite{Dodson_dimension_1, Dodson_dimension_2, Nakanishi_scattering_H1} for the energy scattering.

    \noindent
    \textbf{Case 1.} Let $\frac{4}{n}<2\sigma<\frac{4}{n-2}$, $T>0$ be fixed, and $u_{*}(T)=y(T)$. Using the pre-established results in the deterministic NLS problems, by taking the initial data as $y(T)$, the global well-posedness can be obtained for the time interval $[T,\infty)$. Moreover, using continuity-in-time and time reversal, the solution can be shown to exist uniquely for $t\in[0,\infty)$. In the forthcoming discussion, the conservation of the mass and Hamiltonian will imply that
    \begin{equation}\label{eq:global_Str_det_NLS}
        \Vert y\Vert_{L^{q}_{t}W^{1,p}_{x}}\leq C(M(y(0)),H(y(0)))= C(M(u_*(T)),H(u_*(T)))\leq C(\Vert u_{*}(T)\Vert_{H^{1}_{x}}),
    \end{equation}
    provided that the first inequality holds. As $H^{1}_{x}$-norm of $u_{*}(t)$ remains bounded uniformly in time, the upper bound above for the Strichartz norm of $y(t)$ will be independent of $T$. Our intention is to focus on the first inequality in \eqref{eq:global_Str_det_NLS} so as to conclude the $H^{1}_{x}$ scattering of the deterministic NLS. Note that we adopt the arguments in \cite{Tao_Visan_Sigma} in relation to the upcoming space-time estimates. In this connection, we commence with the following inequality derived from the Morawetz estimates, see \cite{Tao_Visan_Sigma}:
    
    \noindent
    Let $I$ be a compact time interval and $y$ be a solution to \eqref{eq:deterministic_NLS} on $I\times\mathbb{R}^{n}$ with $n\geq 3$. Then
\begin{equation}\label{eq:Morawetz_control}
        \Vert y\Vert_{L^{n+1}_{t}L^{\frac{2(n+1)}{n-1}}_{x}(I\times\mathbb{R}^{n})}\lesssim\Vert y\Vert_{L^{\infty}_{t}H^{1}_{x}(I\times\mathbb{R}^{n})}.
    \end{equation}
    As the solution $y$ is global, using this fact and the bound \eqref{eq:Morawetz_control} together with conservation of the mass and Hamiltonian, we obtain
    \begin{equation*}
        \Vert y\Vert_{L^{n+1}_{t}L^{\frac{2(n+1)}{n-1}}_{x}([0,\infty)\times\mathbb{R}^{n})}\leq C(M(y(0)),H(u(0)))\leq C(\Vert u_{*}(T)\Vert_{H^{1}_{x}}).
    \end{equation*}
    Let $\varepsilon>0$ be a small constant to be determined so that, decomposing $[0,\infty)$ into $K=K(\Vert u_{*}(T)\Vert_{H^{1}_{x}},\varepsilon)$  subintervals $I_{j}=[t_{j},t_{j+1}]$, the following inequality holds true
    \begin{equation*}
        \Vert y\Vert_{L^{n+1}_{t}L^{\frac{2(n+1)}{n-1}}_{x}(I_{j}\times\mathbb{R}^{n})}\lesssim\varepsilon.
    \end{equation*}
   Next define the useful norm: \begin{align}\label{S^1space}
     \Vert y\Vert_{S^{1}(I\times\mathbb{R}^{n})}:=\sup_{(p,q)-admissible}\Vert y\Vert_{L^{q}_{t}W^{1,p}_{x}(I\times\mathbb{R}^{n})}.  
   \end{align}
    Therefore, using the Strichartz estimates and the above series of inequalities we obtain
    \begin{align*}
        \Vert y\Vert_{S^{1}(I_{j}\times\mathbb{R}^{n})}&=\left\Vert S(t)y(t_{j})+i\int_{t_{j}}^{t}S(t-s)\vert y\vert^{2\sigma}y(s)ds\right\Vert_{S^{1}(I_{j}\times\mathbb{R}^{n})}\\
        &\lesssim\Vert y(t_{j})\Vert_{H^{1}_{x}}+\Vert \vert y\vert^{2\sigma}y\Vert_{L^{2}_{t}W^{1,\frac{2n}{n+2}}_{x}(I_{j}\times\mathbb{R}^{n})}\\
        &\lesssim\Vert y(t_{j})\Vert_{H^{1}_{x}}+\Vert y\Vert_{L_{t}^{2+\frac{1}{\theta}}W^{1,\frac{2n(2\theta+1)}{n(2\theta+1)-4\theta}}_{x}}\Vert y\Vert_{L^{n+1}_{t}L^{\frac{2(n+1)}{n-1}}_{x}}^{\frac{n+1}{2(2\theta+1)}}\Vert y\Vert_{L^{\infty}_{t}L^{2}_{x}}^{\alpha(\theta)}\Vert y\Vert_{L^{\infty}_{t}L^{\frac{2n}{n-2}}_{x}}^{\beta(\theta)}\\
        &\lesssim\Vert y(t_{j})\Vert_{H^{1}_{x}}+\varepsilon^{\frac{n+1}{2(2\theta+1)}}\Vert y\Vert_{L^{\infty}_{t}H^{1}_{x}}^{\alpha(\theta)+\beta(\theta)}\Vert y\Vert_{S^{1}(I_{j}\times\mathbb{R}^{n})},
    \end{align*}
    for some sufficiently large $\theta>0$ with
    \begin{equation*}
        \alpha(\theta)=2\sigma\left(1-\frac{n}{2}\right)+\frac{8\theta+1}{2(2\theta+1)}\quad\text{and}\quad \beta(\theta)=\frac{n}{2}\left(2\sigma-\frac{n+8\theta+2}{n(2\theta+1)}\right).
    \end{equation*}
    After taking $\varepsilon>0$ sufficiently small depending on $\Vert u_{*}(T)\Vert_{H^{1}_{x}}$, by the conservation of the mass and the Hamiltonian, we have
    \begin{equation*}
        \Vert y\Vert_{S^{1}(I_{j}\times\mathbb{R}^{n})}\lesssim\Vert y(t_{j})\Vert_{H^{1}_{x}}\leq C(M(y(0)),H(y(0)))\leq C(\Vert u_{*}(T)\Vert_{H^{1}_{x}}).
    \end{equation*}
    Eventually, summing over $I_{j}$'s, we arrive at
    \begin{equation*}
        \Vert y\Vert_{S^{1}([0,\infty)\times\mathbb{R}^{n})}\leq KC(\Vert u_{*}(T)\Vert_{H^{1}_{x}})<\infty.
    \end{equation*}
    This shows that $y$ indeed satisfies the global and uniformly in-time Strichartz estimates. Now consider the sequence $\{S(-t)y(t):t\in[0,\infty)\}$. We aim to show that it is Cauchy in $H^{1}_{x}$. Thus, let $T<t_{1}<t_{2}<\infty$ and consider
    \begin{equation*}
        S(-t_{2})y(t_{2})-S(-t_{1})y(t_{1})=i\int_{t_{1}}^{t_{2}}S(-t')\vert y\vert^{2\sigma}y(t')dt'.
    \end{equation*}
    Making use of \eqref{eq:Morawetz_control} and Strichartz estimates, we arrive at
    \begin{align*}
        \Vert S(-t_{2})y(t_{2})-S(-t_{1})y(t_{1})\Vert_{H^{1}_{x}}&=\left\Vert\int_{t_{1}}^{t_{2}}S(-t')\langle\nabla\rangle\vert y\vert^{2\sigma}y(t')dt'\right\Vert_{L^{2}_{x}}\\
        &\lesssim \Vert \vert y\vert^{2\sigma}y\Vert_{L^{2}_{t}W^{1,\frac{2n}{n+2}}_{x}((t_{1},t_{2})\times\mathbb{R}^{n})}\\
        &\lesssim \Vert y\Vert_{L_{t}^{2+\frac{1}{\theta}}W^{1,\frac{2n(2\theta+1)}{n(2\theta+1)-4\theta}}_{x}}\Vert y\Vert_{L^{n+1}_{t}L^{\frac{2(n+1)}{n-1}}_{x}}^{\frac{n+1}{2(2\theta+1)}}\Vert y\Vert_{L^{\infty}_{t}L^{2}_{x}}^{\alpha(\theta)}\Vert y\Vert_{L^{\infty}_{t}L^{\frac{2n}{n-2}}_{x}}^{\beta(\theta)}\\
        &\lesssim \Vert y\Vert_{L_{t}^{2+\frac{1}{\theta}}W^{1,\frac{2n(2\theta+1)}{n(2\theta+1)-4\theta}}_{x}}\Vert y\Vert_{S^{1}((t_{1},t_{2}))}^{\alpha(\theta)+\beta(\theta)+\frac{n+1}{2(2\theta+1)}}\\
        &=\Vert y\Vert_{L_{t}^{2+\frac{1}{\theta}}W^{1,\frac{2n(2\theta+1)}{n(2\theta+1)-4\theta}}_{x}}\Vert y\Vert_{S^{1}((t_{1},t_{2}))}^{2\sigma},
    \end{align*}
    where all the spacetime norms are taken on the set $(t_{1},t_{2})\times\mathbb{R}^{n}$. Notice that
    \begin{equation*}
        \alpha(\theta)+\beta(\theta)+\frac{n+1}{2(2\theta+1)}=2\sigma,
    \end{equation*}
    and $\left(\frac{2n(2\theta+1)}{n(2\theta+1)-4\theta},2+\frac{1}{\theta}\right)$ is a Strichartz admissible pair. As $S^{1}([0,\infty))$ norm of $y$ is finite, we conclude that
    \begin{equation*}
         \Vert S(-t_{2})y(t_{2})-S(-t_{1})y(t_{1})\Vert_{H^{1}_{x}}\to 0\quad\text{as $t_{1},t_{2}\to\infty$}
    \end{equation*}
which verifies the existence of a unique $y_{+}\in H^{1}_{x}(\mathbb{R}^{n})$ such that $S(-t)y(t)\to y_{+}$ in $H^{1}_{x}(\mathbb{R}^{n})$ as $t\to\infty$.

    \noindent
    \textbf{Case 2.} Let $2\sigma=\frac{4}{n}$, $T>0$ be fixed, and $u_{*}(T)=y(T)\in H^{1}_{x}\subset L^{2}_{x}$. Then by \cite{Dodson_dimension_3}, the deterministic NLS \eqref{eq:deterministic_NLS} has a unique global solution in $L^{2}_{x}$. The conservation of mass and Strichartz estimates then imply that
    \begin{equation*}
        \Vert y\Vert_{L^{2+\frac{4}{n}}_{t,x}([0,\infty)\times\mathbb{R}^{n})}\leq C(M(y(T)))=C(M(u_{*}(T)))<\infty,\quad a.s.,
    \end{equation*}
    and this bound does not depend on $T$ due to \eqref{eq:uniform_bound_of_rand_soln}. As before, let $\varepsilon>0$ be a constant so that dividing $[0,\infty)$ into $K=K(\Vert u_{*}(T)\Vert_{L^{2}_{x}},\varepsilon)$ subintervals $I_{j}=[t_{j},t_{j+1}]$ yields
    \begin{equation*}
        \Vert y\Vert_{L^{2+\frac{4}{n}}_{t,x}(I_{j}\times\mathbb{R}^{n})}\lesssim\varepsilon.
    \end{equation*}
    Note that since the data is taken from $H^{1}_{x}$ and we are in the energy sub-critical regime, the conservation of mass and Hamiltonian implies that
    \begin{equation*}
        \Vert y(t)\Vert_{H^{1}_{x}}^{2}\lesssim M(y(t))+H(y(t))=M(u_{*}(T))+H(u_{*}(T))\leq C(\Vert u_{*}(t)\Vert_{C_{t}H^{1}_{x}([0,\infty)\times\mathbb{R}^{n})})<\infty\quad a.s.
    \end{equation*}
  which leads to
    \begin{align*}
        \Vert y\Vert_{L^{2+\frac{4}{n}}_{t}W^{1,2+\frac{4}{n}}_{x}(I_{j}\times\mathbb{R}^{n})}&\lesssim\Vert y(t_{j})\Vert_{H^{1}_{x}}+\left\Vert\int_{t_{j}}^{t}S(t-t')\langle\nabla\rangle\vert y\vert^{\frac{4}{n}}y(t')dt'\right\Vert_{L^{2+\frac{4}{n}}_{t,x}}\\
        &\lesssim 1+\Vert\vert y\vert^{\frac{4}{n}}y\Vert_{L^{\frac{2n+4}{n+4}}_{t}W^{1,\frac{2n+4}{n+4}}_{x}(I_{j})}\\
        &\lesssim 1+\Vert y\Vert_{L^{2+\frac{4}{n}}_{t}W^{1,2+\frac{4}{n}}_{x}(I_{j})}\Vert y\Vert_{L^{2+\frac{4}{n}}_{t,x}(I_{j})}^{\frac{4}{n}}\\
        &\lesssim 1+\varepsilon^{\frac{4}{n}}\Vert y\Vert_{L^{2+\frac{4}{n}}_{t}W^{1,2+\frac{4}{n}}_{x}(I_{j})}.
    \end{align*}
    Next taking $\varepsilon>0$ sufficiently small, depending on $\Vert u_{*}(T)\Vert_{L^{2}_{x}}$ as in the inter-critical case, the above norm of $y$ can be controlled by $H^{1}_{x}$ norm of $y(t_{j})$, in addition, we know that this Sobolev norm of $y$ is uniformly bounded in time. Thus, on $[0,\infty)$, we have
    \begin{equation*}
        \Vert y\Vert_{L^{2+\frac{4}{n}}_{t}W^{1,2+\frac{4}{n}}_{x}([0,\infty)\times\mathbb{R}^{n})}\lesssim K(\Vert u_{*}(T)\Vert_{L^{2}_{x}}),
    \end{equation*}
    which is finite without any dependence on $T$. Moreover, similar arguments show that $y\in S^{1}([0,\infty))$. Finally  as before we let $T<t_{1}<t_{2}<\infty$ and see that
    \begin{align*}
        \Vert S(-t_{2})y(t_{2})-S(-t_{1})y(t_{1})\Vert_{H^{1}_{x}}&=\left\Vert\int_{t_{1}}^{t}S(-t')\langle\nabla\rangle\vert y\vert^{2\sigma}y(t')dt'\right\Vert_{L^{2}_{x}}\\
        &\lesssim\Vert\vert y\vert^{2\sigma}y\Vert_{L^{\frac{2n+4}{n+4}}_{t}W^{1,\frac{2n+4}{n+4}}_{x}((t_{1},t_{2})\times\mathbb{R}^{n})}\\
        &\lesssim\Vert y\Vert_{L^{2+\frac{4}{n}}_{t,x}((t_{1},t_{2})\times\mathbb{R}^{n})}^{\frac{4}{n}}\Vert y\Vert_{L^{2+\frac{4}{n}}_{t}W^{1,2+\frac{4}{n}}_{x}((t_{1},t_{2})\times\mathbb{R}^{n})}\to 0\quad \text{as $t_{1},t_{2}\to\infty$.}
    \end{align*}
\end{proof}

\section*{Acknowledgement}
The third author would like to thank his M.Sc advisor \"{U}mit I\c{s}lak and the fourth author would like to thank his Ph.D advisor T. Burak G\"{u}rel for many helpful suggestions and comments during the preparation of this manuscript.

\nocite{*}
\bibliographystyle{abbrv}
\bibliography{reference_SNLS.bib}

\end{document}